\documentclass[11pt]{article}
\usepackage[margin=1in]{geometry}
\usepackage{appendix}
\usepackage[export]{adjustbox}
\usepackage{amsthm}
\usepackage{amsmath}
\usepackage{amssymb}
\usepackage{upgreek}
\usepackage{mathtools}
\usepackage{enumerate}
\usepackage{natbib}
\usepackage[colorlinks,citecolor=blue,urlcolor=blue,filecolor=blue,backref=page]{hyperref}
\usepackage{graphicx}
\usepackage{subfig}
\usepackage{dsfont}
\usepackage{snaptodo}
\usepackage{hyperref}
\usepackage{diagbox}
\usepackage{float}
\usepackage{placeins}
\usepackage{authblk}

\usepackage{tikz}
\usetikzlibrary{positioning}
\tikzset{vtx/.style={circle,draw,minimum size=18pt,inner sep=0pt}}

\numberwithin{equation}{section}
\theoremstyle{plain}

\theoremstyle{definition}

\theoremstyle{definition}
\newtheorem{definition}{Definition}

\theoremstyle{assumption}
\newtheorem{assumption}{Assumption}

\theoremstyle{proposition}
\newtheorem{proposition}{Proposition}[section]

\theoremstyle{proposition}
\newtheorem{corollary}{Corollary}[section]

\theoremstyle{example}
\newtheorem{example}{Example}

\theoremstyle{lemma}
\newtheorem{lemma}{Lemma}[section]

\theoremstyle{theorem}
\newtheorem{theorem}{Theorem}[section]

\theoremstyle{remark}
\newtheorem*{remark}{Remark}

\hypersetup{
    colorlinks=true,
    linkcolor=blue,
    filecolor=magenta,      
    urlcolor=cyan,
    pdftitle={Overleaf Example},
    pdfpagemode=FullScreen,
    }

\title{Spectral analysis of multivariate stationary Hawkes processes}
\author[1]{Yifu Tang}
\author[1]{Conor Kresin}
\author[1]{Boris Baeumer}
\author[1]{Ting Wang}
\affil[1]{Department of Mathematics and Statistics, University of Otago, New Zealand}

\begin{document}
\maketitle

\begin{abstract}
    We establish the asymptotic validity of frequency-domain inference for stationary multivariate Hawkes processes under mild conditions, bridging the gap between theory and application. By developing upper-bounds on the reduced cumulant measures from the cluster representation of the Hawkes processes, we prove a functional central limit theorem and, as a consequence, consistency of the Whittle estimator under stationarity alone (i.e., the spectral radius of the interactions matrix $\rho(\boldsymbol\nu)<1$), applicable to Hawkes processes with heavy-tailed mutual-excitation kernels. Under mild extra moment conditions, we further obtain asymptotic normality with an explicit limiting covariance in terms of second- and fourth-order cumulant spectral densities. We also propose a simple frequency-domain method to detect joint independence of subprocesses of a multivariate Hawkes process. The performance of the Whittle estimator and the test of independence are demonstrated via simulation studies.

\end{abstract}

\section{Introduction}
Hawkes processes \citep{Hawkes1971} are a prominent class of self- and mutually-exciting point-process models with broad applications (e.g.\ seismology, finance, and neuroscience; see \citet{reinhart2018review} for a comprehensive review). Despite their popularity, likelihood-based inference \citep{ogata1978} can be computationally burdensome at large event counts and in higher dimensions: even with algorithmic improvements, likelihood evaluation and optimization can remain a bottleneck for large-scale multivariate data \citep{kresin2023parametric}. This has motivated a wide range of alternative estimators, including approximate-likelihood methods \citep{Schoenberg2013}, as well as Bayesian approaches relying on numerical approximation \citep{rasmussen2013bayesian}.

Frequency-domain methods offer a powerful alternative to likelihood-based inference for dependent data. A central tool is the Whittle likelihood \citep{whittle1953analysis,whittle1954stationary}, which replaces the exact likelihood by a Gaussian likelihood of the Fourier transform of the observations. This approximation yields computationally inexpensive estimators with strong statistical properties in many settings \citep{Dahlhaus1988,Brillinger2001}. For temporal point processes, the natural quantity of interest in the frequency domain is the Bartlett spectral density (and, in the multivariate case, the Bartlett spectral density matrix), defined as the Fourier transform of the second-order reduced cumulant measures \citep{Bartlett1963,Brillinger1972,Daley_and_Vere-Jones2003}.

A multivariate stationary Hawkes process is uniquely defined by its first- and second-order statistics \citep[Corollary 1]{Bacry_and_Muzy2016}. This makes Whittle likelihood particularly useful for inference of Hawkes processes because it concerns only the first- and second-order information. Indeed, the Whittle likelihood leverages the Bartlett spectral density matrix, which for stationary multivariate linear Hawkes processes admits an explicit analytic form \citep[Example 8.3(c)]{Daley_and_Vere-Jones2003}. To our knowledge, the first ever use of Whittle likelihood in fitting Hawkes processes is \citet{Adamopoulos1976} who studied earthquake data. However, rigorous asymptotic theory for Whittle estimation of Hawkes processes has developed more slowly than its time-series counterpart. \citet{Tuan1981} first proved some asymptotic properties of the Whittle estimator of univariate Hawkes processes under strong assumptions on the self-excitation kernel. \citet{Cheysson2022_and_Lang2022} studied the binned observations of a univariate Hawkes process, which is equivalent to a realization of an integer-valued time series, and managed to establish the $\alpha$-mixing of the time series by assuming the existence of certain moments of the self-excitation kernel. Then the asymptotic properties of the Whittle estimator can then be shown by employing known results in time series analysis (see e.g. \citet{Dzhaparidze1986}). \citet{Yang_and_Guan2026} derived asymptotic properties related to Whittle likelihood for spatial point processes (with univariate Hawkes processes as special examples) by imposing conditions such as $\alpha$-mixing of the point process, and integrability and smoothness of cumulant spectral density functions. However, the assumptions made in these papers exclude Hawkes processes with heavy-tailed mutual excitation kernels (e.g. kernels which do not admit first moment) representing long-lasting effects. This kind of kernels have proven useful in fields such as seismology and finance, and continue to gain popularity. (\citealt{Vere-Jones_and_Davies1966, Hawkes2018, Chen_et_al2021, Davis_et_al2024, Dupret_and_Hainaut2025, Gupta_and_Maheshwari2026}).

In this paper, we develop frequency-domain asymptotic theory for stationary multivariate linear Hawkes processes observed in continuous time: we observe the marked event sequence (event times and types) on $[0,T]$ and let $T\to\infty$. We consider Whittle estimation based on a parametric family of Bartlett spectral density matrices $\{\Psi_\theta(\omega):\theta\in\Theta\}$ induced by a parametric Hawkes specification (e.g.\ $\theta$ collects background rates, branching coefficients, and any kernel-shape parameters), and we also develop a simple frequency-domain test of independence of subprocesses of Hawkes process. Our analysis builds on the cumulant framework initiated by \citet{Brillinger1972} and developed for Hawkes processes via explicit tree representations of cumulant measures derived in \citep{Jovanovic_et_al2015}. The main contributions are:

\begin{enumerate}
    \item We study a spectral empirical process, defined as the sum of quadratic form of finite Fourier transform over Fourier frequencies, and establish a functional Central Limit Theorem (CLT) under slight conditions on mutual-excitation kernels. The proof is based on results pertaining to reduced cumulant measures without invoking $\alpha$-mixing.
    \item We define the Whittle likelihood and estimator with the help of the spectral empirical process. The consistency is then proven under the sole requirement of stationarity. In our normalized-kernel parameterization, this essentially means that the only assumption needed to establish consistency is $\rho(\boldsymbol{\nu})<1$, where  $\boldsymbol{\nu}$ is the interactions matrix and $\rho(\cdot)$ denotes the spectral radius.
    \item Under additional, and yet still mild moment conditions on the mutual-excitation kernels, we derive asymptotic normality of the Whittle estimator with an explicit limiting covariance expressed in terms of second- and fourth-order cumulant spectral densities thanks to the functional CLT of the spectral empirical process.
    \item We construct a simple test of independence of subprocesses of a multivariate stationary Hawkes process based on an estimation procedure for the spectral density matrix at frequency $0$ inspired by \citet{McElroy_and_Politis2025}.
\end{enumerate}

The remainder of the paper proceeds from probabilistic structure to inferential consequences. Section~\ref{section:hawkes} introduces multivariate Hawkes processes, cumulant measures, and the finite Fourier transform. Section~\ref{sect_spectral} then develops the frequency-domain machinery: Section~\ref{spec_ep_sect} establishes weak convergence of the spectral empirical process, Section~\ref{Whttiel_sect} develops the Whittle likelihood and proves consistency and asymptotic normality, and Section~\ref{section_ind_test} proposes a simple frequency-domain test of independence among subprocesses. Numerical experiments are presented in Section~\ref{section_numerical_experiment}, and Section~\ref{section_conclusion} concludes. Tables and figures, technical proofs and auxiliary results on cumulant measures, Fourier-transform cumulants, and deterministic approximation bounds are collected in the Appendices.

\section{Multivariate Hawkes processes and preliminary notation}\label{section:hawkes}
This section first defines stationary multivariate Hawkes processes \citep{Hawkes1971} and recalls their cluster representation \citep{hawkes1974cluster}. We then introduce cumulant and reduced cumulant measures, whose finiteness is the key structural input behind our frequency-domain asymptotics. Finally, we define the cumulant spectral densities and the finite Fourier transform, which are the main ingredients in Sections~\ref{spec_ep_sect}--\ref{section_ind_test}.

A $D$-variate temporal point process $N=(N_1,\ldots,N_D)$ on $\mathbb{R}$ can be viewed as a collection of random counting measures, where $N_j(B)$ counts the number of type-$j$ events falling in a Borel set $B\subset\mathbb{R}$. $N_1,\ldots,N_D$ are often called subprocesses of $N$. We say that $N$ is (strictly) stationary if for every $s\in\mathbb{N}$, every choice of indices $(j_1,\ldots,j_s)$, and every collection of Borel sets $(B_1,\ldots,B_s)$, the joint law of $(N_{j_1}(B_1),\ldots,N_{j_s}(B_s))$ is invariant under common time shifts \citep[Section 3.2]{Daley_and_Vere-Jones2003}, i.e.
\[
(N_{j_1}(B_1),\ldots,N_{j_s}(B_s)) \stackrel{d}{=} (N_{j_1}(B_1+t),\ldots,N_{j_s}(B_s+t)), \qquad t\in\mathbb{R}.
\]

We will now introduce the point process model of interest to this paper.
\begin{definition}[Multivariate (linear) Hawkes process]\label{HP_def} 
Let $N=(N_1,\ldots,N_D)$ be a $D$-variate simple point process on $\mathbb R$ with natural filtration $\mathcal H_{t-}=\sigma\{N_j((-\infty,s]): s<t,\ j=1,\ldots,D\}$. Denote by $(t^{(j)}_k)_{k\ge1}$ the event times of $N_j$. We say that $N$ is a (linear) Hawkes process with background rate
$\boldsymbol\mu=(\mu_1,\ldots,\mu_D)^\top$, interactions matrix $\boldsymbol\nu=\{\nu_{ij}\}_{i,j=1,\cdots,D}$, and mutual-excitation kernels
$\boldsymbol G=\{g_{ij}\}_{i,j=1}^D$ if its conditional intensities satisfy, for $i=1,\ldots,D$,
\[
\lambda_i(t\mid\mathcal H_{t-})
= \mu_i + \sum_{j=1}^D \nu_{ij}\sum_{k: t^{(j)}_k<t} g_{ij}(t-t^{(j)}_k).
\]
Here $\mu_i>0$, $\nu_{ij}\ge0$, and each $g_{ij}:\mathbb{R}\to[0,\infty)$ satisfies $\int_0^\infty g_{ij}(u)\,du=1$ and $g(x)=0$ for $x<0$. The conditional intensity is understood in the usual sense: for each $i$,
\[
\mathrm{E}[N_i((t,t+\Delta])\mid \mathcal H_{t-}]
= \int_t^{t+\Delta}\lambda_i(s\mid \mathcal H_{s-})\,ds
= \lambda_i(t\mid \mathcal H_{t-})\Delta+o(\Delta),\qquad \Delta\downarrow0,
\]
where $\mathrm{E}$ is the expectation operator. If the spectral radius of $\boldsymbol{\nu}$, denoted by $\rho(\boldsymbol\nu)$, is smaller than 1, then a stationary version exists (and is unique in law); see \citet[Example 8.3(c)]{Daley_and_Vere-Jones2003} or \cite{Hawkes1971}.
\end{definition}

The stationary temporal Hawkes processes admit a Poisson cluster representation~\cite[Example 8.3(c)]{Daley_and_Vere-Jones2003}, which mimics many real-world phenomena, such as the triggering-triggered effect in epidemics and seismology \citep{reinhart2018review}. To facilitate exposition, we view multivariate processes as point processes on $\mathbb{R}\times\{1,\cdots,D\}$ endowed with base measure $l_{\mathrm{Leb}}\times\#$, where $l_{\mathrm{Leb}}$ is the Lebesgue measure on $\mathbb{R}$ and $\#$ is the counting measure on $\{1,\cdots,D\}$. 

Concretely, a multivariate stationary Hawkes process is a collection of immigrants and offspring. The immigrants form a Poisson process with intensity $r_b((t,i))=\mu_i$. Each immigrant $(t_\ast,j)$ then independently produces offspring according to an inhomogeneous Poisson process with intensity
\[
r_d((t,i)\mid (t_\ast,j))=\nu_{ij}g_{ij}(t-t_\ast).
\]
Offspring reproduce in the same manner, generating i.i.d.\ clusters attached to each immigrant. Let $\rho(\cdot)$ denote the spectral radius. The interactions matrix $\boldsymbol{\nu}=\{\nu_{ij}\}_{i,j=1,\cdots,D}$ is therefore the mean offspring (branching) matrix, and the condition $\rho(\boldsymbol{\nu})<1$ ensures subcriticality, finiteness of clusters almost surely, and existence of a stationary version.

In line with \citet{Brillinger1972} and \citet{Tuan1981}, we will rely on the (reduced) cumulant measure of Hawkes processes and their Fourier transforms in the proofs. Recall (cf. equation (2.3.1) of \citet{Brillinger2001}) that the joint cumulant of (complex-valued) random variables $X_1,\cdots,X_s$ is defined as
\begin{align}\label{def_cum_rv}
\mathrm{cum}(X_1,\cdots,X_s)
=\sum_{q=1}^s (-1)^{q-1}(q-1)!\sum_{\{\Xi_{1},\cdots,\Xi_{q}\}}\left(\mathrm{E}\prod_{r\in \Xi_{1}}X_{r}\right)\cdots\left(\mathrm{E}\prod_{r\in \Xi_{q}}X_{r}\right),    
\end{align}
where the summation $\sum_{\{\Xi_{1},\cdots,\Xi_{q}\}}$ is over all partitions of $\{1,\cdots,s\}$ with $q$ blocks. This concept can be extended to point processes in the following way.
\begin{definition}[Cumulant measure and reduced cumulant measure]\label{HPCum_def}
Let $N$ be a $D$-variate stationary process. For any $s\in\mathbb{N}$ and any $j_1,\cdots,j_s\in\{1,\cdots,D\}$, the $j_1\cdots j_s$-th cumulant measure $C_{j_1\cdots j_s}$ of $N$ is defined for Borel rectangles $A_1\times\cdots\times A_s$ by 
\begin{align}\label{cum_def0}
C_{j_1\cdots j_s}(A_1\times\cdots\times A_s)
=\mathrm{cum}(N_{j_1}(A_1),\cdots,N_{j_s}(A_s)),
\end{align}
and extended to a signed measure on the product Borel $\sigma$-algebra in $\mathbb{R}^{s}$. $C_{j_1\cdots j_s}$ exists provided that the expectations on the right-hand side of \eqref{cum_def0} exist. We refer to section 9.5 of \citet{Daley_and_Vere-Jones2008_volume2} for more information about the above-mentioned expectations.

For $s\geqslant2$, since the cumulant measure is translation invariant due to the stationarity of $N$, there exists a signed measure $C_{j_1\cdots j_s}^{\mathrm{red}}$ defined on $\mathbb{R}^{s-1}$ such that
\begin{align*}
\mathrm{d}C_{j_1\cdots j_s}(x_1,\cdots,x_s)=\mathrm{d}C_{j_1\cdots j_s}^{\mathrm{red}}(x_1 - x_s,\cdots,x_{s-1} - x_s)\mathrm{d}x_s,
\end{align*}
if ${C}_{j_1\cdots j_s}$ exists. We call $C_{j_1\cdots j_s}^{\mathrm{red}}$ a reduced cumulant measure.
\end{definition}

By combining \eqref{def_cum_rv}, \eqref{cum_def0} and the Campbell theorem (see equations (9.5.2) and (9.5.9) of \citet{Daley_and_Vere-Jones2008_volume2}), it can be seen that if the cumulant measure $C_{j_1\cdots j_s}(\cdot)$ of $N$ exists, then the following equation holds:
\begin{align}\label{cum_def1}
\mathrm{cum}\left\{\int_{\mathbb{R}}h_{k}(x)\mathrm{d}N_{j_k}(x),k=1,\cdots,s\right\} = \int_{\mathbb{R}^s}\prod_{k=1}^{s}h_{k}(x_k)\mathrm{d}C_{j_1\cdots j_s}(x_1,\cdots,x_{s}),
\end{align}
where $h_1,\cdots,h_s$ are any bounded measurable (complex-valued) functions with bounded support. Equation~\eqref{cum_def1} is important to our proofs as it connects the cumulant of stochastic integrals to a signed measure which is much easier to deal with.

When $s = 1$, the stationarity of $N$ means that the first-order cumulant measure $C_j(\cdot)=\mathrm{E}N_j(\cdot)$ is invariant under translation, which in turn implies that $C_{j} = \lambda_j l_{\mathrm{Leb}}$, $j=1,\cdots,D$, where $\lambda_j = \mathrm{E}N_j((0,1))$ is the average intensity of $N_j$ and $l_{\mathrm{Leb}}$ is the Lebesgue measure on $\mathbb{R}$ (see e.g. Theorem A.1.8 of \citet{Bremaud2020}). In particular, the average intensity vector of a multivariate stationary Hawkes process satisfies $\boldsymbol{\lambda}=(\lambda_1,\cdots,\lambda_D)^{\top}=(\boldsymbol{\mathrm{I}}_D-\boldsymbol{\nu})^{-1}\boldsymbol{\mu}$~\cite[Example 8.3(c)]{Daley_and_Vere-Jones2003}, where $\boldsymbol{\mathrm{I}}_D$ is $D\times D$ identity matrix, $\boldsymbol{\nu}$ and $\boldsymbol{\mu}$ are given in Definition~\ref{HP_def}. For $s\geqslant2$, it turned out that the reduced cumulant measures of a multivariate stationary Hawkes process are finite measures.
\begin{proposition}\label{HPCum_bound_prop}
Let $N$ be a multivariate stationary Hawkes process in Definition~\ref{HP_def}. For any integer $s\geqslant 2$ and any $j_1,\cdots,j_s\in\{1,\cdots,D\}$, the reduced cumulant measure $C_{j_1\cdots j_s}^{\mathrm{red}}$ of $N$ in Definition~\ref{HPCum_def} exists and is a finite measure. Moreover, we have
\begin{align*}
C_{j_1\cdots j_s}^{\mathrm{red}}(\mathbb{R}^{s-1}) \leqslant C_0^s (s-1)!,
\end{align*}
where $C_0>0$ is a constant solely determined by $D$, $\boldsymbol{\mu}$ and $\boldsymbol{\nu}$ which are given in Definition~\ref{HP_def}.
\end{proposition}
The proof is given in Section~\ref{cum_density_HP_sect} of the Appendices. Basically, Proposition~\ref{HPCum_bound_prop} is a direct consequence of the tree representation pertaining to the cumulant measures proven in \citet{Jovanovic_et_al2015}. 
\begin{remark}
{ A common route to frequency-domain limit theory is to assume a Brillinger-type cumulant condition, namely that for each order $s\geqslant2$ the reduced cumulant measures are finite (or have summable densities), which provides control of long-range dependence and ensures well-behaved cumulants of finite Fourier transforms; see \citet{Brillinger1972}. For stationary multivariate Hawkes processes, Proposition~\ref{HPCum_bound_prop} implies precisely this kind of cumulant summability: for every $s\geqslant2$, the reduced cumulant measures $C^{\mathrm{red}}_{j_1\cdots j_s}$ exist and are finite, with explicit factorial growth bounds. These bounds can be viewed as an analogue of Brillinger’s mixing conditions which are stated in terms of reduced factorial cumulant measures (see e.g. Definition 1 of \citet{Heinrich2016}) and serve as the main input for our non-asymptotic cumulant bounds and the ensuing functional CLT for the spectral empirical process, without requiring $\alpha$-mixing assumptions (see e.g. \citet{Cheysson2022_and_Lang2022}).
}    
\end{remark}

Our limit theory concerning Hawkes processes is formulated in the frequency domain, where higher-order cumulants control non-Gaussian fluctuations of the finite Fourier transforms of the point processes. The Fourier transforms of reduced cumulant measures are very useful in this regard. 
\begin{definition}[Cumulant spectral density]\label{HPcsd_def}
For any integer $s\geqslant 2$ and any $j_1,\cdots,j_s\in\{1,\cdots,D\}$, the cumulant spectral density $f_{j_1\cdots j_s}$ of order $(j_1,\cdots, j_s)$ of a stationary multivariate Hawkes process $N$ is defined as the Fourier transform of reduced cumulant measures $C_{j_1\cdots j_s}^{\mathrm{red}}$, that is,
\begin{align*}
f_{j_1\cdots j_s}(\omega_1,\cdots,\omega_{s-1})
= \int_{\mathbb{R}^{s-1}}\exp\left(-\mathrm{i}\sum_{j=1}^{s-1}\omega_j x_j\right)\mathrm{d}C_{j_1\cdots j_s}^{\mathrm{red}}(x_1,\cdots,x_{s-1}),
\end{align*}
where $\mathrm{i}$ is the imaginary unit. Specifically, when $s=2$, we call $f_{j_1 j_2}$ the $(j_1,j_2)$-th Bartlett spectral density (or simply spectral density) of the Hawkes process and the matrix $\boldsymbol{f}_2 = (f_{j_1 j_2})_{j_1,j_2=1,\cdots,D}$ the spectral density matrix.
\end{definition}

{\begin{remark}
For general multivariate stationary point processes, the coherence $|f_{j_1j_2}|/\sqrt{f_{j_1j_1}f_{j_2j_2}}$, similar to the correlation of two random variables, is a measurement of linear dependency between $N_{j_1}$ and $N_{j_2}$, while $f_{j_1\cdots j_s}$ ($s\geqslant3$) governs higher-order dependence. In particular, $f_{j_1j_2j_3j_4}$ appears in the covariance structure of quadratic forms in finite Fourier transforms (cf.\ Theorem~\ref{Thm_CLT}). Surprisingly, for Hawkes processes, as will be shown in Section~\ref{section_ind_test}, the condition that $f_{j_1j_2}(0) = 0$ for any $j_1\neq j_2$ in fact implies the independence of $N_1,\cdots,N_D$.
\end{remark}
}

The Bartlett spectral density matrix $\boldsymbol{f}_2$ deserves special attention as this is the object of interest of this manuscript. As shown in Example 8.3(c) of \citet{Daley_and_Vere-Jones2003}, $\boldsymbol{f}_2$ can be expressed in terms of $\boldsymbol{\mu}$, $\boldsymbol{\nu}$ and $\boldsymbol{G}$ given in Definition~\ref{HP_def}:
\begin{align}
&\boldsymbol{f}_2(\omega) = (\boldsymbol{\mathrm{I}}_{D} - \boldsymbol{\nu}\odot\hat{\boldsymbol{G}}(\omega))^{-1}\mathrm{Diag}(\boldsymbol{\lambda})(\boldsymbol{\mathrm{I}}_{D} - (\boldsymbol{\nu}\odot\hat{\boldsymbol{G}}(-\omega))^{\top})^{-1} \notag \\
&= (\boldsymbol{\mathrm{I}}_{D} - \boldsymbol{\nu}\odot\hat{\boldsymbol{G}}(\omega))^{-1}\mathrm{Diag}(\boldsymbol{\lambda})\left[(\boldsymbol{\mathrm{I}}_{D} - \boldsymbol{\nu}\odot\hat{\boldsymbol{G}}(\omega))^{-1}\right]^{H}, \label{eq_psdmat}
\end{align}
where $\boldsymbol{\mathrm{I}}_{D}$ is $D\times D$ dimensional identity matrix, $\boldsymbol{\lambda} = (\boldsymbol{\mathrm{I}}_{D} - \boldsymbol{\nu})^{-1}\boldsymbol{\mu}$ is the vector of average intensity, $\hat{g}_{ij}({\omega}) = \int_{\mathbb{R}}\exp(-\mathrm{i}{\omega}x)g_{ij}(x)\mathrm{d}x$ is the Fourier transform of the mutual-excitation kernel $g_{ij}$, $\hat{\boldsymbol{G}} = \{\hat{g}_{ij}\}_{i,j=1,\cdots,D}$ and $\odot$ means the Hadamard (element-wise) product of matrices, and $\boldsymbol{A}^H$ denotes the conjugate transpose of matrix $\boldsymbol{A}$. It can be easily seen that 
$\boldsymbol{f}_2({\omega})$ is Hermitian positive-definite for any ${\omega}\in\mathbb{R}$.

Here is an example of a class of multivariate stationary Hawkes processes.
\begin{example}[Multivariate Fractional Hawkes process]\label{example_FHP}
Let  
\begin{align*}
E_{a,b}(z) = \sum_{n=0}^{\infty}\frac{z^n}{\Gamma(an+b)}    
\end{align*}
be the (two-parameter) Mittag-Leffler function. This function appears frequently in the solution of a variety of fractional differential equations. Recently, this function was employed in constructing a new family of Hawkes processes (see \citet{Chen_et_al2021}, \citet{Habyarimana_et_al2023} and \citet{Davis_et_al2024}). Let $\mathds{1}_A$ be the indicator function of set $A$. It can be verified that $f_{\mathrm{ML}}(x;\beta,c)= c^{\beta}x^{\beta - 1}E_{\beta,\beta}(-(cx)^{\beta})\mathds{1}_{\{x>0\}}$, $c>0, \beta\in(0,1]$ is a probability density function on $(0,+\infty)$ and we call the distribution corresponding to $f_{\mathrm{ML}}(x;c,\beta)$ the Mittag-Leffler distribution. Let $g_{ij}(x) = f_{\mathrm{ML}}(x;\beta_{ij}, c_{ij})$, $c_{ij}>0,\beta_{ij}\in(0,1]$, $i,j=1,\cdots,D$, be the mutual-excitation kernels in Definition~\ref{HP_def}, then the resulting Hawkes process is called a multivariate Fractional Hawkes process \citep{Davis_et_al2026}. It is worth noting that if $\beta=1$, $f_{\mathrm{ML}}(x;1,c)$ is exactly the probability density function of an exponential distribution. It follows that the multivariate Fractional Hawkes process can be viewed as a generalization of the classical Hawkes process (see e.g. \citet{Hawkes1971}). When $\beta<1$, the corresponding Mittag-Leffler distribution is heavy-tailed (i.e. $\int_0^{\infty}  x^{\alpha}f_{\mathrm{ML}}(x;\beta,c)\mathrm{d}x = \infty $ for any $\alpha\geqslant\beta$).

According to \citet{Kozubowski2001}, the Mittag-Leffler distribution is a geometric stable distribution and hence its Fourier transform can be explicitly expressed as
\begin{equation*}
\hat{f}_{\mathrm{ML}}(\omega;\beta,c) = \int_0^{\infty}  \exp(-\mathrm{i}\omega x)f_{\mathrm{ML}}(x;\beta,c)\mathrm{d}x 
=[1+c^{-\beta}(\mathrm{i}\omega)^\beta]^{-1} .
\end{equation*}
With this result, the expression of the spectral density matrix of the multivariate Fractional Hawkes process is readily available and is easy to calculate. Computation of the Mittag-Leffler function, on the other hand, is often practically onerous, despite various algorithms for evaluating $E_{a,b}(z)$ (cf. \citet{Gorenflo_et_al2002} and \citet{Garrappa2015}).
\end{example}

Suppose $\{t_1^{k_1},\cdots,t_{n}^{k_{n}}\}$, $k_1,\cdots,k_{n}\in\{1,\cdots,D\}$ is a realization of a multivariate Hawkes process $N = (N_1,\cdots,N_D)^{\top}$ on time interval $(0,T)$. The finite Fourier transform associated with the realization is defined as $J_T(\omega) := (J_T^{1}(\omega),\cdots,J_T^{D}(\omega))^{\top}$, where
\begin{align*}
J_T^{j}(\omega) = T^{-\frac{1}{2}}\sum_{l=1}^{n} \exp(-\mathrm{i}\omega t_l^{k_l})\mathds{1}_{\{k_l = j\}} = T^{-\frac{1}{2}}\int_0^T \exp(-\mathrm{i} \omega t)N_{j}(\mathrm{d}t).
\end{align*}
As shown in the following section, the finite Fourier transform is a basic yet powerful tool for estimating the spectral density matrix and investigating other features reflected in the second-order properties.

\section{Spectral empirical process, the Whittle likelihood and their applications}\label{sect_spectral}
The periodogram matrix $J_T(\omega)J_T^H(\omega)$ is not, by itself, a consistent estimator of $\boldsymbol f_2(\omega)$ at a fixed frequency. The basic idea of the frequency-domain approach is therefore to aggregate information across Fourier frequencies. This leads to the spectral empirical process, which behaves like a Riemann-sum approximation to an integrated spectral functional and provides the bridge from finite Fourier transforms to estimation and testing. We first study this process in Section~\ref{spec_ep_sect}, then use it in Section~\ref{Whttiel_sect} to analyze the Whittle likelihood, and finally apply the same frequency-domain viewpoint in Section~\ref{section_ind_test} to construct an independence test.

\subsection{Spectral empirical process}\label{spec_ep_sect}
The spectral empirical process is defined as
\begin{align}\label{def_SEP}
    A_T(\Phi) = \frac{1}{T}\sum_{p=1}^{M_T} J_T^{H}(\omega_p)\Phi(\omega_p)J_T(\omega_p),
\end{align}
where $\omega_p = 2\pi p/T$ is the Fourier frequencies, $M_T$ is the number of Fourier frequencies used to construct the spectral empirical process and $\Phi(\cdot)=\{\phi_{ij}(\cdot)\}_{i,j=1,\cdots,D}$ is a matrix of (complex-valued) functions. Heuristically, $A_T(\Phi)$ is a data-based Riemann-sum approximation of the integrated functional
\[
\frac{1}{2\pi}\int_0^{2\pi L}\text{tr} \big(\Phi(\omega)\boldsymbol f_2(\omega)\big) d\omega,
\]
with $L=\lim_{T\rightarrow\infty} M_T/T$. 

For a stationary multivariate Hawkes process, the spectral empirical process enjoys desirable asymptotic properties summarized in the following Theorem~\ref{Thm_CLT}, which is a uniform functional Central Limit Theorem (CLT) for $A_T(\cdot)$, whose proof can be found in Section~\ref{proof_CLT_sect} of the Appendices. Thanks to Proposition~\ref{HPCum_bound_prop}, the proof requires minimal stationarity assumptions and relies only on cumulant summability rather than $\alpha$-mixing.
\begin{theorem}[CLT for the spectral empirical process]\label{Thm_CLT}
Let $\Theta\subset\mathbb{R}^d$ be compact. $\Phi_{\theta}(\cdot) = \{\phi_{\theta,ab}(\cdot)\}_{a,b=1,\cdots,D}$ is an Hermitian matrix of continuous, bounded functions on $\mathbb{R}$ satisfying $\tau_0:=\sup_{\theta\in\Theta}\max_{a,b=1,\cdots,D}\|\phi_{\theta,ab}\|_{\infty}< \infty$. Moreover, there exist constants $C,\beta>0$ such that for any $\theta_1,~\theta_2\in\Theta$, $\max_{a,b=1,\cdots,D}\|\phi_{\theta_1, ab} - \phi_{\theta_2, ab}\|_{\infty} \leqslant C\|\theta_1 - \theta_2\|^{\beta}$. Suppose also $M_T/T \rightarrow L < \infty$. Then
\begin{align*}
T^{1/2}(A_T(\Phi_{\cdot}) - \mathrm{E}A_T(\Phi_{\cdot})) \rightarrow G(\cdot) \mathrm{~~weakly~in~} l^{\infty}(\Theta),
\end{align*}
where $l^{\infty}(\Theta)$ is the set of all
uniformly bounded, real functions on $\Theta$ endowed with uniform distance, $\{G(\theta):\theta\in\Theta\}$ denotes a tight, mean zero Gaussian process with covariance structure
\begin{align*}
&\mathrm{cov}(G(\theta_1) , G(\theta_2)) \\
=&\frac{1}{(2\pi)^2}\int_{[0,2\pi L]^2} \sum_{a_1=1}^{D} \sum_{a_2=1}^{D}\sum_{b_1=1}^{D} \sum_{b_2=1}^{D}\phi_{\theta_1, a_2 a_1}(\omega_{1}) \phi_{\theta_2, b_2 b_1}(\omega_{2})f_{a_1 a_2 b_1 b_2}(\omega_{1},-\omega_{1},\omega_{2})\mathrm{d}\omega_1\mathrm{d}\omega_2 \\
&+ \frac{1}{2\pi}\int_{[0,2\pi L]} \sum_{a_1=1}^{D} \sum_{a_2=1}^{D}\sum_{b_1=1}^{D} \sum_{b_2=1}^{D} \phi_{\theta_1, a_2 a_1}(\omega) \phi_{\theta_2, b_2 b_1}(\omega) f_{a_1 b_2}(\omega)f_{a_2 b_1}(-\omega)\mathrm{d}\omega 
\end{align*}
\end{theorem}

\subsection{Whittle likelihood and estimator}\label{Whttiel_sect}
We now specialize the spectral empirical process to the model-based choice $\Phi=\Psi_\theta^{-1}$, which yields the Whittle objective. This converts the second-order structure of the Hawkes process into a tractable frequency-domain criterion. We first study its deterministic limit (which plays the role of a frequency-domain Kullback Leibler contrast), and then use that limit to establish consistency and asymptotic normality of the resulting estimator.

Inspired by Theorem 4.2 of \citet{Brillinger1972} which basically states that the finite Fourier transforms $J_T(x_1),\cdots,J_T(x_k)$ with distinct positive frequencies $x_1,\cdots,x_k$ are asymptotically independent complex normal variates, we define the negative log-Whittle likelihood as a consequence of pseudo-Gaussian likelihood:
\begin{align}\label{Whittle_llike_def_eq}
l_W(\theta) 
= A_T(\Psi_{\theta}^{-1}) + \frac{1}{T}\sum_{p=1}^{M_T} \log\det\Psi_{\theta}(\omega_p)
= \frac{1}{T}\sum_{p=1}^{M_T}\left[J_T^{H}(\omega_p)\Psi_{\theta}^{-1}(\omega_p)J_T(\omega_j) + \log\det\Psi_{\theta}(\omega_p)\right] ,
\end{align}
where $\{\Psi_{\theta}:\theta\in\Theta\}$ is a family of spectral density matrices of stationary Hawkes processes and $\Theta\subset\mathbb{R}^d$ is the parameter space. The Whittle estimator $\hat{\theta}_W$ is the minimizer of $l_W$. We will later impose suitable conditions to ensure the existence and uniqueness of $\hat{\theta}_W$. 
{\begin{remark}
The Whittle likelihood replaces the point-process likelihood by a frequency-domain pseudo-likelihood built from (approximately) uncorrelated Fourier ordinates. If the number of Fourier frequencies $M_T$ satisfies $M_T/T \rightarrow L\in(0,+\infty)$ when $T\rightarrow\infty$, then the negative log-Whittle likelihood can be shown to converge in probability to 
\begin{align}\label{eq_freq_KL}
h_L(\theta) = \frac{1}{2\pi}\int_0^{2\pi L}\left[\log\det\Psi_{\theta}(x) + \mathrm{tr}(\boldsymbol{f}_2(x)\Psi_{\theta}^{-1}(x))  \right]\mathrm{d}x .    
\end{align}
Equation~\eqref{eq_freq_KL} can be regarded as a matrix-valued KL-type divergence on a frequency band $[0,2\pi L]$. The following Proposition~\ref{prop_freq_KL_min} indicates that \eqref{eq_freq_KL} is minimized uniquely when $\Psi_\theta(\omega)=\boldsymbol f_2(\omega)$ on the band. The limit $L$ of $M_T/T$ controls the asymptotic frequency range used by the estimator.
\end{remark}
} 

\begin{proposition}\label{prop_freq_KL_min}
Suppose $\inf_{\theta\in\Theta}\inf_{x\in\mathbb{R}}\det\Psi_{\theta}(x)>0$ and $\Psi_{\theta}(\cdot)$ is continuous for each $\theta\in\Theta$, then
\begin{align*}
    h_L(\theta) \geqslant \frac{1}{2\pi}\int_0^{2\pi L}\log\det\boldsymbol{f}_2(x)\mathrm{d}x + LD .
\end{align*} 
The above inequality becomes an equation if and only if $\Psi_{\theta}(x) = \boldsymbol{f}_2(x)$ for all $x\in[0,2\pi L]$.
\begin{proof}
Note that we can re-write $h_L$ as
\begin{align*}
&h_L(\theta)= \frac{1}{2\pi}\int_0^{2\pi L}\left[-\log\det\boldsymbol{f}_2^{1/2}(x)\Psi_{\theta}^{-1}(x)\boldsymbol{f}_2^{1/2}(x) 
+ \mathrm{tr}(\boldsymbol{f}_2^{1/2}(x)\Psi_{\theta}^{-1}(x)\boldsymbol{f}_2^{1/2}(x))  \right]\mathrm{d}x \\
&+ \frac{1}{2\pi}\int_0^{2\pi L}\log\det\boldsymbol{f}_2(x)\mathrm{d}x .
\end{align*} 
The desired inequality can then be obtained by applying Proposition~\ref{proposition_log_trace_ineq} in the Appendices. Note that $\boldsymbol{f}_2^{1/2}(\cdot)\Psi_{\theta}^{-1}(\cdot)\boldsymbol{f}_2^{1/2}(\cdot)$ is positive-definite and $\log\det\Psi_{\theta}(\cdot) + \mathrm{tr}(\boldsymbol{f}_2(\cdot)\Psi_{\theta}^{-1}(\cdot))$ is a continuous function for each $\theta\in\Theta$. The condition for the equality to hold also follows from Proposition~\ref{proposition_log_trace_ineq}.
\end{proof}
\end{proposition}

Here are some basic assumptions on the model. When deriving the asymptotic normality of the Whittle likelihood, we will make additional assumptions about the smoothness of $\Psi_{\theta}$ with respect to $\theta$.
\begin{assumption}\label{assumption_model} The following regularity conditions are assumed for the remainder of this paper:
\begin{enumerate}[(1)]
\item The parameter space $\Theta\subset\mathbb{R}^d$ is compact.
\item $\Psi_{\theta}^{-1}(\cdot)$ exists and is an Hermitian matrix of uniformly continuous, bounded functions. Moreover, all eigenvalues of $\Psi_{\theta}^{-1}(\omega)$ are bounded from below by a constant $c_1>0$ and from above by a constant $c_2>0$ uniformly in $\omega$ and $\theta$, and there exist constants $C,\beta>0$ such that for any $\tau_1,~\tau_2\in\Theta$, $\max_{a,b}\|[\Psi_{\tau_1}^{-1}(\cdot)]_{ab} - [\Psi_{\tau_2}^{-1}(\cdot)]_{ab}\|_{\infty} \leqslant C\|\tau_1 - \tau_2\|^{\beta}$.
\item There exists $\theta_0$ lying in the interior of $\Theta$ such that $\boldsymbol{f}_2(x) = \Psi_{\theta_0}(x)$ for all $x\in\mathbb{R}$. Moreover, there exists a constant $L>0$ such that for any $\theta\neq\theta_0$, we have $h_L(\theta)> h_L(\theta_0)$.
\end{enumerate}

\end{assumption}
\begin{remark}
As a consequence of Assumption~\ref{assumption_model}.(2), we have $\max_{a,b}\sup_{\theta\in\Theta}\|[\Psi_{\theta}^{-1}(\cdot)]_{ab}\|_{\infty}< \infty$ and $\inf_{\theta\in\Theta}\inf_{x\in\mathbb{R}}\det\Psi_{\theta}^{-1}(x)>0$. Assumption~\ref{assumption_model}.(3) says that $\theta_0$ is a unique global minimum point of $h_L(\cdot)$ on $\Theta$ for some $L>0$, which ensures model identifiability. Proposition~\ref{prop_freq_KL_min} and Assumption~\ref{assumption_model}.(3) together guarantee that if $L$ satisfies Assumption~\ref{assumption_model}.(3), then any $\tilde{L}>L$ also satisfies Assumption~\ref{assumption_model}.(3).
\end{remark}
The following theorem states the consistency of the Whittle estimator when the underlying Hawkes process is stationary and the parametric family satisfies the regularity conditions in Assumption~\ref{assumption_model}.

\begin{theorem}[Consistency of Whittle estimator]\label{theorem_consistency_WE}
Suppose Assumption~\ref{assumption_model} holds. Let $L$ be the constant in Assumption~\ref{assumption_model}.(3). Provided that $M_T/T \rightarrow L < \infty$ when $T\rightarrow\infty$, we will have $\hat{\theta}_W \rightarrow\theta_0$ in $\mathrm{P}$-probability.
\begin{proof}
It can be derived from Theorems~\ref{stochastic_conv_rate_thm} and \ref{deterministic_conv_thm} in the Appendices that
\begin{align*}
\sup_{\theta\in\Theta}\left|l_{W}(\theta) - h_L(\theta)\right| \rightarrow 0
\end{align*}
in $\mathrm{P}$-probability, where $l_{W}$ is the negative log-Whittle likelihood defined in \eqref{Whittle_llike_def_eq}. Recall that $\hat{\theta}_W$ and $\theta_0$ are minimizers of $l_{W}$ and $h_L$, respectively. Therefore,
\begin{align*}
&0 \leqslant h_L(\hat{\theta}_W) - h_L(\theta_0)
= (h_L(\hat{\theta}_W) - l_W(\hat{\theta}_W)) + (l_W(\hat{\theta}_W) - l_{W}(\theta_0)) + (l_{W}(\theta_0) - h_L(\theta_0))  \\
&\leqslant 2\sup_{\theta\in\Theta}\left|h_L(\theta) - l_{W}(\theta)\right|\rightarrow 0
\end{align*}
in $\mathrm{P}$-probability. As $\theta_0$ is the unique minimizer of $h_L$ lying in the interior of $\Theta$ by assumption and that $h_L$ is continuous on $\Theta$, we get $\hat{\theta}_W \rightarrow\theta_0$ in $\mathrm{P}$-probability.
\end{proof}
\end{theorem}

We can further derive the asymptotic normality of the Whittle estimator. However, stronger assumptions such as some smoothness requirements with respect to both the parameter and frequency on the model and the true spectral density matrix need to be imposed. Let $\partial_{i} = \frac{\partial}{\partial\theta_i}$, $\partial_{ij} = \frac{\partial^2}{\partial\theta_i\theta_j}$, $\nabla = (\partial_1,\cdots,\partial_d)^{\top}$ and $\nabla^2 = \{\partial_{ij}\}_{i,j=1\cdots,d}$. For any $i,j=1,\cdots,d$, denote $\partial_{i}\Psi_{\theta}^{-1}(\cdot) = \{\partial_{i}[\Psi_{\theta}^{-1}(\cdot)]_{ab}\}_{a,b=1,\cdots,D}$ and $\partial_{ij}\Psi_{\theta}^{-1}(\cdot) = \{\partial_{ij}[\Psi_{\theta}^{-1}(\cdot)]_{ab}\}_{a,b=1,\cdots,D}$.
\begin{theorem}\label{thm_asymptotic_normality_Whittle}
Suppose Assumption~\ref{assumption_model} holds.  Moreover, for any $i,j=1\cdots,d$ and any $a,b=1,\cdots,D$, $\partial_{ij}[\Psi_{\theta}^{-1}(\omega)]_{ab}$ exists and is uniformly continuous in $\omega$ and H{\"o}lder continuous in $\theta$. Additionally, there exist constants $C>0$ and $\beta>1/2$ such that $\max_i\max_{a,b}\sup_{\theta\in\Theta}|[\partial_i\Psi_{\theta}^{-1}(x)]_{ab} - [\partial_i\Psi_{\theta}^{-1}(y)]_{ab}|\leqslant C|x-y|^{\beta}$. 

We further assume that the mutual-excitation functions of the underlying Hawkes process satisfy $\int_{\mathbb{R}}|x|^{\alpha}g_{ab}(x)\mathrm{d}x<\infty$, $a,b=1,\cdots,D$, for some $\alpha>1/2$. Let $L$ be the constant in Assumption~\ref{assumption_model}.(3).
Provided that $|M_T/T - L| = o(T^{-1/2})$ when $T\rightarrow\infty$, $\sqrt{T}(\hat{\theta}_W -\theta_0)$ converges in distribution to $ N(\boldsymbol{0}_{d\times1}, \boldsymbol{\Gamma}^{-1}\boldsymbol{V}\boldsymbol{\Gamma}^{-1})$, where $\boldsymbol{\Gamma} = \nabla^2 h_L(\theta_0)$, $\boldsymbol{V} = \{v_{ij}\}_{i,j=1,\cdots,d}$,
\begin{align*}
&v_{ij}=
\frac{1}{(2\pi)^2}\int_{[0,2\pi L]^2} \sum_{a_1=1}^{D} \sum_{a_2=1}^{D}\sum_{b_1=1}^{D} \sum_{b_2=1}^{D}[\partial_i\Psi_{\theta_0}^{-1}(\omega_1)]_{a_2 a_1} [\partial_j\Psi_{\theta_0}^{-1}(\omega_2)]_{b_2 b_1} f_{a_1 a_2 b_1 b_2}(\omega_{1},-\omega_{1},\omega_{2})\mathrm{d}\omega_1\mathrm{d}\omega_2 \\
&+ \frac{1}{2\pi}\int_{[0,2\pi L]} \sum_{a_1=1}^{D} \sum_{a_2=1}^{D}\sum_{b_1=1}^{D} \sum_{b_2=1}^{D} [\partial_i\Psi_{\theta_0}^{-1}(\omega)]_{a_2 a_1} [\partial_j\Psi_{\theta_0}^{-1}(\omega)]_{b_2 b_1} f_{a_1 b_2}(\omega)f_{a_2 b_1}(-\omega)\mathrm{d}\omega .
\end{align*}
\begin{proof}
First, note that since $\theta_0$ lies in the interior of $\Theta$ and minimizes $h_L$, the twice continuous differentiability of $h_L$ implies that the gradient $\nabla h_L(\theta_0) = \boldsymbol{0}_{d\times1}$ and the Hessian matrix $\nabla^2h_L(\theta_0)$ is positive definite. By Theorem~\ref{theorem_consistency_WE}, $\hat{\theta}_W \rightarrow\theta_0$ in $\mathrm{P}$-probability. This implies that the probability of $\partial_i l_W(\hat{\theta}_{W}) = 0$ for all $i=1,\cdots,d$ tends to 1 when $T\rightarrow\infty$. Therefore, by the mean value theorem, 
\begin{align*}
- \partial_i l_W(\theta_0)
=\partial_i l_W(\hat{\theta}_{W}) - \partial_{i} l_W(\theta_0) 
=(\nabla\partial_i l_W(\xi_T^{(i)}))^{\top}(\hat{\theta}_{W} - \theta_0)
\end{align*}
with $\|\xi_T^{(i)} - \theta_0 \|\leqslant\|\hat{\theta}_{W} - \theta_0\|$ for $i=1,\cdots,d$ hold with probability tending to 1. Therefore, for each $i=1,\cdots,d$, $\xi_T^{(i)} \rightarrow\theta_0$ in probability. We can re-write the above equations in matrix form, namely
\begin{align*}
-\nabla l_W(\theta_0) = \boldsymbol{\Gamma}_T(\hat{\theta}_W - \theta_0),    
\end{align*}
where $[\boldsymbol{\Gamma}_T]_{ij} = \partial_{ij}l_W(\xi_T^{(j)})$. Note that by assumption, $[\partial_{ij}\Psi_{\theta}^{-1}(\omega)]_{ab}$ is uniformly continuous in $\omega$ and H{\"o}lder continuous in $\theta$. Theorems~\ref{stochastic_conv_rate_thm} and \ref{deterministic_conv_thm} in the Appendices, and the Leibniz integral rule together imply 
\begin{align*}
\sum_{i,j=1}^{d}\sup_{\theta\in\Theta}|\partial_{ij}l_W(\theta) - \partial_{ij}h_L(\theta)| \rightarrow 0   
\end{align*}
in probability. An application of the continuous mapping theorem then shows $\boldsymbol{\Gamma}_T \rightarrow\nabla^2 h_L(\theta_0)=\boldsymbol{\Gamma}$ in probability. Finally, due to the assumption that $[\partial_{i}\Psi_{\theta}^{-1}(\omega)]_{ab}$ is $\beta$-H{\"o}lder continuous in $\omega$ with $\beta>1/2$ and a finite H{\"o}lder coefficient independent of $\theta$, we have by Theorem~\ref{deterministic_conv_rate_thm} in the Appendices that $\sqrt{T}(\mathrm{E}\nabla l_W(\theta_0) - \nabla h_L(\theta_0))\rightarrow \boldsymbol{0}_{d\times1}$. The proof is concluded by noting that $\sqrt{T}(\nabla l_W(\theta_0) - \mathrm{E}\nabla l_W(\theta_0))$ converges in distribution to $N(\boldsymbol{0}_{d\times1}, \boldsymbol{V})$ due to Theorem~\ref{Thm_CLT} and the Cram\'er-Wold device, and $\nabla h_L(\theta_0) = \boldsymbol{0}_{d\times1}$.

\end{proof}

\end{theorem}

\subsection{Simple test of independence}\label{section_ind_test}
In real-world applications, people are often interested in finding out if the subprocesses of a multivariate point process are independent as it may be related to some scientific questions or it could greatly simplify the inference procedure. For instance, earthquakes in a region often occur along several fault segments, so the associated earthquake catalog may naturally be viewed as a multivariate point process whose labels correspond to the fault segments. Testing for independence among the subprocesses is therefore of practical importance, as it can substantially simplify the modeling task by allowing earthquakes from different fault segments to be treated separately. In the class of stationary multivariate Hawkes process, the independence of subprocesses is in fact encoded in the spectral density matrix at frequency $0$. Let $N=(N_1,\cdots,N_D)$ be a stationary multivariate Hawkes process. In view of the conditional intensity of $N$ in Definition~\ref{HP_def} and the fact that the probability structure of $N$ is uniquely determined by the conditional intensity \citep[Proposition 7.3.IV]{Daley_and_Vere-Jones2003}, one can see that $N_1,\cdots,N_D$ are independent if and only if the interactions matrix $\boldsymbol{\nu}$ is diagonal (i.e. there will be no cross-excitation). Noting that when evaluating the spectral density matrix at frequency $0$ (cf. Equation~\eqref{eq_psdmat}), we have
\begin{align*}
\boldsymbol{f}_2(0)=\big(\boldsymbol{\mathrm{I}}_D-\boldsymbol\nu\big)^{-1}\text{Diag}(\boldsymbol\lambda)\big[(\boldsymbol{\mathrm{I}}_D-\boldsymbol\nu)^{-1}\big]^{\top}.  
\end{align*}
Clearly, a diagonal $\boldsymbol{\nu}$ leads to $\boldsymbol{f}_2(0)$ being diagonal. Conversely, if $\boldsymbol{f}_2(0)$ is diagonal, Proposition~\ref{uniq_mat_eq_proposition} in the Appendices guarantees that there is only one nonnegative matrix $\boldsymbol{\nu}$ with spectral radius smaller than 1 satisfying the above equation and this $\boldsymbol{\nu}$ has to be diagonal. Therefore, the independence of the subprocesses of $N$ is equivalent to $\boldsymbol{f}_2(0)$ being diagonal. If there is a consistent estimator $\hat{\boldsymbol{f}}_2(0)$ of $\boldsymbol{f}_2(0)$, we could tell whether the subprocesses are independent by assessing the deviation of $\hat{\boldsymbol{f}}_2(0)$ from a diagonal matrix.

If we are to derive an estimator of $\boldsymbol{f}_2(0)$ using the Fourier transform $J_T^{u}(\omega)$, $u=1,\cdots,D$, setting $\omega$ to $0$ is never a good idea as $J_T^{u}(0) = N_{u}((0,T))/T$ reflects the mean intensity of the process rather than the second-order properties. However, considering the Fourier frequencies $\omega_p = 2\pi p/T$, $p=1,\cdots,M_T$ with $M_T\rightarrow\infty$ and $M_T/T\rightarrow0$, we observe that for any $u,v=1,\cdots,D$,
\begin{align*}
\left|\frac{1}{M_T}\sum_{p=1}^{M_T} f_{uv}(\omega_p) - f_{uv}(0) \right|
\leqslant \sup_{\omega\in[0,2\pi M_T/T]}|f_{uv}(\omega) - f_{uv}(0)|\rightarrow 0
\end{align*}
when $T\rightarrow\infty$ due to the fact that the cross spectral density $f_{uv}$, which is the Fourier transform of a finite measure, is uniformly continuous at frequency 0. This observation prompts us to use Fourier transform with Fourier frequencies close to $0$ to construct an estimator of $\boldsymbol{f}_2(0)$. Motivated by the Taylor expansion of the spectral density of time series at frequency 0, \citet{McElroy_and_Politis2025} studied the following linear regression problem $\mathrm{Re}J_{T}^{u}(\omega_p)J_{T}^{v}(-\omega_p) = a_0 + a_1 \omega_p^2 + \epsilon_s$, $p=1,\cdots,M_T$ and claimed that the resulting intercept $\hat{a}_0$ is a consistent estimator for $[\boldsymbol f_2(0)]_{uv}$. We will show that for a stationary multivariate Hawkes process, the estimators produced by this procedure are also consistent. We further show the asymptotic normality pertaining to the estimators of the off-diagonal entries of $\boldsymbol{f}_2(0)$ under the assumption that the subprocesses are independent. As a consequence, a simple test is then derived from the asymptotic normality.

Let $K$ be an non-negative, bounded function defined on $[-1,1]$ satisfying $K(-x) = K(x)$ for any $x\in[-1,1]$ and $\int_0^1K(t)\mathrm{d}t>0$. $K$ is often called kernel function in the literature. For any $\delta>0$, define $K_{\delta}(x) = (2\pi\delta)^{-1}K(x/(2\pi\delta))$. Let $\mathbf{X}$ be a $M_T\times 2$ matrix with the first column $(1,\cdots,1)^{\top}$ and second column $(\omega_1^2,\cdots,\omega_{M_T}^2)^{\top}$, $\mathbf{y}_{uv}=(\mathrm{Re}J_{T}^{u}(\omega_1)J_{T}^{v}(-\omega_1),\cdots,\mathrm{Re}J_{T}^{u}(\omega_{M_T})J_{T}^{v}(-\omega_{M_T}))^{\top}$ and $\mathbf{W} = \text{Diag}(K_{\delta_T}(\omega_1),\cdots,K_{\delta_T}(\omega_{M_T}))$, where $\delta_T = M_T/T$. Then the weighted least squares (WLS) estimate of the intercept of the above linear regression problem is $\hat{\phi}_{uv} = (1,0)(\mathbf{X}^{\top}\mathbf{W}\mathbf{X})^{-1}\mathbf{X}^{\top}\mathbf{W}\mathbf{y}_{uv}$. The ordinary least squares estimate can be obtained by letting $K(x)=1$. We have the following theorem whose proof can be found in Section~\ref{proof_Thm_McP_sect} in the Appendices.
\begin{theorem}\label{Thm_McP}
Let $\hat{\boldsymbol\Phi} = \{\hat{\phi}_{uv}\}_{u,v=1\cdots,D}$. If $M_T\rightarrow\infty$ and $\delta_T\rightarrow0$ when $T\rightarrow\infty$, then 
\begin{enumerate}[(a)]
    \item $\hat{\boldsymbol\Phi}$ converges in probability to $\boldsymbol f_2(0)$.
    \item If the underlying multivariate Hawkes process has diagonal interactions matrix, then the vector $(\sqrt{M_T}\hat{\phi}_{uv})_{1\leqslant u<v\leqslant D} = (\sqrt{M_T}\hat{\phi}_{12},\cdots,\sqrt{M_T}\hat{\phi}_{1D},\sqrt{M_T}\hat{\phi}_{23},\cdots,\sqrt{M_T}\hat{\phi}_{2D},\cdots,\sqrt{M_T}\hat{\phi}_{D-1~D})^{\top}$ converges in distribution to $(z_{uv})_{1\leqslant u<v\leqslant D}$, which is jointly normally distributed with $\mathrm{E}z_{uv}=0$ and 
    \begin{align*}
       \mathrm{cov}(z_{uv}, z_{pq}) 
       =\begin{cases}
         \frac{1}{2}f_{uu}(0)f_{vv}(0)\int_{0}^{1} \left(\frac{(H_4 - H_2 x^2)K(x)}{H_4H_0 - H_2^2}\right)^2\mathrm{d}x, & \text{if }u=p\text{ and }v=q.\\
         0, & \text{otherwise}.
         \end{cases} 
    \end{align*} 
    for any $1\leqslant u<v\leqslant D,~1\leqslant p<q\leqslant D$, where $H_{2l} = \int_{0}^{1}K(x)x^{2l}\mathrm{d}x$, $l=0,1,2$.
\end{enumerate}
\end{theorem}
Consequently, under the assumption that the subprocesses of the underlying multivariate stationary Hawkes process are jointly independent, 
\begin{align}\label{eq_ind_test_statistic}
M_T\frac{2}{\int_{0}^{1} \left(\frac{(H_4 - H_2 x^2)K(x)}{H_4H_0 - H_2^2}\right)^2\mathrm{d}x}\sum_{1\leqslant u<v\leqslant D}\frac{\hat{\phi}_{uv}^2}{f_{uu}(0)f_{vv}(0)}    
\end{align}
asymptotically follows a chi-square distribution with degrees of freedom $D(D-1)/2$. This allows us to make a judgment on whether the subprocesses are independent based on the p-value calculated using the data. In practice, since $f_{uu}(0)$, $u=1,\cdots,D$ are not known in general, we may substitute $f_{uu}(0)$ with $\hat{\phi}_{uu}$ in Equation~\eqref{eq_ind_test_statistic} when calculating the p-value.

\section{Numerical results}\label{section_numerical_experiment}
In this section, we will demonstrate the proposed methods via a variety of numerical experiments. The behavior of an estimator $\hat{\boldsymbol{\theta}} = (\hat{\theta}_1,\cdots,\hat{\theta}_d)^{\top}$ will be assessed in terms of its relative error defined as
\begin{align}\label{def_relative_error}
\mathrm{Relative~error}=\sum_{j=1}^d \frac{|\hat{\theta}_j - \theta_{0,j}|}{|\theta_{0,j}|},    
\end{align}
where $\boldsymbol{\theta}_0= (\theta_{0,1},\cdots,\theta_{0,d})^{\top}$ is the true value of the parameter. The figures and tables summarizing the numerical results can be found in Appendix~\ref{sec:simstudy_supp}. 

\subsection{Univariate Fractional Hawkes process}\label{example_univariate_FHP}
Let $N$ be a univariate Fractional Hawkes process considered in Example~\ref{example_FHP}. It can be shown that a necessary and sufficient condition for $N$ to be stationary is $0\leqslant\nu_{11}<1$. The spectral density function of $N$ is
\begin{align*}
f_{11}(\omega) = \frac{\mu_1}{1 - \nu_{11}}\frac{1}{|1 - \nu_{11}\hat{f}_{\mathrm{ML}}(\omega;\beta_{11},c_{11})|^2} .
\end{align*}
Suppose $\{t_1,\cdots,t_{N(T)}\}$ is a realization of $N$ on $(0,T)$. We will compute Whittle estimator (WE) and the exact maximum likelihood estimator (MLE) based on the observations and compare their performance. The MLE is the minimizer of the negative log-likelihood (given in equation (7.2.4) of \citet{Daley_and_Vere-Jones2003}):
\begin{align}\label{true_likelihood}
l(\theta) = -\sum_{k=1}^{N(T)} \log(\lambda_{1}(t_{k}|\mathcal{H}_{t_{k}-})) + \int_{0}^{T} \lambda_{1}(t|\mathcal{H}_{t-})\mathrm{d}t , 
\end{align}
where the conditional intensity $\lambda_{1}(t|\mathcal{H}_{t-})$ is given in Definition~\ref{HP_def} and $\boldsymbol{\theta} = (\mu_1,\nu_{11},\beta_{11},c_{11})$. For efficient computation, the R code for computing the Mittag-Leffler function is adapted from function \texttt{mlf} in \texttt{MittagLeffleR} R package (\citet{MittagLeffleR}) which is based on the Laplace-Inversion algorithm by \citet{Garrappa2015}. In order to compare the performance of WE with that of MLE, we consider 4 univariate Fractional Hawkes processes FH$1.$--FH$4.$ as data generating processes (DGP): Let $(b_1,b_2,b_3,b_4)= (0.4, 0.5, 0.6, 0.9)$. For $i=1,2,3,4$, the parameters of Hawkes process FH$i.$ are 
\begin{equation*}
\mu_1 = 1, \nu_{11} = 0.5, \beta_{11}=b_i, c_{11}=1. 
\end{equation*}  
Clearly, the four Hawkes processes FH${1.}$--FH${4.}$ share the same background rate $\mu_1$, branching ratio $\nu_{11}$ and the scale parameter of the Mittag-Leffler kernel $c_{11}$ but differ in the tail of the ML kernel, which is controlled by $\beta_{11}$. It can be seen that Theorem~\ref{thm_asymptotic_normality_Whittle}, which states the asymptotic normality of the WE, does not hold for FH1 and FH2 because $\int_0^{\infty}x^{\alpha}f_{\mathrm{ML}}(x;\beta_{11},c_{11})\mathrm{d}x=\infty$ for any $\alpha>0.5$ implied by $\beta_{11}\leqslant0.5$. From FH1 to FH4, the tail of the kernel becomes lighter. For each of the univariate Fractional Hawkes process, we generate 1000 independent copies on $(0,T)$ with $T=1250,2500$. To study the effect of the number of Fourier frequencies, we choose $M_T = \lfloor 2T\rfloor$ and $\lfloor T\log T\rfloor$. Table~\ref{table_comparison_uniFH} of Appendix~\ref{sec:simstudy_supp} contains the median (IQR) of the relative errors for all 1000 independent copies of each DGP. Table~\ref{table_runtime_uniFH} of Appendix~\ref{sec:simstudy_supp} contains the corresponding runtime for obtaining the estimates. Figure~\ref{fig_comparison_uniFH} of Appendix~\ref{sec:simstudy_supp} is a graphical complement to Table~\ref{table_comparison_uniFH}.

Figure~\ref{fig_comparison_uniFH} and Tables~\ref{table_comparison_uniFH}--\ref{table_runtime_uniFH} show a clear accuracy--speed tradeoff. The MLE is uniformly more accurate, but it is substantially slower to compute. This is unsurprising: under correct specification the MLE is asymptotically efficient, whereas its evaluation here requires repeated computation of \eqref{true_likelihood}, which is expensive because of the $O(n^2)$ cost in the event count $n$ and the numerical burden of evaluating the Mittag-Leffler function. As the self-excitation kernel becomes lighter-tailed, the gap in accuracy between MLE and WE narrows. Regarding the choice of $M_T$, both choices appear consistent, but $M_T=\lfloor T\log T\rfloor$ is usually slightly more accurate than $M_T=\lfloor 2T\rfloor$; that difference becomes small in the lighter-tailed cases. This is consistent with the findings of \citet[Section~5.1.2]{Bonnet_et_al2025}, who study the exponentially decaying case.

\subsection{Bivariate Fractional Hawkes process}
\label{example_bivariate_FHP}
Let $N$ be a bivariate Fractional Hawkes process considered in Example~\ref{example_FHP}. It can be shown that a necessary and sufficient condition for $N$ to be stationary is $0\leqslant\nu_{11},\nu_{22}<1$, $\nu_{12}\nu_{21}<(1 - \nu_{11})(1 - \nu_{22})$. The spectral density matrix is
\begin{align*}
\boldsymbol{f}_2(\omega) = 
\begin{pmatrix}
f_{11}(\omega) & f_{12}(\omega) \\
f_{21}(\omega) & f_{22}(\omega)
\end{pmatrix} ,
\end{align*}
where
\begin{align*}
\boldsymbol{\lambda} = 
\begin{pmatrix}
\lambda_1 \\
\lambda_2
\end{pmatrix}
=(\boldsymbol{\mathrm{Id}}_2 - \boldsymbol{\nu})^{-1}\boldsymbol{\mu}
=\frac{1}{(1 - \nu_{11})(1 - \nu_{22}) -\nu_{12}\nu_{21}}\begin{pmatrix}
\mu_1(1 - \nu_{22}) + \mu_2\nu_{12} \\
\mu_2(1 - \nu_{11}) + \mu_1\nu_{21}
\end{pmatrix} ,
\end{align*}
\begin{align*}
f_{11}(\omega) =  \frac{\lambda_1|1 - \nu_{22}\hat{f}_{\mathrm{ML}}(\omega;\beta_{22},c_{22})|^2 + \lambda_2|\nu_{12}\hat{f}_{\mathrm{ML}}(\omega;\beta_{12},c_{12})|^2}{|(1 - \nu_{11}\hat{f}_{\mathrm{ML}}(\omega;\beta_{11},c_{11}))(1 - \nu_{22}\hat{f}_{\mathrm{ML}}(\omega;\beta_{22},c_{22})) -\nu_{12}\hat{f}_{\mathrm{ML}}(\omega;\beta_{12},c_{12})\nu_{21}\hat{f}_{\mathrm{ML}}(\omega;\beta_{21},c_{21})|^2}   ,
\end{align*}
\begin{align*}
f_{12}(\omega) =  \frac{\lambda_1(1 - \nu_{22}\hat{f}_{\mathrm{ML}}(-\omega;\beta_{22},c_{22}))\nu_{21}\hat{f}_{\mathrm{ML}}(\omega;\beta_{21},c_{21}) + \lambda_2\nu_{12}\hat{f}_{\mathrm{ML}}(-\omega;\beta_{12},c_{12})(1 - \nu_{11}\hat{f}_{\mathrm{ML}}(-\omega;\beta_{11},c_{11}))}{|(1 - \nu_{11}\hat{f}_{\mathrm{ML}}(\omega;\beta_{11},c_{11}))(1 - \nu_{22}\hat{f}_{\mathrm{ML}}(\omega;\beta_{22},c_{22})) -\nu_{12}\hat{f}_{\mathrm{ML}}(\omega;\beta_{12},c_{12})\nu_{21}\hat{f}_{\mathrm{ML}}(\omega;\beta_{21},c_{21})|^2}  ,
\end{align*}
$f_{21}(\omega) = f_{12}(-\omega)$ and $f_{22}$ is obtained by swapping indices $1$ and $2$ in the expression of $f_{11}$. 

To demonstrate the consistency of the Whittle estimator, we consider the following data generating process:
\begin{enumerate}
 \item[FH5.] $N=(N_1,N_2)$ is a bivariate Fractional Hawkes process with $\boldsymbol{\mu} = (0.2, 0.1)^{\top}$ (or equivalently, $\boldsymbol{\lambda} = (13/3, 17/6)^{\top}$),
\begin{align*}
\boldsymbol{\nu} = 
\begin{pmatrix}
0.3 & 1 \\
0.5 & 0.2
\end{pmatrix} ,~
\begin{pmatrix}
c_{11} & c_{12} \\
c_{21} & c_{22}
\end{pmatrix} = 
\begin{pmatrix}
0.8 & 1.0 \\
0.9 & 1.1
\end{pmatrix} ,~
\begin{pmatrix}
\beta_{11} & \beta_{12} \\
\beta_{21} & \beta_{22}
\end{pmatrix} = 
\begin{pmatrix}
0.75 & 0.85 \\
0.8 & 0.9
\end{pmatrix} .
\end{align*}
\end{enumerate}
The parameters to be estimated are $\boldsymbol{\theta} = (\lambda_1^{-1},\lambda_2^{-1},\nu_{11},\nu_{21},\nu_{12},\nu_{22},\beta_{11},\beta_{21},\beta_{12},\beta_{22},c_{11},c_{21},c_{12},c_{22})$.
We generate $1000$ independent copies of $N$ on $(0,T)$ with $T=1250,~2500$. Whittle estimator (WE) is then calculated for each of the copy. We choose $M_T = \lfloor 2T\rfloor$ and $\lfloor T\log T\rfloor$ for this experiment. The median and IQR (in brackets) of the relative errors are presented in Table~\ref{table_comparison_biFH} of Appendix~\ref{sec:simstudy_supp}. A summary of the runtime is recorded in Table~\ref{table_runtime_biFH} of Appendix~\ref{sec:simstudy_supp}. It seems that both choices of $M_T$ make the estimator consistent and the performance of WE with $M_T = \lfloor T\log T\rfloor$ is slightly better than that with $M_T = \lfloor 2T\rfloor$.

\subsection{Detecting independence among subprocesses}\label{example_estimation_nu}
In this section, we will assess the performance of the simple test statistic \eqref{eq_ind_test_statistic}, we consider the following multivariate stationary Fractional Hawkes processes:
\begin{enumerate}
 \item[FH6.] $N=(N_1,N_2)^{\top}$ is a class of bivariate stationary fractional Hawkes processes with $\boldsymbol{\mu} = (0.5, 0.5)^{\top}$,
 \begin{align*}
  \boldsymbol{\nu} = 
\begin{pmatrix}
1/2 & a \\
b & 1/2
\end{pmatrix} ,~
\begin{pmatrix}
c_{11} & c_{12} \\
c_{21} & c_{22}
\end{pmatrix} = 
\begin{pmatrix}
1 & 1.1 \\
0.9 & 1
\end{pmatrix} ,~
\begin{pmatrix}
\beta_{11} & \beta_{12} \\
\beta_{21} & \beta_{22}
\end{pmatrix} = 
\begin{pmatrix}
1 & 0.9 \\
0.8 & 1
\end{pmatrix} ,  
 \end{align*}
where $a,b\geqslant0$, $ab<1/4$. 
\end{enumerate}
We can show that $\boldsymbol{\lambda} = (1/4 - ab)^{-1}(1/4 + a/2, 1/4 + b/2)^{\top}$ and therefore,
\begin{align}\label{eq_f12(0)}
f_{12}(0) = f_{21}(0) = \frac{ab/2 + (a+b)/8}{(1/4 -ab)^3}.
\end{align}
It is easy to see that $f_{12}(0)$ is nonnegative, and  equals $0$ if any only if $a=b=0$. For any fixed $a$ (respectively $b$), $f_{12}(0)$ is a strictly monotonically increasing function of $b$ (respectively $a$). Moreover, the larger $a+b$ is, the more dependent $N_1$ and $N_2$ will be.

We select a set of grid points $a,b=0,0.1,0.2,0.3$, resulting in 16 bivariate stationary Hawkes processes. For each process, we generate 1000 independent copies of $N$ on $(0,T)$, where $T=5000$. The number of Fourier frequencies is chosen to be $M_T=\lfloor 10\sqrt{T}\rfloor$. 
For each realization, we calculate the p-value associated with \eqref{eq_ind_test_statistic} and reject the null hypothesis of independence when the p-value is less than $0.05$. For simplicity, we adopt the ordinary least squares estimates when computing the test statistic (i.e. the kernel function is $K(x)=1$). In Table~\ref{table_FH6} of Appendix~\ref{sec:simstudy_supp}, we report the percentage of realizations for which independence is rejected. Judging from Table~\ref{table_FH6}, when $N_1$ and $N_2$ are indeed independent (i.e.\ $(a,b)=(0,0)$), only 22 out of 1000 realizations lead to rejection, which is reasonably close to the nominal $5\%$ level. When $\{a,b\}=\{0,0.1\}$, the subprocesses are dependent but only weakly so in view of \eqref{eq_f12(0)}, and the test nevertheless rejects independence in over $90\%$ of realizations. For the remaining cases, the dependence signal reflected by \eqref{eq_f12(0)} is strong enough that dependence is detected in every realization.

\section{Discussion and outlook}\label{section_conclusion}
In this paper, we established consistency and asymptotic normality of the Whittle estimator for multivariate stationary Hawkes processes. Our results accommodate heavy tailed kernels, a regime that while important to many applications (e.g. seismology), is often excluded from existing asymptotic theory. We use the reduced cumulant measures to yield frequency-domain asymptotics without relying on strong mixing and moment assumptions as in \citet{Cheysson2022_and_Lang2022}. Finally, we showed that the same cumulant-based approach also leads to a simple and effective frequency-domain test for independence among subprocesses.

Recent years have seen proliferation of nonparametric inference of mutual-excitation kernels of Hawkes processes (\citealp{Bonnet_and_Sangnier2025,Kim_and_Iwata2026}) as well as the spectral density of time series (\citealp{Tang_et_al2026}). Despite the fact that this paper concerns parametric inference of Hawkes processes using Whittle likelihood, we believe the Whittle likelihood can also facilitate nonparametric modeling of Hawkes processes via Bartlett spectral density in terms of both theoretical soundness and computational cost.

Cross-validation is an important way to assess statistical model's ability of prediction and is known to be difficult when there is only one realization of a point process. Recently, \citet{Herrera_and_Cheysson2026} proposed a frequency-domain penalized least-square estimation procedure for Hawkes processes and conduct cross-validation with the help of a thinning mechanism inspired by \citet{Cronie_et_al2024}. We anticipate that the proof techniques used in this paper can be of great help in discovering the asymptotic properties of the said cross-validation procedure. Furthermore, we believe that the Whittle estimation proposed in this paper will certainly benefit from the cross-validation procedure when applying to real-world data.

\section{Acknowledgments} This work was supported by the Royal Society of New Zealand Marsden Fund under grants MFP-UOO2323 and MFP-UOO2518. The authors wish to acknowledge use of the eResearch Infrastructure Platform hosted by the Crown company, Research and Education Advanced Network New Zealand (REANNZ) Ltd., and funded by the Ministry of Business, Innovation \& Employment. URL: https://www.reannz.co.nz



\bibliographystyle{plainnat}
\bibliography{ref}
\newpage

\appendix
\appendixpage
\section{Tables and figures}\label{sec:simstudy_supp}
\FloatBarrier
\begin{figure}[!h]
\includegraphics[width=1\linewidth]{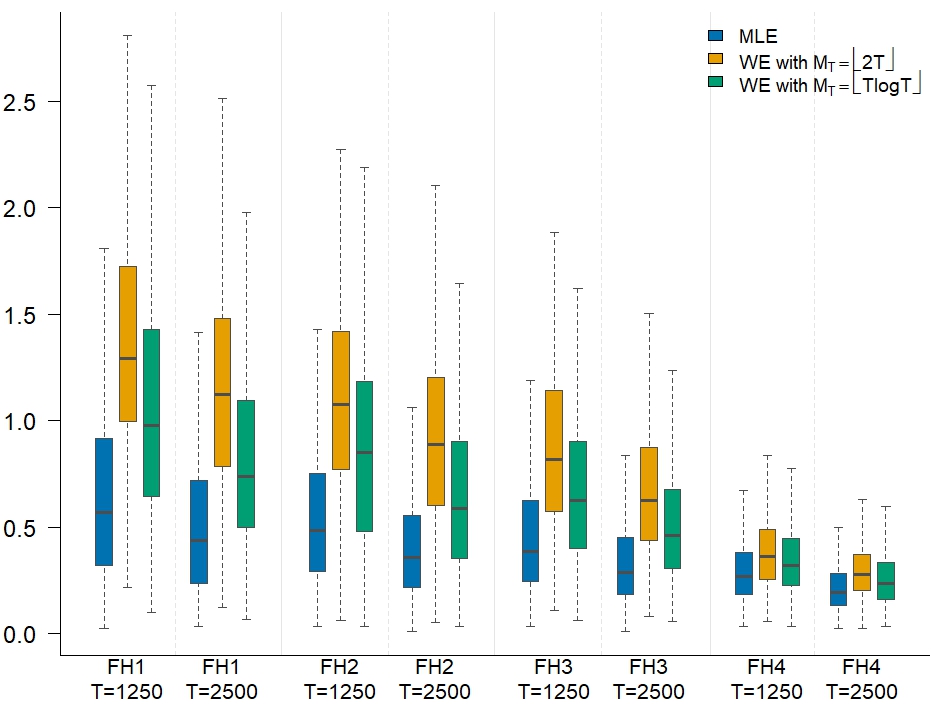}
\caption{The boxplot of relative errors (defined in \eqref{def_relative_error}) of the estimates in Section~\ref{example_univariate_FHP}. For the sake of visualization, the y-axis represents $\log(1 + \mathrm{Relative~error)}$. MLE is indicated by blue, WE with $M_T = \lfloor 2T\rfloor$ is in orange, and WE with $M_T = \lfloor T\log T\rfloor$ is in green.}
\label{fig_comparison_uniFH}
\end{figure}

\begin{table}[!h]
\caption{Median and IQR (in brackets) of relative errors (defined in \eqref{def_relative_error}) of the Whittle and maximum likelihood estimates  in Section~\ref{example_univariate_FHP}.}
\label{table_comparison_uniFH}
\begin{center}
\begin{tabular}{|c c c c|} 
 \hline
 DGP & MLE & WE with $M_T = \lfloor 2T\rfloor$ & WE with $M_T = \lfloor T\log T\rfloor$  \\ [0.5ex] 
 \hline\hline
 FH1. $T=1250$ & 0.7722(1.1229) & 2.6463(2.8904) & 1.6567(2.2685) \\ 
 \hline
 FH1. $T=2500$ & 0.5542(0.7814) & 2.0731(2.1956) & 1.0941(1.3383)\\
 \hline
 \hline
 FH2. $T=1250$ & 0.6287(0.7856) & 1.9359(1.9593) & 1.3410(1.6462)\\
 \hline
 FH2. $T=2500$ & 0.4334(0.4999) & 1.4316(1.5055) & 0.8018(1.0401)\\ 
 \hline
 \hline
 FH3. $T=1250$ & 0.4761(0.5885) & 1.2638(1.3663) & 0.8732(0.9758)\\
 \hline
 FH3. $T=2500$ & 0.3364(0.3676) & 0.8732(0.8504) & 0.5884(0.6160)\\
 \hline
 \hline
 FH4. $T=1250$ &  0.3120(0.2614) & 0.4386(0.3395) & 0.3787(0.3088) \\ 
 \hline
 FH4. $T=2500$ &  0.2175(0.1827) & 0.3227(0.2292) & 0.2671(0.2244)\\
 \hline
\end{tabular}
\end{center}
\end{table}

\begin{table}[!h]
\caption{Median and IQR (in brackets) of runtime (in minutes) of calculating the Whittle and maximum likelihood estimates in Section~\ref{example_univariate_FHP}. For comparison, we use a fixed initial value for optimization procedure across all realizations.}
\label{table_runtime_uniFH}
\begin{center}
\begin{tabular}{|c c c c|} 
 \hline
 DGP & MLE & WE with $M_T = \lfloor 2T\rfloor$ & WE with $M_T = \lfloor T\log T\rfloor$  \\ [0.5ex] 
 \hline\hline
 FH1. $T=1250$ & 3.0165(1.4670) & 0.0117(0.0021) & 0.0383(0.0063) \\ 
 \hline
 FH1. $T=2500$ & 10.9002(5.8283) & 0.0405(0.0039) & 0.1616(0.0141)\\
 \hline
 \hline
 FH2. $T=1250$ & 2.4757(1.2086) & 0.0119(0.0019) & 0.0410(0.0064)\\
 \hline
 FH2. $T=2500$ & 7.9377(3.0711) & 0.0422(0.0036) & 0.1592(0.0120)\\ 
 \hline
 \hline
 FH3. $T=1250$ & 3.4843(1.9421) & 0.0111(0.0017) & 0.0393(0.0045)\\
 \hline
 FH3. $T=2500$ & 9.4609(3.6281) & 0.0399(0.0026) & 0.1524(0.0089)\\
 \hline
 \hline
 FH4. $T=1250$ & 3.6125(1.7914) & 0.0116(0.0015) & 0.0420(0.0041) \\ 
 \hline
 FH4. $T=2500$ & 15.7390(5.2685) & 0.0447(0.0030) & 0.1630(0.0081)\\
 \hline
\end{tabular}
\end{center}
\end{table}

\begin{table}[!h]
\caption{Median and IQR (in brackets) of relative errors (defined in \eqref{def_relative_error}) of the Whittle estimates in Section~\ref{example_bivariate_FHP}.}
\label{table_comparison_biFH}
\begin{center}
\begin{tabular}{|c c c|} 
 \hline
 DGP & WE with $M_T = \lfloor 2T\rfloor$ & WE with $M_T = \lfloor T\log T\rfloor$  \\ [0.5ex] 
 \hline\hline
 FH5. $T=1250$ & 2.0966(0.9756) & 1.7695(0.8302) \\ 
 \hline
 FH5. $T=2500$ & 1.5519(0.7381) & 1.2824(0.6272)\\
 \hline
\end{tabular}
\end{center}
\end{table}
\FloatBarrier
\begin{table}[!h]
\caption{Median and IQR (in brackets) of runtime (in minutes) of calculating the Whittle and maximum likelihood estimates in Section~\ref{example_bivariate_FHP}. For comparison, we use a fixed initial value for the optimization procedure across all realizations.}
\label{table_runtime_biFH}
\begin{center}
\begin{tabular}{|c c c|} 
 \hline
 DGP & WE with $M_T = \lfloor 2T\rfloor$ & WE with $M_T = \lfloor T\log T\rfloor$  \\ [0.5ex] 
 \hline\hline
 FH5. $T=1250$ & 0.0121(0.0012) & 0.0423(0.0056) \\ 
 \hline
 FH5. $T=2500$ & 0.0238(0.0021) & 0.0776(0.0319)\\
 \hline
\end{tabular}
\end{center}
\end{table}

\begin{table}[!h]
\caption{Percentage of realizations in Section~\ref{example_estimation_nu} for which the null hypothesis of independence is rejected at level $0.05$.}
\label{table_FH6}
\begin{center}
\begin{tabular}{|c|c|c|c|c|}
\hline
\diagbox{$a$}{$b$} & 0 & 0.1 & 0.2 & 0.3\\
\hline
 0 & 2.2\% & 92.9\% & 100\% & 100\% \\
 \hline
 0.1 & 96\% & 100\% & 100\% & 100\% \\
 \hline
 0.2 & 100\% & 100\% & 100\% & 100\% \\
 \hline
 0.3 & 100\% & 100\% & 100\% & 100\% \\
 \hline
\end{tabular}
\end{center}
\end{table}

\section{Results concerning the reduced cumulant measures and associated quantities}

\subsection{Proof of Proposition~\ref{HPCum_bound_prop}}\label{cum_density_HP_sect}
For any $s\geqslant2$ and any $j_1,\cdots,j_s\in\{1,\cdots,D\}$, $C_{j_1\cdots j_s}$, and consequently $C_{j_1\cdots j_s}^{\mathrm{red}}$ exist because $\mathrm{E}\exp(a\sum_{j=1}^D N_j(A))<\infty$ for some $a>0$ and any bounded $A\subset\mathbb{R}$ (see Theorem 1 of \citet{Leblanc2024}). The non-negativity of $C_{j_1\cdots j_s}$ and $C_{j_1\cdots j_s}^{\mathrm{red}}$ can be seen from equation (24) of \citet{Jovanovic_et_al2015}. In order to show the finiteness of $C_{j_1\cdots j_s}^{\mathrm{red}}$, we need some notation. Let $k^{j_1\cdots j_s}(x_1,\cdots,x_s)$ be defined in equation (21) of \citet{Jovanovic_et_al2015}. Then we have $\mathrm{d}C_{j_1\cdots j_s}(x_1,\cdots,x_s) = k^{j_1\cdots j_s}(x_1,\cdots,x_s)\mathrm{d}x_1\cdots\mathrm{d}x_s$. Therefore, we only need to show that
\begin{align*}
C_{j_1\cdots j_s}^{\mathrm{red}}(\mathbb{R}^{s-1}) = \int_{\mathbb{R}^{s-1}}k^{j_1\cdots j_s}(x_1,\cdots,x_{s-1},x_s)\mathrm{d}x_1\cdots\mathrm{d}x_{s-1} 
\end{align*}
is finite. Note that the right-hand side of the equation is invariant with respect to $x_s$ due to the stationarity of $N$. The finiteness of the right-hand side can be seen from equation (50) of \citet{Jovanovic_et_al2015}.

To derive a stronger upperbound, we need to delve into the crucial tree representation of the cumulant measures derived in \citet{Jovanovic_et_al2015}. Recall that $\boldsymbol{\nu} = (\nu_{ij})_{i,j=1,\cdots,D}$ happens to be the integral of the excitation matrix over $\mathbb{R}$ (i.e. the matrix $\mathbf{G}$ in equation (13) of \citet{Jovanovic_et_al2015}). Let $\boldsymbol{\Psi} = \sum_{k=1}^{\infty}\boldsymbol{\nu}^k$, $\mathbf{R} = \boldsymbol{\Psi} + \boldsymbol{\mathrm{I}}_{D}$, $R = \max_{1\leqslant i,j\leqslant D} [\mathbf{R}]_{ij}$, $\psi = \max_{1\leqslant i,j\leqslant D} [\boldsymbol{\Psi}]_{ij}$ and $\lambda = \max_{1\leqslant i\leqslant D}[\mathbf{R}\boldsymbol{\mu}]_{i}$, where $\boldsymbol{\mu}$ is the vector of background rates and $\boldsymbol{\mathrm{I}}_{D}$ is the $D$-dimensional identity matrix. The above Neumann series is convergent due to the fact that the spectral radius of $\boldsymbol{\nu}$ is smaller than 1. It is worth noting that $R > \max\{1, \psi\}$. In view of the algorithm at the end of section 3 of \citet{Jovanovic_et_al2015}, we have
$C_{j_1\cdots j_s}^{\mathrm{red}}(\mathbb{R}^{s-1}) = \sum_{i} \int_{\mathbb{R}^{s-1}}k_{i}^{j_1\cdots j_s}(x_1,\cdots,x_s)\mathrm{d}x_1\cdots\mathrm{d}x_{s-1}$, where $k_{i}^{j_1\cdots j_s}$ is generated by steps (2) and (3) of the above-mentioned algorithm and the summation is over all rooted multifurcation trees with $s$ labelled leaves. According to Chapter 3 of \citet{Felsenstein2004}, the number of internal nodes $m_i$ of tree $i$ (with which $k_{i}^{j_1\cdots j_s}$ is associated) ranges from $1$ to $s-1$. Moreover, it can be seen from Section 4 of \citet{Jovanovic_et_al2015} that
\begin{align*}
&\int_{\mathbb{R}^{s-1}}k_{i}^{j_1\cdots j_s}(x_1,\cdots,x_s)\mathrm{d}x_1\cdots\mathrm{d}x_{s-1} \\
&\leqslant \sum_{p_1,\cdots,p_{m_i}=1}^{D} \lambda R^s \psi^{m_i - 1} = \lambda R^s D^{m_i}\psi^{m_i - 1}
\leqslant \lambda R^{s + m_i - 1} D^{m_i}
\leqslant \lambda R^{2s - 2} D^{s-1} .
\end{align*}
Note that the right-hand side of the inequality is independent of $i$. Therefore, we get $C_{j_1\cdots j_s}^{\mathrm{red}}(\mathbb{R}^{s-1}) \leqslant \lambda R^{2s - 2} D^{s-1}T_{s}$, where $T_{s}$ is the number of rooted multifurcation trees with $s$ labelled leaves. Although there is no closed-form formula for $T_s$, it is known (see e.g. entry \texttt{A000311} of On-Line Encyclopedia of Integer Sequences \citep{oeis}) that $T_s \sim s^{s-1} 2^{-1/2}\exp(-s)(2\log 2 - 1)^{-s + 1/2}$ \citep{Kotesovec2014} when $s\rightarrow\infty$. This asymptotic relation, in conjunction with the Stirling's formula $s! \sim \sqrt{2\pi } s^{s+1/2} \exp(-s)$, imply that there exists constant $C_0 > 0$ such that
\begin{align*}
C_{j_1\cdots j_s}^{\mathrm{red}}(\mathbb{R}^{s-1})
  \leqslant C_0^s (s-1)!.
\end{align*}

\subsection{Second-order tail measures}
The first proposition reveals the connection between the moments of mutual-excitation kernels and that of the convolution of the kernels, which plays an important role in the construction of cumulant density of Hawkes processes.
\begin{proposition}\label{alpha_moment_convolution_prop}
Let $\boldsymbol{K}(\cdot) = \boldsymbol{\nu}\odot\boldsymbol{G}(\cdot)$, where $\boldsymbol{\nu} = \{\nu_{ij}\}_{i,j=1,\cdots,D}$ is a $D\times D$ matrix whose entries are all nonnegative, $\boldsymbol{G}(\cdot)=\{g_{ij}(\cdot)\}_{i,j=1,\cdots,D}$ is a $D\times D$ matrix of probability density functions satisfying $\max_{i,j}\int_{\mathbb{R}}|x|^{\alpha}g_{ij}(x)\mathrm{d}x<\infty$ for some $0<\alpha\leqslant 1$, and $\odot$ means the Hadamard (element-wise) product of matrices. Let $\boldsymbol{K}^{*1}(\cdot) = \boldsymbol{K}(\cdot)$, $\boldsymbol{K}^{*(m+1)}(x) = \int_{\mathbb
R}\boldsymbol{K}^{*(m)}(x - y)\boldsymbol{K}(y)\mathrm{d}y$, $k\in\mathbb{N}$. Then for any $m\in\mathbb{N}$ and any $i,j=1,\cdots,D$, $\int_{\mathbb{R}}|x|^{\alpha}[\boldsymbol{K}^{*m}(x)]_{ij}\mathrm{d}x\leqslant m[\boldsymbol{\nu}^m]_{ij}\max_{a,b}\int_{\mathbb{R}}|x|^{\alpha}g_{ab}(x)\mathrm{d}x$.
\begin{proof}
First, note that $\int_{\mathbb{R}}\boldsymbol{K}^{*(m)}(x)\mathrm{d}x = \boldsymbol{\nu}^m$ as the convolution of probability density functions is again a probability density function. We will now prove the proposition by induction.

When $m=1$, the proposition can be verified directly. Suppose now the proposition holds for $m = n$, then when $m=n+1$, since $|u+v|^{\alpha}\leqslant |u|^{\alpha} + |v|^{\alpha}$ for any $u,v\in\mathbb{R}$, we have
\begin{align*}
&\int_{\mathbb{R}}|x|^{\alpha}[\boldsymbol{K}^{*(n+1)}(x)]_{ij}\mathrm{d}x
= \int_{\mathbb{R}}\int_{\mathbb{R}}|x|^{\alpha}[\boldsymbol{K}^{*(n)}(x - y)\boldsymbol{K}(y)]_{ij}\mathrm{d}y\mathrm{d}x \\
&\leqslant \int_{\mathbb{R}}\int_{\mathbb{R}}|z|^{\alpha}[\boldsymbol{K}^{*(n)}(z)\boldsymbol{K}(y)]_{ij}\mathrm{d}y\mathrm{d}z
+ \int_{\mathbb{R}}\int_{\mathbb{R}}|y|^{\alpha}[\boldsymbol{K}^{*(n)}(z)\boldsymbol{K}(y)]_{ij}\mathrm{d}y\mathrm{d}z \\
&=\sum_{k=1}^D \int_{\mathbb{R}}|z|^{\alpha}[\boldsymbol{K}^{*(n)}(z)]_{ik}\mathrm{d}z\int_{\mathbb{R}}[\boldsymbol{K}(y)]_{kj}\mathrm{d}y
+\sum_{k=1}^D \int_{\mathbb{R}}[\boldsymbol{K}^{*(n)}(z)]_{ik}\mathrm{d}z\int_{\mathbb{R}}[|y|^{\alpha}\boldsymbol{K}(y)]_{kj}\mathrm{d}y \\
&\leqslant n\left(\sum_{k=1}^D [\boldsymbol{\nu}^n]_{ik}\nu_{kj}\right)\max_{a,b}\int_{\mathbb{R}}|x|^{\alpha}g_{ab}(x)\mathrm{d}x
+\left(\sum_{k=1}^D[\boldsymbol{\nu}^n]_{ik}\nu_{kj}\right)\max_{a,b}\int_{\mathbb{R}}|x|^{\alpha}g_{ab}(x)\mathrm{d}x\\
&= (n+1) [\boldsymbol{\nu}^{n+1}]_{ij}\max_{a,b}\int_{\mathbb{R}}|x|^{\alpha}g_{ab}(x)\mathrm{d}x
\end{align*}
for any $i,j=1,\cdots,D$. This concludes the proof.
\end{proof}
\end{proposition}

\begin{proposition}
If the mutual-excitation kernels $g_{ij}$, $i,j=1,\cdots,D$ of the Hawkes process satisfy $\max_{i,j}\int_{\mathbb{R}}|x|^{\beta}g_{ij}(x)\mathrm{d}x<\infty$ for some $0<\beta\leqslant 1$, then $\sum_{i,j}\int_{\mathbb{R}}|u|^{\beta}\mathrm{d}C_{ij}^{\mathrm{red}}(u)<\infty$.
\begin{proof}
Let $\boldsymbol{K}(\cdot) = \boldsymbol{\nu}\odot\boldsymbol{G}(\cdot)$, $\boldsymbol{\lambda} = (\lambda_1,\cdots,\lambda_D)^{\top} = (\boldsymbol{\mathrm{I}}_{D} - \boldsymbol{\nu})^{-1}\boldsymbol{\mu}$ 
where $\boldsymbol{\nu}$, $\boldsymbol{\mu}$ and $\boldsymbol{G}(\cdot)$ are defined in Definition \ref{HP_def}. Define $\Psi_t^{ij} = [\sum_{n=1}^{\infty}\boldsymbol{K}^{*n}(t)]_{ij}$ and $R_t^{ij} = \mathds{1}_{\{i=j\}}\delta(t) + \Psi_t^{ij}$, $i,j=1,\cdots,D$, where $\boldsymbol{K}^{*n}(\cdot)$ is the matrix function convolution defined in Proposition~\ref{alpha_moment_convolution_prop} and $\delta(\cdot)$ is the Dirac delta function. It is worth noting that $\int_{\mathbb{R}}\Psi_t^{ij}\mathrm{d}t = \sum_{n=1}^{\infty}[\boldsymbol{\nu}^n]_{ij} = [(\boldsymbol{\mathrm{I}}_{D} - \boldsymbol{\nu})^{-1}\boldsymbol{\nu}]_{ij}$ in view of the proof of Proposition~\ref{alpha_moment_convolution_prop} and the fact that the spectral radius of $\boldsymbol{\nu}$ is smaller than 1. Taking into account equation (37) of \citet{Jovanovic_et_al2015} and $|x+y|^{\beta}\leqslant|x|^{\beta} + |y|^{\beta}$ for any $x,y\in\mathbb{R}$, we have
\begin{align*}
&\int_{\mathbb{R}}|u|^{\beta}\mathrm{d}C_{ij}^{\mathrm{red}}(u)
= \sum_{m=1}^D\lambda_m\int_{\mathbb{R}}\int_{\mathbb{R}}|u|^{\beta} R_{u - v}^{im}R_{ - v}^{mj} \mathrm{d}v \mathrm{d}u\\
&=\sum_{m=1}^D\lambda_m\int_{\mathbb{R}}\int_{\mathbb{R}}|u|^{\beta} \Psi_{u + v}^{im}\Psi_{ v}^{mj} \mathrm{d}v \mathrm{d}u + (\lambda_i+\lambda_j)\int_{\mathbb{R}}|u|^{\beta} \Psi_{u}^{ij} \mathrm{d}u\\
&\leqslant \sum_{m=1}^D\lambda_m\int_{\mathbb{R}}\int_{\mathbb{R}}|x|^{\beta} \Psi_{x}^{im}\Psi_{ y}^{mj} \mathrm{d}x\mathrm{d}y  
+\sum_{m=1}^D\lambda_m\int_{\mathbb{R}}\int_{\mathbb{R}}|y|^{\beta} \Psi_{x}^{im}\Psi_{ y}^{mj} \mathrm{d}x\mathrm{d}y + (\lambda_i+\lambda_j)\int_{\mathbb{R}}|u|^{\beta} \Psi_{u}^{ij} \mathrm{d}u\\
&\leqslant \sum_{m=1}^D\lambda_m [(\boldsymbol{\mathrm{I}}_{D} - \boldsymbol{\nu})^{-1}\boldsymbol{\nu}]_{mj}\int_{\mathbb{R}}|x|^{\beta} \Psi_{x}^{im} \mathrm{d}x + \sum_{m=1}^D \lambda_m [(\boldsymbol{\mathrm{I}}_{D} - \boldsymbol{\nu})^{-1}\boldsymbol{\nu}]_{im}\int_{\mathbb{R}}|y|^{\beta} \Psi_{y}^{mj} \mathrm{d}y + (\lambda_i+\lambda_j)\int_{\mathbb{R}}|u|^{\beta} \Psi_{u}^{ij} \mathrm{d}u
\end{align*}
for any $i,j=1,\cdots,D$. The proof is completed by noting that
\begin{align*}
&\int_{\mathbb{R}}|x|^{\beta} \Psi_{x}^{pq} \mathrm{d}x
=\sum_{l=1}^{\infty} \int_{\mathbb{R}}|x|^{\beta} [\boldsymbol{K}^{* l}(x)]_{pq} \mathrm{d}x
\leqslant \left(\sum_{l=1}^{\infty} l[\boldsymbol{\nu}^l]_{pq}\right)\max_{a,b}\int_{\mathbb{R}}|x|^{\alpha}g_{ab}(x)\mathrm{d}x \\
&=[(\boldsymbol{\mathrm{I}}_{D} - \boldsymbol{\nu})^{-2}\boldsymbol{\nu}]_{pq}\max_{a,b}\int_{\mathbb{R}}|x|^{\alpha}g_{ab}(x)\mathrm{d}x
<\infty
\end{align*}
for any $p,q=1,\cdots,D$ due to Proposition~\ref{alpha_moment_convolution_prop}.
\end{proof}
\end{proposition}

\section{Cumulants of finite Fourier transform and the Whittle likelihood}\label{decomp_periodogram_sect}
This section contains results about the asymptotic behavior of the cumulants of the finite Fourier transform. There are the building blocks of the functional CLT of the spectral empirical process.

\subsection{Decompositions concerning second-order cumulants}
The following function is needed in the decomposition of the second-order cumulants of the finite Fourier transform.
\begin{proposition}\label{prop_R2}
\begin{align*}
&R_{T,2}(u,\bar{\omega}):=\int_{\mathbb{R}}( \mathds{1}_{\{0<\frac{u+v}{T}<1\}} - 1 )\mathds{1}_{\{0<\frac{v}{T}<1\}}\exp(-\mathrm{i}\bar{\omega}v)\mathrm{d}v \\
&= -\mathds{1}_{\{|u|\geqslant T\}} \Delta_T(\bar{\omega})
-\mathds{1}_{\{|u|< T\}}S_{T,2}(u,\bar{\omega}).
\end{align*}
where $\Delta_T(\bar{\omega}) = \int_{0}^T\exp(-\mathrm{i}\bar{
\omega} v) \mathrm{d}v$ is the Dirichlet kernel, 
\begin{align}\label{def_S2_eq}
& S_{T,2}(u,\bar{\omega}) := \mathds{1}_{\{u\geqslant0\}}\int_{T-u}^{T}\exp(-\mathrm{i}\bar{\omega}v)\mathrm{d}v + \mathds{1}_{\{u<0\}}\int_{0}^{-u}\exp(-\mathrm{i}\bar{\omega}v)\mathrm{d}v \\
&=\mathds{1}_{\{\bar{\omega} = 0\}}|u| + \mathds{1}_{\{\bar{\omega} \neq 0\}}\left(\mathds{1}_{\{u\geqslant0\}} \frac{\exp(-\mathrm{i}\bar{\omega}T)(1 - \exp(\mathrm{i}\bar{\omega}u))}{-\mathrm{i}\bar{\omega}} 
+ \mathds{1}_{\{u<0\}}\frac{ 1 - \exp(\mathrm{i}\bar{\omega}u) }{\mathrm{i}\bar{\omega}}\right) \notag
\end{align}
satisfying $|S_{T,2}(u,\bar{\omega})|\leqslant |u|$ for any $u,\bar{\omega}\in\mathbb{R}$.

\begin{proof}
\begin{align*}
&\int_{\mathbb{R}}( \mathds{1}_{\{0<\frac{u+v}{T}<1\}} - 1 )\mathds{1}_{\{0<\frac{v}{T}<1\}}\exp(-\mathrm{i}\bar{\omega}v)\mathrm{d}v
=\int_{0}^T( \mathds{1}_{\{0<\frac{u+v}{T}<1\}} - 1 )\exp(-\mathrm{i}\bar{\omega}v)\mathrm{d}v \\
&=-\mathds{1}_{\{u\geqslant0\}}\int_{\max\{T-u, 0\}}^T\exp(-\mathrm{i}\bar{\omega}v)\mathrm{d}v
-\mathds{1}_{\{u<0\}}\int_{0}^{\min\{-u, T\}}\exp(-\mathrm{i}\bar{\omega}v)\mathrm{d}v.
\end{align*}
When $|u|\geqslant T$, we have
\begin{align*}
&-\mathds{1}_{\{u\geqslant0\}}\int_{\max\{T-u, 0\}}^T\exp(-\mathrm{i}\bar{\omega}v)\mathrm{d}v
-\mathds{1}_{\{u<0\}}\int_{0}^{\min\{-u, T\}}\exp(-\mathrm{i}\bar{\omega}v)\mathrm{d}v \\
&=-\int_{0}^T\exp(-\mathrm{i}\bar{\omega}v)\mathrm{d}v
=-\Delta_T(\bar{\omega}) .
\end{align*}
When $|u|< T$,
\begin{align*}
&-\mathds{1}_{\{u\geqslant0\}}\int_{\max\{T-u, 0\}}^T\exp(-\mathrm{i}\bar{\omega}v)\mathrm{d}v
-\mathds{1}_{\{u<0\}}\int_{0}^{\min\{-u, T\}}\exp(-\mathrm{i}\bar{\omega}v)\mathrm{d}v \\
&=-\mathds{1}_{\{u\geqslant0\}}\int_{T-u}^T\exp(-\mathrm{i}\bar{\omega}v)\mathrm{d}v
-\mathds{1}_{\{u<0\}}\int_{0}^{-u}\exp(-\mathrm{i}\bar{\omega}v)\mathrm{d}v
=-S_{T,2}(u,\bar{\omega}).
\end{align*}
Moreover,
\begin{align*}
&|S_{T,2}(u,\bar{\omega})|
=\left|\mathds{1}_{\{u\geqslant0\}}\int_{T-u}^T\exp(-\mathrm{i}\bar{\omega}v)\mathrm{d}v
+\mathds{1}_{\{u<0\}}\int_{0}^{-u}\exp(-\mathrm{i}\bar{\omega}v)\mathrm{d}v \right| \\
&\leqslant \mathds{1}_{\{u\geqslant0\}}\int_{T-u}^T\mathrm{d}v
+\mathds{1}_{\{u<0\}}\int_{0}^{-u}\mathrm{d}v 
\leqslant |u|
\end{align*}
holds for any $u,\bar{\omega}\in\mathbb{R}$. When $\bar{\omega} = 0$, $S_{T,2}(u,\bar{\omega}) = |u|$. When $\bar{\omega} \neq 0$, straightforward calculations yield Equation~\eqref{def_S2_eq}.
\end{proof}
\end{proposition}

The following proposition introduces a useful decomposition of the second-order cumulants of the finite Fourier transform.
\begin{proposition}\label{cov_J_prop}
For any $\omega_1,~\omega_2\in\mathbb{R}$ and any $j_1,j_2\in\{1,\cdots,D\}$,
\begin{align*}
&\mathrm{cum}(J_T^{j_1}(\omega_{1}), J_T^{j_2}(\omega_{2})) 
=T^{-1}\Delta_T(\omega_1+\omega_2)\left(f_{j_1j_2}(\omega_1) - \int_{|u|\geqslant T} \exp(-\mathrm{i}\omega_1 u)\mathrm{d}C_{j_1j_2}^{\mathrm{red}}(u) \right) \\
&-T^{-1} \int_{|u|< T} S_{T,2}(u,\omega_1+\omega_2) \exp(-\mathrm{i}\omega_1 u)\mathrm{d}C_{j_1j_2}^{\mathrm{red}}(u) ,
\end{align*}
where $\Delta_T(\bar{\omega})$ and $S_{T,2}(u,\bar{\omega})$ are defined in Proposition~\ref{prop_R2}.
\begin{proof}
In view of Definition~\ref{HPCum_def} and Equation~\eqref{cum_def1}, we have by direct calculation
\begin{align*}
&\mathrm{cum}(J_T^{j_1}(\omega_{1}), J_T^{j_2}(\omega_{2})) 
= T^{-1}\int_{0}^T\int_{0}^T \exp(-\mathrm{i}\omega_1 x)\exp(-\mathrm{i}\omega_2 y) \mathrm{d}C_{j_1j_2}(x,y) \\
&= T^{-1}\int_{\mathbb{R}}\int_{\mathbb{R}}\mathds{1}_{\{0<x< T\}}\mathds{1}_{\{0<y< T\}} \exp(-\mathrm{i}\omega_1 x)\exp(-\mathrm{i}\omega_2 y) \mathrm{d}C_{j_1j_2}^{\mathrm{red}}(x-y)\mathrm{d}y \\
&= T^{-1}\int_{\mathbb{R}}\int_{\mathbb{R}}\mathds{1}_{\{0<u+v< T\}}\mathds{1}_{\{0<v< T\}} \exp(-\mathrm{i}\omega_1 u)\exp(-\mathrm{i}(\omega_1+\omega_2) v) \mathrm{d}C_{j_1j_2}^{\mathrm{red}}(u)\mathrm{d}v \\
&=T^{-1}\int_{\mathbb{R}}\exp(-\mathrm{i}\omega_1 u)\int_{\mathbb{R}}\mathds{1}_{\{0<u+v< T\}}\mathds{1}_{\{0<v< T\}} \exp(-\mathrm{i}(\omega_1+\omega_2) v) \mathrm{d}v\mathrm{d}C_{j_1j_2}^{\mathrm{red}}(u) \\
&=T^{-1}\int_{\mathbb{R}}\exp(-\mathrm{i}\omega_1 u)\left(R_{T,2}(u,\omega_1+\omega_2) + \Delta_T(\omega_1+\omega_2)\right)\mathrm{d}C_{j_1j_2}^{\mathrm{red}}(u) \\
&=T^{-1}\Delta_T(\omega_1+\omega_2)\int_{\mathbb{R}}\exp(-\mathrm{i}\omega_1 u)\mathrm{d}C_{j_1j_2}^{\mathrm{red}}(u) +T^{-1}\int_{\mathbb{R}}\exp(-\mathrm{i}\omega_1 u)R_{T,2}(u,\omega_1+\omega_2)\mathrm{d}C_{j_1j_2}^{\mathrm{red}}(u)\\
&=T^{-1} \Delta_T(\omega_1+\omega_2) f_{j_1j_2}(\omega_1) +T^{-1}\int_{\mathbb{R}}\exp(-\mathrm{i}\omega_1 u)R_{T,2}(u,\omega_1+\omega_2)\mathrm{d}C_{j_1j_2}^{\mathrm{red}}(u),
\end{align*}
where $f_{j_1j_2}$ comes from Definition~\ref{HPcsd_def}. Since $R_{T,2}(u,\bar{\omega})= -\mathds{1}_{\{|u|\geqslant T\}} \Delta_T(\bar{\omega})- \mathds{1}_{\{|u|< T\}}S_{T,2}(u,\bar{\omega})$ according to Proposition~\ref{prop_R2}, we get
\begin{align*}
&\int_{\mathbb{R}}\exp(-\mathrm{i}\omega_1 u) R_{T,2}(u,\omega_1+\omega_2)\mathrm{d}C_{j_1j_2}^{\mathrm{red}}(u) \\
&=-\Delta_T(\omega_1+\omega_2)\int_{|u|\geqslant T} \exp(-\mathrm{i}\omega_1 u)\mathrm{d}C_{j_1j_2}^{\mathrm{red}}(u)
- \int_{|u|< T} S_{T,2}(u,\omega_1+\omega_2) \exp(-\mathrm{i}\omega_1 u)\mathrm{d}C_{j_1j_2}^{\mathrm{red}}(u).
\end{align*}
Putting things together, we finally have
\begin{align*}
&\mathrm{cum}(J_T(\omega_{1}), J_T(\omega_{2})) 
=T^{-1} f_{j_1 j_2}(\omega_1) \Delta_T(\omega_1+\omega_2) \\
&-T^{-1}\Delta_T(\omega_1+\omega_2)\int_{|u|\geqslant T} \exp(-\mathrm{i}\omega_1 u)\mathrm{d}C_{j_1j_2}^{\mathrm{red}}(u)
-T^{-1} \int_{|u|< T} S_{T,2}(u) \exp(-\mathrm{i}\omega_1 u)\mathrm{d}C_{j_1j_2}^{\mathrm{red}}(u) \\
&=T^{-1}\Delta_T(\omega_1+\omega_2)\left(f_{j_1 j_2}(\omega_1) -  \int_{|u|\geqslant T} \exp(-\mathrm{i}\omega_1 u)\mathrm{d}C_{j_1j_2}^{\mathrm{red}}(u)\right)
-T^{-1} \int_{|u|< T} S_{T,2}(u) \exp(-\mathrm{i}\omega_1 u)\mathrm{d}C_{j_1j_2}^{\mathrm{red}}(u) .
\end{align*}
\end{proof}
\end{proposition}

\begin{corollary}\label{cov_J_corollary}
For any $s_1,s_2\in\mathbb{Z}$, $j_1,j_2\in\{1,\cdots,D\}$,
\begin{align}\label{cov_J_for_ff_eq1}
&\mathrm{cum}(J_T^{j_1}(2\pi s_1/T), J_T^{j_2}(2\pi s_2/T)) 
=\mathds{1}_{\{s_1 + s_2 = 0\}}\left(f_{j_1 j_2}(2\pi s_1/T) - \int_{|u|\geqslant T} \exp(-\mathrm{i}2\pi s_1 u/T)\mathrm{d}C_{j_1j_2}^{\mathrm{red}}(u) \right) \notag \\
&-T^{-1} \int_{|u|< T} S_{T,2}(u, 2\pi(s_1+s_2)/T) \exp(-\mathrm{i}2\pi s_1 u/T)\mathrm{d}C_{j_1j_2}^{\mathrm{red}}(u) \\
&=\mathds{1}_{\{s_1 + s_2 = 0\}}\left(f_{j_1 j_2}(2\pi s_1/T) - \int_{|u|\geqslant T} \exp(-\mathrm{i}2\pi s_1 u/T)\mathrm{d}C_{j_1j_2}^{\mathrm{red}}(u) - \frac{1}{T}\int_{|u|<T} |u| \exp(-\mathrm{i}2\pi s_1 u/T)\mathrm{d}C_{j_1j_2}^{\mathrm{red}}(u) \right) \notag \\
&- \mathds{1}_{\{s_1 + s_2 \neq 0\}}V_{T}^{j_1 j_2}(s_1,s_2),\label{cov_J_for_ff_eq2}
\end{align} 
where
\begin{align*}
V_{T}^{j_1 j_2}(s_1,s_2) = \frac{1}{T}\int_{|u|<T}\left(-\mathds{1}_{\{u \geqslant 0\}} + \mathds{1}_{\{u < 0\}}\right)\frac{1 - \exp(\mathrm{i}2\pi(s_1+s_2) u /T)}{\mathrm{i}2\pi(s_1+s_2)/T}\mathrm{d}C_{j_1j_2}^{\mathrm{red}}(u) .    
\end{align*}
Furthermore, we have
$\mathds{1}_{\{s_1 + s_2 \neq 0\}}V_{T}^{j_1 j_2}(s_1,s_2) = o(1)$ and
\begin{align*}
\mathrm{cum}(J_T^{j_1}(2\pi s_1/T), J_T^{j_2}(2\pi s_2/T)) 
=\mathds{1}_{\{s_1 + s_2 = 0\}}( f_{j_1 j_2}(2\pi s_1/T) + o(1)) + \mathds{1}_{\{s_1 + s_2 \neq 0\}}o(1).
\end{align*}
All $o$ and $O$ terms are independent of $s_1$, $s_2$.
\begin{proof}
Notice that $\Delta_T(2\pi s/T) = T\mathds{1}_{\{s = 0\}}$ for any $s\in\mathbb{Z}$, then \eqref{cov_J_for_ff_eq1} is a direct consequence of Proposition~\ref{cov_J_prop}. \eqref{cov_J_for_ff_eq2} can be derived from \eqref{cov_J_for_ff_eq1} and \eqref{def_S2_eq} by noting
\begin{align*}
&S_{T,2}(u, 2\pi(s_1+s_2)/T) 
= \mathds{1}_{\{s_1 + s_2 = 0\}} |u| \\
&+  \mathds{1}_{\{s_1 + s_2 \neq 0\}}\left(\mathds{1}_{\{u \geqslant 0\}} \frac{1 - \exp(\mathrm{i}2\pi(s_1+s_2) u /T)}{-\mathrm{i}2\pi(s_1+s_2)/T}
+\mathds{1}_{\{u < 0\}} \frac{1 - \exp(\mathrm{i}2\pi(s_1+s_2) u /T)}{\mathrm{i}2\pi(s_1+s_2)/T}\right) .
\end{align*}
Recall that the cumulant measure $C_{j_1 j_2}$ of the Hawkes process is a finite measure (see Proposition~\ref{HPCum_bound_prop}), we first have
\begin{align*}
\left|\int_{|u|\geqslant T} \exp(-\mathrm{i}2\pi s_1 u/T)\mathrm{d}C_{j_1j_2}^{\mathrm{red}}(u)\right|
\leqslant C_{j_1j_2}^{\mathrm{red}}((-\infty, -T] \cup[T, \infty))
=o(1) .
\end{align*}
Since $|S_{T,2}(u,2\pi(s_1+s_2)/T)|\leqslant |u|$ according to Proposition~\ref{prop_R2}, then for any $u\in \mathbb{R}$, we have $\mathds{1}_{\{|u|/T<1\}}(|u|/T) \leqslant 1$, where the right-hand side is $C_{j_1j_2}^{\mathrm{red}}$-integrable and $\mathds{1}_{\{|u|/T<1\}}(|u|/T) \rightarrow 0$ when $T\rightarrow\infty$. By the dominated convergence theorem, we get
\begin{align*}
&\mathds{1}_{\{s_1 + s_2 \neq 0\}}|V_{T}^{j_1 j_2}(s_1,s_2)|\leqslant\frac{1}{T}\left|\int_{|u|< T} S_{T,2}(u, 2\pi(s_1+s_2)/T) \exp(-\mathrm{i}2\pi s_1 u/T)\mathrm{d}C_{j_1j_2}^{\mathrm{red}}(u)\right| \\
&\leqslant \int_{|u|< T} \frac{|u|}{T} \mathrm{d}C_{j_1j_2}^{\mathrm{red}}(u)
=o(1).
\end{align*}
The proof is finished by noting the above two $o(1)$ bounds and \eqref{cov_J_for_ff_eq2}. 

\end{proof}
\end{corollary}

\begin{remark}
For the sake of readability, there are a few details we didn't explicitly mention in the above proof. The dominated convergence theorem only allows us to show that for arbitrary positive sequence $\{T_n\}_{n\in\mathbb{N}}$ with $T_n\rightarrow\infty$ when $n\rightarrow\infty$, we have $\lim_{n\rightarrow\infty}\int_{|u|< T_n} (|u|/T_n) \mathrm{d}C_{j_1j_2}^{\mathrm{red}}(u)= 0$. To further claim $\lim_{T\rightarrow\infty}\int_{|u|< T} (|u|/T) \mathrm{d}C_{j_1j_2}^{\mathrm{red}}(u)= 0$, we need to invoke Heine's theorem (see e.g.~\citet[Proposition 1, Section 3.2.1]{Zorich2015}).
\end{remark}
    
\begin{corollary}\label{cov_J_corollary2}
If $\sum_{j_1,j_2=1\cdots,D}\int_{\mathbb{R}}|u|^{\beta}\mathrm{d}C_{j_1j_2}^{\mathrm{red}}(u)<\infty$ for some $1/2<\beta\leqslant1$. Then for any $s\in\mathbb{N}$,
\begin{align*}
\mathrm{cum}(J_T^{j_1}(2\pi s/T), J_T^{j_2}(-2\pi s/T))
=f_{j_1 j_2}(2\pi s/T) + o(1/\sqrt{T}),
\end{align*} 
where the little o term is independent of $j_1,j_2$ and $s$. Furthermore, $f_{j_1 j_2}$ is H{\"o}lder continuous satisfying $\max_{j_1, j_2=1,\cdots,D}|f_{j_1 j_2}(u) - f_{j_1 j_2}(v)|\leqslant C|u-v|^{\beta}$ for any $u,v\in\mathbb{R}$ and some $C>0$.

\begin{proof}
According to Corollary~\ref{cov_J_corollary},
\begin{align*}
&\mathrm{cum}(J_T^{j_1}(2\pi s/T), J_T^{j_2}(-2\pi s/T)) \\
&=f_{j_1 j_2}(2\pi s/T) - \int_{|u|\geqslant T} \exp(-\mathrm{i}2\pi s u/T)\mathrm{d}C_{j_1j_2}^{\mathrm{red}}(u) - \frac{1}{T}\int_{|u|<T} |u| \exp(-\mathrm{i}2\pi s u/T)\mathrm{d}C_{j_1j_2}^{\mathrm{red}}(u) .
\end{align*} 
The first claim can be shown by noticing that
\begin{align*}
\int_{|u|\geqslant T} \mathrm{d}C_{j_1j_2}^{\mathrm{red}}(u) 
\leqslant \frac{1}{\sqrt{T}}\int_{|u|\geqslant T} \sqrt{u} \mathrm{d}C_{j_1j_2}^{\mathrm{red}}(u) 
\leqslant \frac{1}{\sqrt{T}}\int_{|u|\geqslant T} \sqrt{u} \sum_{j_1,j_2}\mathrm{d}C_{j_1j_2}^{\mathrm{red}}(u) 
= o(1/\sqrt{T}) ,
\end{align*}
and
\begin{align*}
\frac{1}{\sqrt{T}}\int_{|u|< T} |u| \mathrm{d}C_{j_1j_2}^{\mathrm{red}}(u) 
\leqslant \frac{1}{\sqrt{T}}\int_{|u|< T} |u| \sum_{j_1,j_2}\mathrm{d}C_{j_1j_2}^{\mathrm{red}}(u)
= o(1) ,
\end{align*}
where the second equation is due to the dominated convergence theorem since $\mathds{1}_{\{|u|< T\}}|u|/\sqrt{T} \rightarrow 0$ for any $u\in\mathbb{R}$, $\mathds{1}_{\{|u|< T\}}|u|/\sqrt{T} \leqslant \sqrt{|u|}$ and $\sqrt{|u|}$ is integrable with respect to $\sum_{j_1,j_2}\mathrm{d}C_{j_1j_2}^{\mathrm{red}}(u)$ under the current assumption.

We will now show the H{\"o}lder continuity of $f_{j_1 j_2}$. Recall that $C_{j_1j_2}^{\mathrm{red}}$ are finite measures according to Proposition~\ref{HPCum_bound_prop}, then for any $u, v\in\mathbb{R}$,
\begin{align*}
&|f_{j_1 j_2}(u) - f_{j_1 j_2}(v)|
=\left|\int_{\mathbb{R}}(\cos(x u) - \cos(y u))\mathrm{d}C_{j_1j_2}^{\mathrm{red}}(u) - \mathrm{i}\int_{\mathbb{R}}(\sin(x u) - \sin(y u))\mathrm{d}C_{j_1j_2}^{\mathrm{red}}(u)\right| \\
&\leqslant \int_{\mathbb{R}}|\cos(x u) - \cos(y u)|\mathrm{d}C_{j_1j_2}^{\mathrm{red}}(u) + \int_{\mathbb{R}}|\sin(x u) - \sin(y u)|\mathrm{d}C_{j_1j_2}^{\mathrm{red}}(u) \\
&=2\int_{\mathbb{R}}\left|\sin\left(\frac{x+y}{2}u\right)\sin\left(\frac{x-y}{2}u\right)\right|\mathrm{d}C_{j_1j_2}^{\mathrm{red}}(u) + 2\int_{\mathbb{R}}\left|\cos\left(\frac{x+y}{2}u\right)\sin\left(\frac{x-y}{2}u\right)\right|\mathrm{d}C_{j_1j_2}^{\mathrm{red}}(u) \\
&\leqslant 2^{2-\beta} |x - y|^{\beta} \int_{\mathbb{R}}|u|^{\beta}\mathrm{d}C_{j_1j_2}^{\mathrm{red}}(u) 
\leqslant C |x - y|^{\beta},
\end{align*}
where the last inequality is due to $|\sin x|\leqslant\min\{1,|x|^{\beta}\}$ for any $x\in\mathbb{R}$ and $C=2^{2-\beta}\max_{j_1,j_2}\int_{\mathbb{R}}|u|^{\beta}\mathrm{d}C_{j_1j_2}^{\mathrm{red}}(u)$.
\end{proof}
\end{corollary}

\subsection{Decompositions concerning higher order cumulants}
To prove the asymptotic normality, we need information about the cumulants of all orders of the finite Fourier transform. To this end, we will decompose the higher-order cumulants in a manner similar to the previous section.
\begin{proposition}\label{cum_J_prop}
Let $l\geqslant2$. For any $\omega_1,\cdots,\omega_{l-1}\in\mathbb{R}$, and any $k_1,\cdots,k_l\in\{1,\cdots,D\}$,
\begin{align*}
&\mathrm{cum}(J_T^{k_j}(\omega_j);j=1,\cdots,l) \\
&=T^{-\frac{l}{2}}\left(f_{k_1 \cdots k_l}(\omega_1,\cdots,\omega_{l-1}) -\int_{\cup_{j=1}^{l-1}\{|y_j|>T\}}\exp\left(-\mathrm{i}\sum_{j=1}^{l-1}\omega_j y_j\right) \mathrm{d}C_{k_1\cdots k_l}^{\mathrm{red}}(y_1, \cdots, y_{l-1}) \right)\Delta_T\left(\sum_{j=1}^l\omega_j\right) \\
&- T^{-\frac{l}{2}}\int_{[-T,T]^{l-1}} \exp\left(-\mathrm{i}\sum_{j=1}^{l-1}\omega_j y_j\right) S_{T,l}(y_1,\cdots,y_{l-1};\sum_{j=1}^l\omega_j)\mathrm{d}C_{k_1\cdots k_l}^{\mathrm{red}}(y_1, \cdots, y_{l-1}), 
\end{align*}
where 
\begin{align*}
S_{T,l}(y_1,\cdots,y_{l-1};\bar{\omega}) = \mathds{1}_{\cap_{j=1}^{l-1} \{|y_j| \leqslant T\}}\int_0^T \mathds{1}_{\cup_{j=1}^{l-1} \{v \leqslant -y_j\}\cup\{v \geqslant T - y_j\}}  \exp\left(-\mathrm{i}\bar{\omega}v \right)\mathrm{d}v 
\end{align*}
satisfies $S_{T,l}(y_1,\cdots,y_{l-1};\bar{\omega})\leqslant\mathds{1}_{\cap_{j=1}^{l-1} \{|y_j| \leqslant T\}}\sum_{j=1}^{l-1} |y_j|$ . Moreover, we can derive the following expression:
\begin{align*}
&S_{T,l}(y_1,\cdots,y_{l-1};\bar{\omega}) \\
& = \left[\left(\max_{j:y_j<0}|y_j| + \max_{j:y_j\geqslant0}|y_j|\right)\mathds{1}_{\{\max_{j:y_j<0}|y_j| + \max_{j:y_j\geqslant0}|y_j|\leqslant T\}} \right. \\
& \left. + T\mathds{1}_{\cap_{j=1}^{l-1} \{|y_j| \leqslant T\}\cap\{\max_{j:y_j<0}|y_j| + \max_{j:y_j\geqslant0}|y_j|> T\}}\right]\mathds{1}_{\{\bar{\omega} = 0\}} \\
& + \left[\frac{1 - \exp(-\mathrm{i}\bar{\omega}T) }{\mathrm{i}\bar{\omega}}\mathds{1}_{\cap_{j=1}^{l-1} \{|y_j| \leqslant T\}\cap\{\max_{j:y_j<0}|y_j| + \max_{j:y_j\geqslant0}|y_j|> T\}} \right. \\
& \left. \left(\frac{1 - \exp(-\mathrm{i}\bar{\omega}\max_{j:y_j<0}|y_j|) }{\mathrm{i}\bar{\omega}} +  \exp(-\mathrm{i}\bar{\omega}T)\frac{\exp(\mathrm{i}\bar{\omega}\max_{j:y_j\geqslant0}|y_j|) - 1 }{\mathrm{i}\bar{\omega}}\right)\mathds{1}_{\{\max_{j:y_j<0}|y_j| + \max_{j:y_j\geqslant0}|y_j|\leqslant T\}} \right]\mathds{1}_{\{\bar{\omega} \neq 0\}} .
\end{align*}
In the expressions above, we adopt the convention that $\max_{j\in\emptyset}a_j = 0$.
\begin{proof}
\begin{align*}
&\mathrm{cum}(J_T^{k_j}(\omega_j);j=1,\cdots,l)
=T^{-\frac{l}{2}}\int_{(0,T)^l} \exp\left(-\mathrm{i}\sum_{j=1}^l\omega_j x_j\right) \mathrm{d}C_{k_1\cdots k_l}^{\mathrm{red}}(x_1 - x_l, \cdots, x_{l-1} - x_{l})\mathrm{d}x_l \\
&=T^{-\frac{l}{2}}\int_{\mathbb{R}^l} \mathds{1}_{\cap_{j=1}^{l-1} \{0< y_j + y_l < T\}} \mathds{1}_{ \{0< y_l < T\}} \exp\left(-\mathrm{i}\sum_{j=1}^{l-1}\omega_j y_j\right) \exp\left(-\mathrm{i}y_l\sum_{j=1}^l\omega_j \right) \mathrm{d}C_{k_1\cdots k_l}^{\mathrm{red}}(y_1, \cdots, y_{l-1})\mathrm{d}y_l \\
&=T^{-\frac{l}{2}}\int_{\mathbb{R}^{l-1}} \exp\left(-\mathrm{i}\sum_{j=1}^{l-1}\omega_j y_j\right) \mathrm{d}C_{k_1\cdots k_l}^{\mathrm{red}}(y_1, \cdots, y_{l-1})\mathrm{d}y_1\cdots\mathrm{d}y_{l-1} \Delta_T\left(\sum_{j=1}^l\omega_j\right) \\
&+T^{-\frac{l}{2}}\int_{\mathbb{R}^{l-1}} \exp\left(-\mathrm{i}\sum_{j=1}^{l-1}\omega_j y_j\right) R_{T,l}(y_1,\cdots,y_{l-1}; \sum_{j=1}^l\omega_j)\mathrm{d}C_{k_1\cdots k_l}^{\mathrm{red}}(y_1, \cdots, y_{l-1}) \\
&=T^{-\frac{l}{2}}f_{k_1\cdots k_l}(\omega_1,\cdots,\omega_{l-1})\Delta_T\left(\sum_{j=1}^l\omega_j\right) \\
&+ T^{-\frac{l}{2}}\int_{\mathbb{R}^{l-1}} \exp\left(-\mathrm{i}\sum_{j=1}^{l-1}\omega_j y_j\right) R_{T,l}(y_1,\cdots,y_{l-1}; \sum_{j=1}^l\omega_j)\mathrm{d}C_{k_1\cdots k_l}^{\mathrm{red}}(y_1, \cdots, y_{l-1}),
\end{align*}
where
\begin{align*}
R_{T,l}(y_1,\cdots,y_{l-1}; \bar{\omega})
=\int_0^T \left(\mathds{1}_{\cap_{j=1}^{l-1} \{0< y_j + v < T\}} - 1 \right)\exp\left(-\mathrm{i}\bar{\omega}v \right)\mathrm{d}v .
\end{align*}
If $\max_{j=1,\cdots,l-1}|y_j|>T$, then $\mathds{1}_{\cap_{j=1}^{l-1} \{0< y_j + v < T\}}=0$ for any $v\in(0,T)$. Therefore,
\begin{align*}
&R_{T,l}(y_1,\cdots,y_{l-1}; \bar{\omega})
=-\mathds{1}_{\cup_{j=1}^{l-1}\{|y_j|>T\}}\int_0^T \exp\left(-\mathrm{i}\bar{\omega}v \right)\mathrm{d}v 
-\mathds{1}_{\cap_{j=1}^{l-1} \{|y_j| \leqslant T\}}S_{T,l}(y_1,\cdots,y_{l-1};\bar{\omega}) \\
&=-\mathds{1}_{\cup_{j=1}^{l-1}\{|y_j|>T\}}\Delta_T\left(\bar{\omega}\right)
-\mathds{1}_{\cap_{j=1}^{l-1} \{|y_j| \leqslant T\}}S_{T,l}(y_1,\cdots,y_{l-1};\bar{\omega}) ,
\end{align*}
where
\begin{align*}
S_{T,l}(y_1,\cdots,y_{l-1};\bar{\omega}) = \mathds{1}_{\cap_{j=1}^{l-1} \{|y_j| \leqslant T\}}\int_0^T \mathds{1}_{\cup_{j=1}^{l-1} \{v \leqslant -y_j\}\cup\{v \geqslant T - y_j\}}  \exp\left(-\mathrm{i}\bar{\omega}v \right)\mathrm{d}v .
\end{align*}
Under the condition $\max_{j=1,\cdots,l-1}|y_j|\leqslant T$, we have
\begin{align*}
&\cup_{j=1}^{l-1} \{v \leqslant -y_j\}\cap\{0<v<T\} = \cup_{j: y_j<0}\{0<v\leqslant-y_j\} = \{0<v\leqslant\max_{j:y_j<0}\{-y_j\}\} \\
&=\{0<v\leqslant\max_{j:y_j<0}|y_j|\}
\end{align*}
and
\begin{align*}
&\cup_{j=1}^{l-1} \{v \geqslant T-y_j\}\cap\{0<v<T\} = \cup_{j: y_j\geqslant0}\{T-y_j\leqslant v<T\} = \{\min_{j:y_j\geqslant0}\{T-y_j\}\leqslant v<T\} \\
&=\{T-\max_{j:y_j\geqslant0}|y_j|\leqslant v<T\} .
\end{align*}
In the expressions above, we adopt the convention that $\cup_{j\in\emptyset}A_j = \emptyset$. Therefore,
\begin{align*}
&|S_{T,l}(y_1,\cdots,y_{l-1};\bar{\omega})|\leqslant \mathds{1}_{\cap_{j=1}^{l-1} \{|y_j| \leqslant T\}}\int_0^T \left(\mathds{1}_{\{0<v\leqslant\max_{j:y_j<0}|y_j|\}} +  \mathds{1}_{\{T-\max_{j:y_j\geqslant0}|y_j|\leqslant v<T\}}\right)\mathrm{d}v \\
&=\mathds{1}_{\cap_{j=1}^{l-1} \{|y_j| \leqslant T\}} \left(\max_{j:y_j<0}|y_j| + \max_{j:y_j\geqslant0}|y_j|\right)
\leqslant \mathds{1}_{\cap_{j=1}^{l-1} \{|y_j| \leqslant T\}}\sum_{j=1}^{l-1} |y_j|.
\end{align*}

We now further decompose $S_{T,l}$. When $\max_{j:y_j<0}|y_j| + \max_{j:y_j\geqslant0}|y_j|> T$, we have 
\begin{align*}
\int_0^T \mathds{1}_{\cup_{j=1}^{l-1} \{v \leqslant -y_j\}\cup\{v \geqslant T - y_j\}}  \exp\left(-\mathrm{i}\bar{\omega}v \right)\mathrm{d}v
=\int_0^T \exp\left(-\mathrm{i}\bar{\omega}v \right)\mathrm{d}v
=T\mathds{1}_{\{\bar{\omega} = 0\}} + \frac{1 - \exp(-\mathrm{i}\bar{\omega}T)}{\mathrm{i}\bar{\omega}}\mathds{1}_{\{\bar{\omega} \neq 0\}} .
\end{align*}
When $\max_{j:y_j<0}|y_j| + \max_{j:y_j\geqslant0}|y_j|\leqslant T$,
\begin{align*}
&\int_0^T \mathds{1}_{\cup_{j=1}^{l-1} \{v \leqslant -y_j\}\cup\{v \geqslant T - y_j\}}  \exp\left(-\mathrm{i}\bar{\omega}v \right)\mathrm{d}v
=\int_0^{\max_{j:y_j<0}|y_j|} \exp\left(-\mathrm{i}\bar{\omega}v \right)\mathrm{d}v
+\int_{T - \max_{j:y_j\geqslant0}|y_j|}^{T} \exp\left(-\mathrm{i}\bar{\omega}v \right)\mathrm{d}v \\
&=\left(\max_{j:y_j<0}|y_j| + \max_{j:y_j\geqslant0}|y_j|\right)\mathds{1}_{\{\bar{\omega} = 0\}} \\
&+ \left(\frac{1 - \exp(-\mathrm{i}\bar{\omega}\max_{j:y_j<0}|y_j|)}{\mathrm{i}\bar{\omega}} +  \exp(-\mathrm{i}\bar{\omega}T)\frac{\exp(\mathrm{i}\bar{\omega}\max_{j:y_j\geqslant0}|y_j|) - 1 }{\mathrm{i}\bar{\omega}}\right)\mathds{1}_{\{\bar{\omega} \neq 0\}} .
\end{align*}

\end{proof}
\end{proposition}

\begin{corollary}\label{cum_J_corollary}
Let $\omega_i = 2\pi i/T$, where $i\in Q$, $Q=[i_1,\cdots,i_{l}]$ is a multi-subset of $\mathbb{Z}$ (i.e. the elements of $Q$ are integers and we allow multiple instances) with cardinality (the sum of the multiplicities of all elements) $|Q|=l\geqslant2$. Then for any $k_1,\cdots,k_l\in\{1,\cdots,D\}$,
\begin{align*}
&\mathrm{cum}(J_T^{k_j}(\omega_{i_j});j=1,\cdots,l) \\
&= T^{1-\frac{l}{2}}\left[f_{k_1\cdots k_l}(\omega_{i_1},\cdots,\omega_{i_{l-1}}) -\int_{\{\max_{j:y_j<0}|y_j| + \max_{j:y_j\geqslant0}|y_j|> T\}}\exp\left(-\mathrm{i}\sum_{j=1}^{l-1}\omega_{i_j} y_j\right) \mathrm{d}C_{k_1\cdots k_l}^{\mathrm{red}}(y_1, \cdots, y_{l-1}) \right. \\
&\left. -\frac{1}{T}\int_{\{\max_{j:y_j<0}|y_j| + \max_{j:y_j\geqslant0}|y_j|\leqslant T\}}\left(\max_{j:y_j<0}|y_j| + \max_{j:y_j\geqslant0}|y_j|\right)\exp\left(-\mathrm{i}\sum_{j=1}^{l-1}\omega_{i_j} y_j\right) \mathrm{d}C_{k_1\cdots k_l}^{\mathrm{red}}(y_1, \cdots, y_{l-1})\right]\mathds{1}_{\{\sum_{i\in Q}i = 0\}} \\
&+ T^{-\frac{l}{2}} U_{T,l}^{k_1\cdots k_l}\left(Q\right)\mathds{1}_{\{\sum_{i\in Q}i \neq 0\}}
\end{align*}
with
\begin{align*}
&U_{T,l}^{k_1\cdots k_l}\left(Q\right) \\
&=\int_{\{\max_{j:y_j<0}|y_j| + \max_{j:y_j\geqslant0}|y_j|\leqslant T\}} \left\{\frac{\exp\left[\mathrm{i}\left(\sum_{i\in Q}\omega_i\right)\left(\max_{j:y_j\geqslant0}|y_j|\right)\right] - \exp\left[-\mathrm{i}\left(\sum_{i\in Q}\omega_i\right)\left(\max_{j:y_j<0}|y_j|\right)\right] }{\mathrm{i}\sum_{i\in Q}\omega_i} \right.\\
&\left. \exp\left(-\mathrm{i}\sum_{j=1}^{l-1}\omega_{i_j} y_j\right) \right\} \mathrm{d}C_{k_1\cdots k_l}^{\mathrm{red}}(y_1, \cdots, y_{l-1}) .
\end{align*}
Moreover,
\begin{enumerate}[(a)]
\item We have
\begin{align*}
\left|\int_{\{\max_{j:y_j<0}|y_j| + \max_{j:y_j\geqslant0}|y_j|> T\}}\exp\left(-\mathrm{i}\sum_{j=1}^{l-1}\omega_{i_j} y_j\right) \mathrm{d}C_{k_1\cdots k_l}^{\mathrm{red}}(y_1, \cdots, y_{l-1}) \right|
=o(1),
\end{align*}
\begin{align*}
\frac{1}{T}\left|\int_{\{\max_{j:y_j<0}|y_j| + \max_{j:y_j\geqslant0}|y_j|\leqslant T\}}\left(\max_{j:y_j<0}|y_j| + \max_{j:y_j\geqslant0}|y_j|\right)\exp\left(-\mathrm{i}\sum_{j=1}^{l-1}\omega_{i_j} y_j\right) \mathrm{d}C_{k_1\cdots k_l}^{\mathrm{red}}(y_1, \cdots, y_{l-1})\right| = o(1)
\end{align*}
and
\begin{align*}
U_{T,l}^{k_1\cdots k_l}\left(Q\right) = o(T) 
\end{align*}
when $T\rightarrow\infty$. Consequently,
\begin{align*}
&\mathrm{cum}(J_T^{k_j}(\omega_{i_j});j=1,\cdots,l) 
= T^{1-\frac{l}{2}}\left[\left(f_{k_1\cdots k_l}(\omega_{i_1},\cdots,\omega_{i_{l-1}}) +o(1) 
\right)\mathds{1}_{\{\sum_{i\in Q}i = 0\}} 
+ o(1)\mathds{1}_{\{\sum_{i\in Q}i \neq 0\}} \right] \\
&=T^{1-\frac{l}{2}}\left[O(1) 
\mathds{1}_{\{\sum_{i\in Q}i = 0\}} 
+ o(1)\mathds{1}_{\{\sum_{i\in Q}i \neq 0\}} \right],
\end{align*}
where the big O and little o terms might depend on $l$ but do not depend on $Q$ or $k_1,\cdots,k_l$.
\item Also,
\begin{align*}
&\left|f_{k_1\cdots k_l}(\omega_{i_1},\cdots,\omega_{i_{l-1}}) -\int_{\{\max_{j:y_j<0}|y_j| + \max_{j:y_j\geqslant0}|y_j|> T\}}\exp\left(-\mathrm{i}\sum_{j=1}^{l-1}\omega_{i_j} y_j\right) \mathrm{d}C_{k_1\cdots k_l}^{\mathrm{red}}(y_1, \cdots, y_{l-1})\right. \\
& \left. -\frac{1}{T}\int_{\{\max_{j:y_j<0}|y_j| + \max_{j:y_j\geqslant0}|y_j|\leqslant T\}}\left(\max_{j:y_j<0}|y_j| + \max_{j:y_j\geqslant0}|y_j|\right)\exp\left(-\mathrm{i}\sum_{j=1}^{l-1}\omega_{i_j} y_j\right) \mathrm{d}C_{k_1\cdots k_l}^{\mathrm{red}}(y_1, \cdots, y_{l-1}) \right| \\
&\leqslant C_{k_1\cdots k_l}^{\mathrm{red}}\left(\mathbb{R}^{l-1}\right)  , 
\end{align*}
and
\begin{align*}
|U_{T,l}^{k_1\cdots k_l}\left(Q\right)|
\leqslant T\frac{C_{k_1\cdots k_l}^{\mathrm{red}}\left(\mathbb{R}^{l-1}\right)}{\pi} \frac{1}{|\sum_{j\in Q} j|} .
\end{align*}
As a result,
\begin{align*}
|\mathrm{cum}(J_T^{k_j}(\omega_{i_j});j=1,\cdots,l)|
\leqslant T^{1 - \frac{l}{2}}C_{k_1\cdots k_l}^{\mathrm{red}}\left(\mathbb{R}^{l-1}\right) \left(\mathds{1}_{\{\sum_{i\in Q}i = 0\}} 
+ \frac{1}{|\sum_{j\in Q} j|}\mathds{1}_{\{\sum_{i\in Q}i \neq 0\}} \right)
\end{align*}

\end{enumerate}

\begin{proof}
First, for Fourier frequencies $\omega_j = 2\pi j/T$, $j\in Q$, we have $\Delta_T(\sum_{j\in Q}\omega_j) = T\mathds{1}_{\{\sum_{j\in Q}\omega_j = 0\}} = T\mathds{1}_{\{\sum_{j\in Q}j = 0\}}$ and $\exp(-\mathrm{i}T\sum_{j\in Q}\omega_j) = 1$. Moreover, since $\cup_{j=1}^{l-1}\{|y_j|>T\} \subset \{\max_{j:y_j<0}|y_j| + \max_{j:y_j\geqslant0}|y_j|> T\}$, we have $\cup_{j=1}^{l-1}\{|y_j|>T\} \cup \left(\cap_{j=1}^{l-1}\{|y_j|\leqslant T\}\cap\{\max_{j:y_j<0}|y_j| + \max_{j:y_j\geqslant0}|y_j|> T\}\right)=\{\max_{j:y_j<0}|y_j| + \max_{j:y_j\geqslant0}|y_j|> T\}$. Then the first equation is a direct consequence of Proposition~\ref{cum_J_prop}. 

Regarding part (a), notice that 
\begin{align*}
&\left|\int_{\{\max_{j:y_j<0}|y_j| + \max_{j:y_j\geqslant0}|y_j|> T\}}\exp\left(-\mathrm{i}\sum_{j=1}^{l-1}\omega_{i_j} y_j\right) C_{k_1\cdots k_l}^{\mathrm{red}}(y_1,\cdots,y_{l-1}) \right| \\
&\leqslant C_{k_1\cdots k_l}^{\mathrm{red}}\left(\left\{\max_{j:y_j<0}|y_j| + \max_{j:y_j\geqslant0}|y_j|> T\right\}\right) 
\leqslant C_{k_1\cdots k_l}^{\mathrm{red}}\left(\cup_{j=1}^{l-1}\left\{|y_j|>\frac{T}{2}\right\}\right)
=o(1),
\end{align*}
where the last equation is due to the fact that $C_{k_1\cdots k_l}^{\mathrm{red}}$ is a finite measure (see Proposition~\ref{HPCum_bound_prop}). The second equation of part (a) can be proven by noticing 
\begin{align*}
&\frac{1}{T}\left|\int_{\{\max_{j:y_j<0}|y_j| + \max_{j:y_j\geqslant0}|y_j|\leqslant T\}}\left(\max_{j:y_j<0}|y_j| + \max_{j:y_j\geqslant0}|y_j|\right)\exp\left(-\mathrm{i}\sum_{j=1}^{l-1}\omega_{i_j} y_j\right) \mathrm{d}C_{k_1\cdots k_l}^{\mathrm{red}}(y_1, \cdots, y_{l-1})\right| \\
&\leqslant \frac{1}{T}\int_{\{\max_{j:y_j<0}|y_j| + \max_{j:y_j\geqslant0}|y_j|\leqslant T\}}\left(\max_{j:y_j<0}|y_j| + \max_{j:y_j\geqslant0}|y_j|\right)\mathrm{d}C_{k_1\cdots k_l}^{\mathrm{red}}(y_1, \cdots, y_{l-1})    
\end{align*}
and applying the dominated convergence theorem in the same manner as the proof of Corollary~\ref{cov_J_corollary} (a). The third equation can be derived similarly by noting
\begin{align*}
|U_{T,l}^{k_1\cdots k_l}\left(Q\right)|
\leqslant\int_{\{\max_{j:y_j<0}|y_j| + \max_{j:y_j\geqslant0}|y_j|\leqslant T\}}\left(\max_{j:y_j<0}|y_j| + \max_{j:y_j\geqslant0}|y_j|\right) \mathrm{d}C_{k_1\cdots k_l}^{\mathrm{red}}(y_1, \cdots, y_{l-1}).
\end{align*}

For part (b), by direct calculations,
\begin{align*}
&\left|f_{k_1\cdots k_l}(\omega_{i_1},\cdots,\omega_{i_{l-1}}) -\int_{\{\max_{j:y_j<0}|y_j| + \max_{j:y_j\geqslant0}|y_j|> T\}}\exp\left(-\mathrm{i}\sum_{j=1}^{l-1}\omega_{i_j} y_j\right) \mathrm{d}C_{k_1\cdots k_l}^{\mathrm{red}}(y_1, \cdots, y_{l-1})\right. \\
& -\left.\frac{1}{T}\int_{\{\max_{j:y_j<0}|y_j| + \max_{j:y_j\geqslant0}|y_j|\leqslant T\}}\left(\max_{j:y_j<0}|y_j| + \max_{j:y_j\geqslant0}|y_j|\right)\exp\left(-\mathrm{i}\sum_{j=1}^{l-1}\omega_{i_j} y_j\right) \mathrm{d}C_{k_1\cdots k_l}^{\mathrm{red}}(y_1, \cdots, y_{l-1}) \right| \\
&= \left|\int_{\{\max_{j:y_j<0}|y_j| + \max_{j:y_j\geqslant0}|y_j|\leqslant T\}} \left[ 1 - \frac{1}{T}\left( \max_{j:y_j<0}|y_j| + \max_{j:y_j\geqslant0}|y_j|\right)\right]\exp\left(-\mathrm{i}\sum_{j=1}^{l-1}\omega_{i_j} y_j\right) \mathrm{d}C_{k_1\cdots k_l}^{\mathrm{red}}(y_1, \cdots, y_{l-1}) \right| \\
&\leqslant C_{k_1\cdots k_l}^{\mathrm{red}}\left(\left\{\max_{j:y_j<0}|y_j| + \max_{j:y_j\geqslant0}|y_j|\leqslant T\right\}\right)
\leqslant C_{k_1\cdots k_l}^{\mathrm{red}}\left(\mathbb{R}^{l-1}\right)   
\end{align*}
and
\begin{align*}
&|U_{T,l}^{k_1\cdots k_l}\left(Q\right)|
\leqslant 2C_{k_1\cdots k_l}^{\mathrm{red}}\left(\left\{\max_{j:y_j<0}|y_j| + \max_{j:y_j\geqslant0}|y_j|\leqslant T\right\}\right) \frac{1}{|\sum_{i\in Q}\omega_i|} \\
&=T\frac{1}{\pi}C_{k_1\cdots k_l}^{\mathrm{red}}\left(\left\{\max_{j:y_j<0}|y_j| + \max_{j:y_j\geqslant0}|y_j|\leqslant T\right\}\right) \frac{1}{|\sum_{i\in Q}i|}
\leqslant T \frac{C_{k_1\cdots k_l}^{\mathrm{red}}\left(\mathbb{R}^{l-1}\right)}{\pi} \frac{1}{|\sum_{i\in Q}i|} .
\end{align*}
\end{proof}
\end{corollary}

\section{Results concerning deterministic terms}\label{deterministic_sect}
This section contains inequalities and results about the convergence (rate) of sequences of deterministic functions.
\begin{proposition}\label{proposition_log_trace_ineq}
Suppose $\boldsymbol{A}$ is a $D\times D$ matrix whose eigenvalues are all positive. Then we have
\begin{equation*}
\mathrm{tr}(\boldsymbol{A}) - \log\det\boldsymbol{A} \geqslant D,
\end{equation*}
where the equality holds if and only if the eigenvalues of $\boldsymbol{A}$ are all $1$.
\begin{proof}
Let $\lambda_1,\cdots,\lambda_{D}$ be the eigenvalues of $\boldsymbol{A}$. Then by AM-GM inequality,
\begin{align*}
\frac{1}{D}\mathrm{tr}(\boldsymbol{A})
= \frac{1}{D}\sum_{j=1}^D \lambda_j 
\geqslant \left(\prod_{j=1}^D \lambda_j\right)^{\frac{1}{D}}
=(\det\boldsymbol{A})^{\frac{1}{D}}.
\end{align*}
Combining the above inequality and $x-1\geqslant\log x$ for any $x>0$, we therefore have
\begin{align*}
\frac{1}{D}\mathrm{tr}(\boldsymbol{A}) - 1
\geqslant \log\left(\frac{1}{D}\mathrm{tr}(\boldsymbol{A})\right)
\geqslant \log\left[(\det\boldsymbol{A})^{\frac{1}{D}}\right]
=\frac{1}{D}\log\det\boldsymbol{A} .
\end{align*}
The target inequality can then be obtained by simple re-arrangements. For the equality to hold, $\boldsymbol{A}$ should have identical eigenvalues and at the same time satisfies $\mathrm{tr}(\boldsymbol{A})/D = 1$. These two conditions together lead to $\lambda_1=\cdots=\lambda_D=1$.

\end{proof}
\end{proposition}

\subsection{Convergence of Riemann integral and expected empirical process}\label{consistency_sect}
This subsection is devoted to proving the convergence of deterministic terms including the expectation of the spectral empirical process and other related quantities.

The following proposition describes the convergence of Riemann sum to Riemann integral.
\begin{proposition}\label{conv_Riemann_sum_prop}
Let $\phi:\mathbb{R}\rightarrow\mathbb{C}$ be uniformly continuous. For any $\epsilon>0$, there exists $t_0>0$ such that for any $T>t_0$, we have
\begin{align*}
\left|\frac{1}{T}\sum_{p=1}^{M_T} \phi(\omega_p) - \frac{1}{2\pi}\int_0^{2\pi\frac{M_T}{T}}\phi(x)\mathrm{d}x \right|
\leqslant\frac{1}{2\pi}\frac{M_T}{T}\epsilon ,
\end{align*}
where $\omega_p = 2\pi p/T$ are the Fourier frequencies.
\begin{proof}
\begin{align*}
\frac{1}{T}\sum_{p=1}^{M_T} \phi(\omega_p)
=\frac{1}{2\pi}\sum_{p=1}^{M_T} \int_{\omega_{p-1}}^{\omega_p} \phi(\omega_p)\mathrm{d}x
\end{align*}
and
\begin{align*}
\frac{1}{2\pi}\int_0^{2\pi\frac{M_T}{T}}\phi(x)\mathrm{d}x
=\frac{1}{2\pi}\sum_{p=1}^{M_T} \int_{\omega_{p-1}}^{\omega_p} \phi(x)\mathrm{d}x .
\end{align*}
Therefore,
\begin{align*}
\left|\frac{1}{T}\sum_{p=1}^{M_T} \phi(\omega_p) - \frac{1}{2\pi}\int_0^{2\pi\frac{M_T}{T}}\phi(x)\mathrm{d}x \right|
\leqslant\frac{1}{2\pi}\sum_{p=1}^{M_T} \int_{\omega_{p-1}}^{\omega_p} |\phi(\omega_p) - \phi(x)|\mathrm{d}x .
\end{align*}
By the uniform continuity of $\phi$, for any $\epsilon > 0$, there exists $\delta>0$ such that for any $x,~y$ satisfying $|x-y|<\delta$, we have $|\phi(x) - \phi(y)|<\epsilon$. Since for any $x\in [\omega_{p-1}, \omega_{p}]$, we have $|\omega_{p} - x|\leqslant 2\pi/T$. Hence, when $T>t_0:= 2\pi/\delta$, we get $|\phi(\omega_p) - \phi(x)|<\epsilon$. Consequently, 
\begin{align*}
\left|\frac{1}{T}\sum_{p=1}^{M_T} \phi(\omega_p) - \frac{1}{2\pi}\int_0^{2\pi\frac{M_T}{T}}\phi(x)\mathrm{d}x \right|
\leqslant\frac{1}{2\pi}\sum_{p=1}^{M_T} \int_{\omega_{p-1}}^{\omega_p} |\phi(\omega_p) - \phi(x)|\mathrm{d}x 
\leqslant \frac{1}{2\pi}\frac{M_T}{T}\epsilon.
\end{align*}
\end{proof}
\end{proposition}
The following result for 2D integral also holds. The proof is similar to the that of the above proposition and hence is omitted.
\begin{proposition}\label{conv_Riemann_sum_prop2}
Let $\phi:\mathbb{R}^2\rightarrow\mathbb{C}$ be continuous on a sufficiently large compact neighborhood of the origin. If $M_T/T\rightarrow L<\infty$ when $T\rightarrow\infty$,
\begin{align*}
\left|\frac{1}{T^2}\sum_{p_1=1}^{M_T}\sum_{p_2=1}^{M_T} \phi(\omega_{p_1}, \omega_{p_2}) - \frac{1}{(2\pi)^2}\int_{[0, 2\pi L]^2}\phi(x_1, x_2)\mathrm{d}x_1 \mathrm{d}x_2 \right|
\rightarrow 0 ,
\end{align*}
where $\omega_p = 2\pi p/T$ are the Fourier frequencies.
\end{proposition}

\begin{theorem}\label{deterministic_conv_thm}
Let $\Theta\subset \mathbb{R}^d$ be a compact set (equipped with the Euclidean norm). For each $\theta\in\Theta$, $\Phi_{\theta}(\cdot)=\{\phi_{\theta, ab}(\cdot)\}_{a,b=1,\cdots,D}$ is a matrix of uniformly continuous, bounded (complex-valued) functions satisfying $\max_{a,b=1\cdots,D}\sup_{\theta\in\Theta}\|\phi_{\theta, ab}\|_{\infty}< \infty$. Moreover, there exist constants $C,\alpha>0$ such that for any $\theta_1,~\theta_2\in\Theta$, $\max_{a,b=1,\cdots,D}\|\phi_{\theta_1, ab} - \phi_{\theta_2, ab}\|_{\infty} \leqslant C\|\theta_1 - \theta_2\|^{\alpha}$. Provided that $M_T/T \rightarrow L < \infty$ when $T\rightarrow\infty$, we will have
\begin{align}\label{E_weighted_periodogram_conv_eqn}
\sup_{\theta\in\Theta}\left|\mathrm{E}A_T(\Phi_{\theta}) - \frac{1}{2\pi}\int_0^{2\pi L}\sum_{a=1}^D\sum_{b=1}^D\phi_{\theta, ab}(x)f_{ba}(x)\mathrm{d}x \right|\rightarrow 0,
\end{align}
where $A_T(\cdot)$ is defined in Equation~\ref{def_SEP} and $f_{ab}$, $a,b=1,\cdots,D$ are the Bartlett spectral density of the underlying Hawkes process. In addition, if $\Phi_{\theta}(\cdot)$ is positive-definite and $\inf_{\theta\in\Theta}\inf_{x\in\mathbb{R}}\det\Phi_{\theta}(x)>0$, we also have
\begin{align}\label{log_phi_conv_eqn}
\sup_{\theta\in\Theta}\left|\frac{1}{T}\sum_{p=1}^{M_T}\log\det\Phi_{\theta}(\omega_p)
- \frac{1}{2\pi}\int_0^{2\pi L}\log\det\Phi_{\theta}(x)\mathrm{d}x \right|\rightarrow 0 .
\end{align}
\begin{proof}
By Corollary~\ref{cov_J_corollary}, we have for any $s\in\mathbb{Z}$, $s\neq 0$,
\begin{align*}
&\mathrm{E}J_T^{a}(2\pi s/T)J_T^{b}(-2\pi s/T)=\mathrm{cum}(J_T^{a}(2\pi s/T), J_T^{b}(-2\pi s/T)) 
=f_{ab}(2\pi s/T) + o(1)
\end{align*}
due to the fact that $\mathrm{E}J_T^{a}(2\pi s/T)=0$ for any $s\in\mathbb{Z}$ and $a=1,\cdots,D$; and
\begin{align*}
&\mathrm{E}A_T(\Phi_{\theta})
=\frac{1}{T}\sum_{p=1}^{M_T}\sum_{a=1}^D\sum_{b=1}^D\phi_{\theta, ba}(\omega_p)\mathrm{E}J_T^{b}(-\omega_p)J_T^{a}(\omega_p) \\
&=\frac{1}{T}\sum_{p=1}^{M_T}\sum_{a=1}^D\sum_{b=1}^D\phi_{\theta, ab}(\omega_p)f_{ba}(\omega_p)
+o(1)\frac{1}{T}\sum_{p=1}^{M_T}\sum_{a=1}^D\sum_{b=1}^D\phi_{\theta, ab}(\omega_p) .
\end{align*}
Since 
\begin{align*}
\sup_{\theta\in\Theta}\left|o(1)\frac{1}{T}\sum_{p=1}^{M_T}\sum_{a=1}^D\sum_{b=1}^D\phi_{\theta, ab}(\omega_p) \right|
\leqslant o(1)\frac{M_T}{T}D\max_{a,b=1,\cdots,D}\sup_{\theta\in\Theta}\|\phi_{\theta, ab}\|_{\infty} \rightarrow 0
\end{align*}
when $T\rightarrow\infty$, \eqref{E_weighted_periodogram_conv_eqn} reduces to
\begin{align*}
\sup_{\theta\in\Theta}\left|\frac{1}{T}\sum_{p=1}^{M_T}\sum_{a=1}^D\sum_{b=1}^D\phi_{\theta}(\omega_p)f_{ba}(\omega_p)
- \frac{1}{2\pi}\int_0^{2\pi L}\sum_{a=1}^D\sum_{b=1}^D\phi_{\theta, ab}(x)f_{ba}(x)\mathrm{d}x \right|\rightarrow 0 .
\end{align*}
As both $\phi_{\theta, ab}$ and $f_{ab}$ are uniformly continuous and bounded, their product is also uniformly continuous and bounded. Therefore, for each $\theta\in\Theta$, Proposition~\ref{conv_Riemann_sum_prop} and the boundedness of $\phi_{\theta, ab}f_{ab}$ together yield
\begin{align*}
\left|\frac{1}{T}\sum_{p=1}^{M_T}\sum_{a=1}^D\sum_{b=1}^D\phi_{\theta, ab}(\omega_p)f_{ba}(\omega_p)
- \frac{1}{2\pi}\int_0^{2\pi L} \sum_{a=1}^D\sum_{b=1}^D\phi_{\theta, ab}(x)f_{ba}(x)\mathrm{d}x \right|\rightarrow 0 .
\end{align*}
We now verify the equicontinuity of the left-hand side of the above equation. For any $\theta_1,~\theta_2\in\Theta$,
\begin{align*}
&\left|\frac{1}{T}\sum_{p=1}^{M_T}\sum_{a=1}^D\sum_{b=1}^D\phi_{\theta_1, ab}(\omega_p)f_{ba}(\omega_p)
- \frac{1}{T}\sum_{p=1}^{M_T}\sum_{a=1}^D\sum_{b=1}^D\phi_{\theta_2, ab}(\omega_p)f_{ba}(\omega_p) \right|
\leqslant \frac{1}{T}\sum_{p=1}^{M_T}\sum_{a=1}^D\sum_{b=1}^D|f_{ba}(\omega_p)| \|\phi_{\theta_1, ab} -\phi_{\theta_2, ab}\|_{\infty} \\
&\leqslant O(1)\|\theta_1 - \theta_2\|^{\alpha}.
\end{align*}
Here, we utilized the H{\"o}lder continuity of $\{\phi_{\theta, ab}\}_{\theta\in\Theta}$ and the boundedness of $f_{ba}$. Similarly, 
\begin{align*}
\left|\frac{1}{2\pi}\int_0^{2\pi L}\sum_{a=1}^D\sum_{b=1}^D\phi_{\theta_1, ab}(x)f_{ba}(x)\mathrm{d}x
- \frac{1}{2\pi}\int_0^{2\pi L}\sum_{a=1}^D\sum_{b=1}^D\phi_{\theta_2, ab}(x)f_{ba}(x)\mathrm{d}x \right|
\leqslant O(1)\|\theta_1 - \theta_2\|^{\alpha}.
\end{align*}
The equicontinuity is verified by combining the two inequalities. The proof of \eqref{E_weighted_periodogram_conv_eqn} is concluded by an application of the Arzel{\`a}-Ascoli theorem (see e.g. Theorem 21.7 of \citet{davidson}).

As for \eqref{log_phi_conv_eqn}, in view of the above derivation, it suffices to verify that for each $\theta\in\Theta$, $\log\det\Phi_{\theta}(\cdot)$ is uniformly continuous and the family $\{\log\det\Phi_{\theta}\}_{\theta\in\Theta}$ is H{\"o}lder continuous. Since $1 - 1/u\leqslant\log u \leqslant u - 1$ for any $u > 0$,
\begin{align*}
&|\log\det\Phi_{\theta}(x) - \log\det\Phi_{\theta}(y)| \leqslant \max\left\{\left| 1 - \frac{\det\Phi_{\theta}(y)}{\det\Phi_{\theta}(x)}\right|,\left| \frac{\det\Phi_{\theta}(x)}{\det\Phi_{\theta}(y)} - 1 \right|\right\} \\
&\leqslant \frac{1}{\inf_{v\in\Theta}\inf_{u\in\mathbb{R}}\det\Phi_{v}(u)}|\det\Phi_{\theta}(x) -\det\Phi_{\theta}(y)|
\end{align*}
holds for any $x,~y\in\mathbb{R}$ and any $\theta\in\Theta$. This shows that $\log\det\Phi_{\theta}(\cdot)$ is uniformly continuous due to the uniform continuity of $\phi_{\theta, ab}(\cdot)$. Similarly, for any $\theta_1,~\theta_2\in\Theta$, we have
\begin{align*}
\|\log\det\Phi_{\theta_1} - \log\det\Phi_{\theta_2}\|_{\infty}
\leqslant \frac{1}{\inf_{v\in\Theta}\inf_{u\in\mathbb{R}}\det\Phi_{v}(u)}\|\det\Phi_{\theta_1} -\det\Phi_{\theta_2}\|_{\infty},
\end{align*}
implying that $\{\log\det\Phi_{\theta}\}_{\theta\in\Theta}$ inherits H{\"o}lder continuity from $\{\phi_{\theta, ab}\}_{\theta\in\Theta}$.
\end{proof}
\end{theorem}

\subsection{Deterministic convergence rate}
\begin{proposition}\label{conv_rate_Riemann_sum_prop}
Suppose $\phi:\mathbb{R}\rightarrow\mathbb{C}$ satisfies that for any $|x-y|\leqslant 1$, $|\phi(x) - \phi(y)|\leqslant b|x-y|^{\beta}$, where $b>0$ and $0<\beta\leqslant1$. Then for sufficiently large $T$,
\begin{align*}
\left|\frac{1}{T}\sum_{p=1}^{M_T} \phi(\omega_p) - \frac{1}{2\pi}\int_0^{2\pi\frac{M_T}{T}}\phi(x)\mathrm{d}x \right|
\leqslant b \frac{(2\pi)^{\beta}M_T}{(\beta+1)T^{\beta+1}} ,
\end{align*}
where $\omega_p = 2\pi p/T$ are the Fourier frequencies.
\begin{proof}
\begin{align*}
\frac{1}{T}\sum_{p=1}^{M_T} \phi(\omega_p)
=\frac{1}{2\pi}\sum_{p=1}^{M_T} \int_{\omega_{p-1}}^{\omega_p} \phi(\omega_p)\mathrm{d}x
\end{align*}
and
\begin{align*}
\frac{1}{2\pi}\int_0^{2\pi\frac{M_T}{T}}\phi(x)\mathrm{d}x
=\frac{1}{2\pi}\sum_{p=1}^{M_T} \int_{\omega_{p-1}}^{\omega_p} \phi(x)\mathrm{d}x .
\end{align*}
Therefore, for any $T> 2\pi$,
\begin{align*}
&\left|\frac{1}{T}\sum_{p=1}^{M_T} \phi(\omega_p) - \frac{1}{2\pi}\int_0^{2\pi\frac{M_T}{T}}\phi(x)\mathrm{d}x \right|
\leqslant\frac{1}{2\pi}\sum_{p=1}^{M_T} \int_{\omega_{p-1}}^{\omega_p} |\phi(\omega_p) - \phi(x)|\mathrm{d}x \\
&\leqslant \frac{b}{2\pi}\sum_{p=1}^{M_T} \int_{\omega_{p-1}}^{\omega_p} (\omega_p - x)^{\beta}\mathrm{d}x
=b \frac{(2\pi)^{\beta}M_T}{(\beta+1)T^{\beta+1}}.
\end{align*}

\end{proof}
\end{proposition}

\begin{theorem}\label{deterministic_conv_rate_thm}
Let $\Theta\subset \mathbb{R}^d$ be a compact set (equipped with the Euclidean norm). For each $\theta\in\Theta$, $\Phi_{\theta}(\cdot)=\{\phi_{\theta, ab}(\cdot)\}_{a,b=1,\cdots,D}$ is a matrix of uniformly continuous, bounded (complex-valued) functions satisfying $\max_{a,b}\sup_{\theta\in\Theta}\|\phi_{\theta, ab}\|_{\infty}< \infty$ and for any $x,y\in\mathbb{R}$, there exist $b>0$ and $1/2<\beta\leqslant1$ such that $\max_{a,b}\sup_{\theta\in\Theta}|\phi_{\theta, ab}(x) - \phi_{\theta, ab}(y)|\leqslant B_1|x-y|^{\beta}$. Provided that $|M_T/T - L| = o(T^{-1/2})$ and $\max_{a,b}\int_{\mathbb{R}}|u|^{\beta}\mathrm{d}C_{ab}^{\mathrm{red}}(u)<\infty$, we will have
\begin{align}\label{E_weighted_periodogram_conv_rate_eqn}
\sup_{\theta\in\Theta}\left|\mathrm{E}A_T(\Phi_{\theta}) - \frac{1}{2\pi}\int_0^{2\pi L}\sum_{a=1}^D\sum_{b=1}^D\phi_{\theta, ab}(x)f_{ba}(x)\mathrm{d}x \right|
= o(1/\sqrt{T}),
\end{align}
where $A_T(\cdot)$ is defined in Equation~\ref{def_SEP} and $\boldsymbol{f}_2 = \{f_{ab}\}_{a,b=1,\cdots,D}$ is the the Bartlett spectral density matrix of the underlying Hawkes process. In addition, if $\inf_{\theta\in\Theta}\inf_{x\in\mathbb{R}}\det\Phi_{\theta}(x)>0$, we also have
\begin{align}\label{log_phi_conv_rate_eqn}
\sup_{\theta\in\Theta}\left|\frac{1}{T}\sum_{p=1}^{M_T}\log\det\Phi_{\theta}(\omega_p)
- \frac{1}{2\pi}\int_0^{2\pi L}\log\det\Phi_{\theta}(x)\mathrm{d}x \right|
= o(1/\sqrt{T}).
\end{align}
\begin{proof}
By Proposition~\ref{cov_J_corollary2} and direct calculation, we have
\begin{align*}
&\mathrm{E}A_T(\Phi_{\theta})
=\frac{1}{T}\sum_{p=1}^{M_T}\sum_{a=1}^D\sum_{b=1}^D\phi_{\theta, ab}(\omega_p)f_{ab}(\omega_p)
+o(M_T/T^{3/2}) ,
\end{align*}
and $\max_{a,b}|f_{ab}(u) - f_{ab}(v)|\leqslant B_2|u-v|^{\beta}$ for some $B_2>0$ and any $u,v\in\mathbb{R}$.

For any $|x-y|\leqslant 1$,
\begin{align*}
&|\phi_{\theta, ab}(x)f_{ba}(x) - \phi_{\theta, ab}(y)f_{ba}(y)|
\leqslant |\phi_{\theta, ab}(x)||f_{ba}(x) - f_{ba}(y)| + |f_{ba}(y)||\phi_{\theta, ab}(x)- \phi_{\theta, ab}(y)| \\
&\leqslant B_2\|\phi_{\theta, ab}\|_{\infty}|x-y|^{\beta} + B\|f_{ba}\|_{\infty}|x-y|^{\beta}
\leqslant (\max_{a,b}B_2\sup_{\theta\in\Theta}\|\phi_{\theta, ab}\|_{\infty} + B\|f_{ba}\|_{\infty})|x-y|^{\beta} . 
\end{align*}
This in fact shows that the product of two H{\"o}lder continuous functions with exponent $\beta$ is also H{\"o}lder continuous with exponent $\beta$. Therefore, Proposition~\ref{conv_rate_Riemann_sum_prop} yields
\begin{align*}
&\sup_{\theta\in\Theta}\left|\frac{1}{T}\sum_{p=1}^{M_T}\sum_{a=1}^D\sum_{b=1}^D\phi_{\theta, ab}(\omega_p)f_{ba}(\omega_p)
- \frac{1}{2\pi}\int_0^{2\pi \frac{M_T}{T}}\sum_{a=1}^D\sum_{b=1}^D\phi_{\theta, ab}(x)f_{ba}(x)\mathrm{d}x \right| \\
&\leqslant D^2\max_{a,b}(B_2\sup_{\theta\in\Theta}\|\phi_{\theta, ab}\|_{\infty} + B\|f_{ba}\|_{\infty}) \frac{(2\pi)^{\beta}}{(\beta+1)T^{\beta}} \frac{M_T}{T}
= o(1/\sqrt{T}) .
\end{align*}
Also,
\begin{align*}
&\sup_{\theta\in\Theta}\left|\frac{1}{2\pi}\int_0^{2\pi \frac{M_T}{T}}\sum_{a=1}^D\sum_{b=1}^D\phi_{\theta, ab}(x)f_{ba}(x)\mathrm{d}x - \frac{1}{2\pi}\int_0^{2\pi L}\sum_{a=1}^D\sum_{b=1}^D\phi_{\theta, ab}(x)f_{ba}(x)\mathrm{d}x\right| \\
&\leqslant D^2\max_{a,b}\sup_{\theta\in\Theta}\|\phi_{\theta, ab}\|_{\infty} \|f_{ba}\|_{\infty}\left|\frac{M_T}{T} - L\right|
=o(1/\sqrt{T}).
\end{align*}
The desired result is obtained by adding the above two equations.

The proof of \eqref{log_phi_conv_rate_eqn} is similar. Noting that since $1 - 1/u\leqslant\log u \leqslant u - 1$ for any $u > 0$, we have for any $x,y\in\mathbb{R}$,
\begin{align*}
&|\log\det\Phi_{\theta}(x) - \log\det\Phi_{\theta}(y)| \leqslant \max\left\{\left| 1 - \frac{\det\Phi_{\theta}(y)}{\det\Phi_{\theta}(x)}\right|,\left| \frac{\det\Phi_{\theta}(x)}{\det\Phi_{\theta}(y)} - 1 \right|\right\} \\
&\leqslant \frac{1}{\inf_{v\in\Theta}\inf_{u\in\mathbb{R}}\det\Phi_{v}(u)}|\det\Phi_{\theta}(x) -\det\Phi_{\theta}(y)|\leqslant \frac{C}{\inf_{v\in\Theta}\inf_{u\in\mathbb{R}}\det\Phi_{v}(u)}|x-y|^{\beta} ,
\end{align*}
where $C>0$ is a constant and the last inequality is due to the fact that the product of H{\"o}lder continuous functions with exponent $\beta$ is again H{\"o}lder continuous with exponent $\beta$.
\end{proof}
\end{theorem}

\section{Weak convergence of the spectral empirical process}\label{app_sect_weak_convergence}
This section is devoted to showing the CLT of the empirical process using the results in Sections~\ref{decomp_periodogram_sect} and \ref{deterministic_sect}.
\subsection{Useful results}
\begin{proposition}\label{cov_periodogram_prop}
For any Fourier frequencies $\omega_{s_1} = 2\pi s_1/T,~\omega_{s_2}=2\pi s_2/T,~ s_1,s_2\in\mathbb{N}$, and $j_1,j_2,k_1,k_2\in\{1,\cdots,D\}$,
\begin{align*}
&\mathrm{cov}(J_T^{j_1}(\omega_{s_1})J_T^{j_2}(-\omega_{s_1}), J_T^{k_1}(\omega_{s_2})J_T^{k_2}(-\omega_{s_2}))
=T^{-1}(f_{j_1 j_2 k_1 k_2}(\omega_{s_1},-\omega_{s_1},\omega_{s_2}) + o(1)) \\
&+ \mathds{1}_{\{s_1+s_2\neq0\}}V_{T}^{j_1 k_1}(s_1,s_2)V_{T}^{j_2 k_2}(-s_1,-s_2)\\ 
&+\mathds{1}_{\{s_1-s_2=0\}}(f_{j_1 k_2}(\omega_{s_1}) + o(1))(f_{j_2 k_1}(-\omega_{s_1}) + o(1)) + \mathds{1}_{\{s_1-s_2\neq0\}}V_{T}^{j_1 k_2}(s_1,-s_2)V_{T}^{j_2 k_1}(-s_1,s_2)  .
\end{align*}
where $V_T$ is defined in Corollary~\ref{cov_J_corollary} and the $o(1)$ terms are independent of $s_1$, $s_2$.
\begin{proof}
We have
\begin{align*}
&\mathrm{cov}(J_T^{j_1}(\omega_{s_1})J_T^{j_2}(-\omega_{s_1}), J_T^{k_1}(\omega_{s_2})J_T^{k_2}(-\omega_{s_2})) 
=\mathrm{cum}(J_T^{j_1}(\omega_{s_1}),J_T^{j_2}(-\omega_{s_1}), J_T^{k_1}(\omega_{s_2}),J_T^{k_2}(-\omega_{s_2})) \\
&+ \mathrm{cum}(J_T^{j_1}(\omega_{s_1}), J_T^{k_1}(\omega_{s_2}))\mathrm{cum}(J_T^{j_2}(-\omega_{s_1}), J_T^{k_2}(-\omega_{s_2})) 
+\mathrm{cum}(J_T^{k_1}(\omega_{s_1}), J_T^{j_2}(-\omega_{s_2}))\mathrm{cum}(J_T^{j_2}(-\omega_{s_1}), J_T^{k_1}(\omega_{s_2}))  .
\end{align*}
In the above expression, the terms involving first cumulants vanish because $\mathrm{E}J_T^{j}(2\pi s/T)=0$ for $|s| = 1,2,\cdots$ and $j\in\{1,\cdots,D\}$. According to Corollary~\ref{cum_J_corollary}, the fourth cumulant satisfies
\begin{align*}
&\mathrm{cum}(J_T^{j_1}(\omega_{s_1}),J_T^{j_2}(-\omega_{s_1}), J_T^{k_1}(\omega_{s_2}),J_T^{k_2}(-\omega_{s_2})) 
=T^{-1}\left(f_{j_1 j_2 k_1 k_2}(\omega_{s_1},-\omega_{s_1},\omega_{s_2}) + o(1)\right).
\end{align*}
By Corollary~\ref{cov_J_corollary}, we get
\begin{align*}
&\mathrm{cum}(J_T^{j_1}(\omega_{s_1}), J_T^{k_1}(\omega_{s_2}))\mathrm{cum}(J_T^{j_2}(-\omega_{s_1}), J_T^{k_2}(-\omega_{s_2})) \\
&=\mathds{1}_{\{s_1+s_2=0\}}(f_{j_1 k_1}(\omega_{s_1}) + o(1))(f_{j_2 k_2}(-\omega_{s_1}) + o(1)) + \mathds{1}_{\{s_1+s_2\neq0\}}V_{T}^{j_1 k_1}(s_1,s_2)V_{T}^{j_2 k_2}(-s_1,-s_2) \\
&= \mathds{1}_{\{s_1+s_2\neq0\}}V_{T}^{j_1 k_1}(s_1,s_2)V_{T}^{j_2 k_2}(-s_1,-s_2)
\end{align*}
and
\begin{align*}
&\mathrm{cum}(J_T^{j_1}(\omega_{s_1}), J_T^{k_2}(-\omega_{s_2}))\mathrm{cum}(J_T^{j_2}(-\omega_{s_1}), J_T^{k_1}(\omega_{s_2})) \\
&=\mathds{1}_{\{s_1-s_2=0\}}(f_{j_1 k_2}(\omega_{s_1}) + o(1))(f_{j_2 k_1}(-\omega_{s_1}) + o(1)) + \mathds{1}_{\{s_1-s_2\neq0\}}V_{T}^{j_1 k_2}(s_1,-s_2)V_{T}^{j_2 k_1}(-s_1,s_2) . 
\end{align*}
All the $o(1)$ terms above are independent of $\omega_1$, $\omega_2$. 
\end{proof}
\end{proposition}

\begin{proposition}\label{cov_sep_prop}
Let
\begin{align*}
A_T(\Phi) = \frac{1}{T}\sum_{p=1}^{M_T} J_T^{H}(\omega_p)\Phi(\omega_p) J_T(\omega_p),
\end{align*}
where $\omega_p = 2\pi p/T$, $p=1,2,\cdots$ are the Fourier frequencies and $\Phi(\cdot) = \{\phi_{ij}(\cdot)\}_{i,j=1,\cdots,D}$ is a $D\times D$ matrix whose elements are functions. Then for any matrices of continuous and bounded (complex-valued) functions $\Phi_j =\{\phi_{jk}^{(r)}\}_{j,k=1,\cdots D}$, $r=1,2$, we have
\begin{enumerate}[(a)]
\item 
\begin{align*}
&\left|\mathrm{cum}(A_T(\Phi_1), A_T(\Phi_2))\right|
\leqslant\max_{r_1,s_1,r_2,s_2\in\{1,\cdots,D\}}\|\phi_{r_1 s_1}^{(1)}\|_{\infty}\|\phi_{r_2 s_2}^{(2)}\|_{\infty}(O(M_T^2/T^3) 
+O(M_T/T^2) ),
\end{align*}
where the big O terms do not depend on $j_1,j_2,k_1,k_2$ in view of Proposition~\ref{HPCum_bound_prop}.  
\item Moreover, when $M_T/T\rightarrow L\in(0,+\infty)$, we have
\begin{align*}
&\mathrm{cum}(\sqrt{T}A_T(\Phi_1), \sqrt{T}A_T(\Phi_2))
\rightarrow \\
&\frac{1}{(2\pi)^2}\int_{[0,2\pi L]^2} \sum_{j_1=1}^{D} \sum_{j_2=1}^{D}\sum_{k_1=1}^{D} \sum_{k_2=1}^{D}\phi_{j_2 j_1}^{(1)}(\omega_{1}) \phi_{k_2 k_1}^{(2)}(\omega_{2})f_{j_1 j_2 k_1 k_2}(\omega_{1},-\omega_{1},\omega_{2})\mathrm{d}\omega_1\mathrm{d}\omega_2 \\
&+ \frac{1}{2\pi}\int_{[0,2\pi L]} \sum_{j_1=1}^{D} \sum_{j_2=1}^{D}\sum_{k_1=1}^{D} \sum_{k_2=1}^{D} \phi_{j_2 j_1}^{(1)}(\omega) \phi_{k_2 k_1}^{(2)}(\omega) f_{j_1 k_2}(\omega)f_{j_2 k_1}(-\omega)\mathrm{d}\omega 
\end{align*}
when $T\rightarrow\infty$ .
\end{enumerate}
\begin{proof}
First, note that we can expand
\begin{align*}
A_T(\Phi) = \frac{1}{T}\sum_{p=1}^{M_T} \sum_{j_2=1}^{D} \sum_{j_1=1}^{D}J_T^{j_2}(-\omega_p)\phi_{j_2 j_1}(\omega_p) J_T^{j_1}(\omega_p)
= \frac{1}{T}\sum_{j_1=1}^{D} \sum_{j_2=1}^{D}\sum_{p=1}^{M_T} J_T^{j_1}(\omega_p) J_T^{j_2}(-\omega_p)\phi_{j_2 j_1}(\omega_p)
\end{align*}
and consequently,
\begin{align*}
&\mathrm{cum}(A_T(\Phi_1), A_T(\Phi_2)) \\
&=\sum_{j_1=1}^{D} \sum_{j_2=1}^{D}\sum_{k_1=1}^{D} \sum_{k_2=1}^{D} \left[T^{-2}\sum_{p_1=1}^{M_T}\sum_{p_2=1}^{M_T}\phi_{j_2 j_1}^{(1)}(\omega_{p_1}) \phi_{k_2 k_1}^{(2)}(\omega_{p_2})\mathrm{cum}\left(J_T^{j_1}(\omega_{p_1}) J_T^{j_2}(-\omega_{p_1}), J_T^{k_1}(\omega_{p_2}) J_T^{k_2}(-\omega_{p_2})\right)\right] .
\end{align*} 
By Proposition~\ref{cov_periodogram_prop} and direct calculations, for any $r_1,s_1,r_2,s_2\in\{1,\cdots,D\}$, we have
\begin{align*}
&T^{-2}\sum_{p_1=1}^{M_T}\sum_{p_2=1}^{M_T}\phi_{j_2 j_1}^{(1)}(\omega_{p_1}) \phi_{k_2 k_1}^{(2)}(\omega_{p_2})\mathrm{cum}\left(J_T^{j_1}(\omega_{p_1}) J_T^{j_2}(-\omega_{p_1}), J_T^{k_1}(\omega_{p_2}) J_T^{k_2}(-\omega_{p_2})\right) \\
&=T^{-2}\sum_{p_1=1}^{M_T}\sum_{p_2=1}^{M_T}\phi_{j_2 j_1}^{(1)}(\omega_{p_1}) \phi_{k_2 k_1}^{(2)}(\omega_{p_2})[T^{-1}(f_{j_1 j_2 k_1 k_2}(\omega_{p_1},-\omega_{p_1},\omega_{p_2}) + o(1))] 
\\
&+ T^{-2}\sum_{p_1=1}^{M_T}\sum_{p_2=1}^{M_T}\phi_{j_2 j_1}^{(1)}(\omega_{p_1}) \phi_{k_2 k_1}^{(2)}(\omega_{p_2}) V_{T}^{j_1 k_1}(p_1,p_2)V_{T}^{j_2 k_2}(-p_1,-p_2) \\
&+T^{-2}\sum_{p_1=1}^{M_T}\sum_{p_2=1}^{M_T}\mathds{1}_{\{p_1-p_2=0\}}\phi_{j_2 j_1}^{(1)}(\omega_{p_1}) \phi_{k_2 k_1}^{(2)}(\omega_{p_2})(f_{j_1 k_2}(\omega_{p_1}) + o(1))(f_{j_2 k_1}(-\omega_{p_1}) + o(1)) \\
&+T^{-2}\sum_{p_1=1}^{M_T}\sum_{p_2=1}^{M_T}\mathds{1}_{\{p_1-p_2\neq0\}}\phi_{j_2 j_1}^{(1)}(\omega_{p_1}) \phi_{k_2 k_1}^{(2)}(\omega_{p_2})V_{T}^{j_1 k_2}(p_1,-p_2) V_{T}^{j_2 k_1}(-p_1,p_2) \\
&=T^{-3}\sum_{p_1=1}^{M_T}\sum_{p_2=1}^{M_T}\phi_{j_2 j_1}^{(1)}(\omega_{p_1}) \phi_{k_2 k_1}^{(2)}(\omega_{p_2})f_{j_1 j_2 k_1 k_2}(\omega_{p_1},-\omega_{p_1},\omega_{p_2}) 
+ \|\phi_{j_2 j_1}^{(1)}\|_{\infty}\|\phi_{k_2 k_1}^{(2)}\|_{\infty}o(M_T^2/T^3 ) \\
&+T^{-2}\sum_{p_1=1}^{M_T}\sum_{p_2=1}^{M_T}\phi_{j_2 j_1}^{(1)}(\omega_{p_1}) \phi_{k_2 k_1}^{(2)}(\omega_{p_2}) V_{T}^{j_1 k_1}(p_1,p_2) V_{T}^{j_2 k_2}(-p_1,-p_2) \\
&+ T^{-2}\sum_{p=1}^{M_T}\phi_{j_2 j_1}^{(1)}(\omega_{p}) \phi_{k_2 k_1}^{(2)}(\omega_{p}) f_{j_1 k_2}(\omega_p)f_{j_2 k_1}(-\omega_p)
+ \|\phi_{j_2 j_1}^{(1)}\phi_{k_2 k_1}^{(2)}\|_{\infty}o(M_T/T^2 )\\
&+T^{-2}\sum_{\substack{p_1,p_2=1 \\ p_1\neq p_2}}^{M_T}\phi_{j_2 j_1}^{(1)}(\omega_{p_1}) \phi_{k_2 k_1}^{(2)}(\omega_{p_2}) V_{T}^{j_1 k_2}(p_1,-p_2) V_{T}^{j_2 k_1}(-p_1,p_2) .
\end{align*} 

For part (a), first notice that from Definition~\ref{HPcsd_def}, we get
\begin{align*}
T^{-3}\left|\sum_{p_1=1}^{M_T}\sum_{p_2=1}^{M_T}\phi_{j_2 j_1}^{(1)}(\omega_{p_1}) \phi_{k_2 k_1}^{(2)}(\omega_{p_2})f_{j_1 j_2 k_1 k_2}(\omega_{p_1},-\omega_{p_1},\omega_{p_2}) \right|
\leqslant \frac{M_T^2}{T^3} \|\phi_{j_2 j_1}^{(1)}\|_{\infty}\|\phi_{k_2 k_1}^{(2)}\|_{\infty} C_{j_1 j_2 k_1 k_2}^{\mathrm{red}}(\mathbb{R}^3) ,
\end{align*}
and
\begin{align*}
T^{-2}\left|\sum_{p=1}^{M_T}\phi_{j_2 j_1}^{(1)}(\omega_{p}) \phi_{k_2 k_1}^{(2)}(\omega_{p}) f_{j_1 k_2}(\omega_p)f_{j_2 k_1}(-\omega_p)\right|
\leqslant \frac{M_T}{T^2} \|\phi_{j_2 j_1}^{(1)}\|_{\infty}\|\phi_{k_2 k_1}^{(2)}\|_{\infty} C_{j_1 k_2}^{\mathrm{red}}(\mathbb{R}) C_{j_2 k_1}^{\mathrm{red}}(\mathbb{R}).
\end{align*}
From the definition of $V_T^{\cdot \cdot}(\cdot,\cdot)$ in Corollary~\ref{cov_J_corollary}, we get for any $j_1,j_2\in\{1,\cdots,D\}$ and any $p_1, p_2\in\mathbb{Z}$, 
\begin{align}\label{V2_bound}
|V_{T}^{j_1 j_2}(p_1,p_2)|
\leqslant \frac{C_{j_1 j_2}^{\mathrm{red}}(\mathbb{R})}{\pi} \frac{1}{|p_1 + p_2|} .
\end{align}
Therefore, by Corollary~\ref{indecomp_sum_corollary} (1), we have
\begin{align*}
&T^{-2}\left|\sum_{\substack{p_1,p_2=1 \\ p_1\neq p_2}}^{M_T}\phi_{j_2 j_1}^{(1)}(\omega_{p_1}) \phi_{k_2 k_1}^{(2)}(\omega_{p_2})V_{T}^{j_1 k_2}(p_1,-p_2) V_{T}^{j_2 k_1}(-p_1,p_2)\right|
\leqslant \|\phi_{j_2 j_1}^{(1)}\|_{\infty}\|\phi_{k_2 k_1}^{(2)}\|_{\infty}O\left(T^{-2}\sum_{\substack{p_1,p_2=1 \\ p_1\neq p_2}}^{M_T} \frac{1}{|p_1 - p_2|^2}\right) \\
&=\|\phi_{j_2 j_1}^{(1)}\|_{\infty}\|\phi_{k_2 k_1}^{(2)}\|_{\infty}O\left(M_T /T^{2}\right) ,
\end{align*}
and
\begin{align*}
T^{-2}\left|\sum_{p_1,p_2=1}^{M_T}\phi_{j_2 j_1}^{(1)}(\omega_{p_1}) \phi_{k_2 k_1}^{(2)}(\omega_{p_2})V_{T}^{j_1 k_2}(p_1,p_2) V_{T}^{j_2 k_2}(-p_1,-p_2) \right|
\leqslant \|\phi_{j_2 j_1}^{(1)}\|_{\infty}\|\phi_{k_2 k_1}^{(2)}\|_{\infty}O\left(\ln M_T/ T^{2}\right) .
\end{align*}
The above results together yield
\begin{align*}
&\left|\mathrm{cov}(A_T(\Phi_1), A_T(\Phi_2))\right|
\leqslant\max_{r_1,s_1,r_2,s_2\in\{1,\cdots,D\}}\|\phi_{r_1 s_1}^{(1)}\|_{\infty}\|\phi_{r_2 s_2}^{(2)}\|_{\infty}(O(M_T^2/T^3) 
+O(M_T/T^2) ),
\end{align*}
where the big O terms do not depend on $j_1,j_2,k_1,k_2$ in view of Proposition~\ref{HPCum_bound_prop}.

As for part (b),
\begin{align*}
&\mathrm{cum}(\sqrt{T}A_T(\Phi_1), \sqrt{T}A_T(\Phi_2)) 
=T\mathrm{cum}(A_T(\Phi_1), A_T(\Phi_2)) \\
&=\sum_{j_1=1}^{D} \sum_{j_2=1}^{D}\sum_{k_1=1}^{D} \sum_{k_2=1}^{D} \\
&=\left[T^{-2}\sum_{p_1=1}^{M_T}\sum_{p_2=1}^{M_T}\phi_{j_2 j_1}^{(1)}(\omega_{p_1}) \phi_{k_2 k_1}^{(2)}(\omega_{p_2})f_{j_1 j_2 k_1 k_2}(\omega_{p_1},-\omega_{p_1},\omega_{p_2}) 
+ \|\phi_{j_2 j_1}^{(1)}\|_{\infty}\|\phi_{k_2 k_1}^{(2)}\|_{\infty}o(M_T^2/T^2 ) \right. \\
&+T^{-1}\sum_{p_1=1}^{M_T}\sum_{p_2=1}^{M_T}\phi_{j_2 j_1}^{(1)}(\omega_{p_1}) \phi_{k_2 k_1}^{(2)}(\omega_{p_2}) V_{T}^{j_1 k_1}(p_1,p_2) V_{T}^{j_2 k_2}(-p_1,-p_2) \\
&+ T^{-1}\sum_{p=1}^{M_T}\phi_{j_2 j_1}^{(1)}(\omega_{p}) \phi_{k_2 k_1}^{(2)}(\omega_{p}) f_{j_1 k_2}(\omega_p)f_{j_2 k_1}(-\omega_p)
+ \|\phi_{j_2 j_1}^{(1)}\phi_{k_2 k_1}^{(2)}\|_{\infty}o(M_T/T )\\
&\left.+T^{-1}\sum_{\substack{p_1,p_2=1 \\ p_1\neq p_2}}^{M_T}\phi_{j_2 j_1}^{(1)}(\omega_{p_1}) \phi_{k_2 k_1}^{(2)}(\omega_{p_2}) V_{T}^{j_1 k_2}(p_1,-p_2) V_{T}^{j_2 k_1}(-p_1,p_2) \right].
\end{align*}
In view of the fact that the cumulant density functions are bounded and continuous (due to Proposition~\ref{HPCum_bound_prop}) and the proof of part (a), we have
\begin{align*}
&T^{-2}\sum_{p_1=1}^{M_T}\sum_{p_2=1}^{M_T}\phi_{j_2 j_1}^{(1)}(\omega_{p_1}) \phi_{k_2 k_1}^{(2)}(\omega_{p_2})f_{j_1 j_2 k_1 k_2}(\omega_{p_1},-\omega_{p_1},\omega_{p_2})  \\
&\rightarrow \frac{1}{(2\pi)^2}\int_{[0,2\pi L]^2} \phi_{j_2 j_1}^{(1)}(\omega_{1}) \phi_{k_2 k_1}^{(2)}(\omega_{2})f_{j_1 j_2 k_1 k_2}(\omega_{1},-\omega_{1},\omega_{2})\mathrm{d}\omega_1\mathrm{d}\omega_2 ,
\end{align*}
\begin{align*}
T^{-1}\sum_{p=1}^{M_T}\phi_{j_2 j_1}^{(1)}(\omega_{p}) \phi_{k_2 k_1}^{(2)}(\omega_{p}) f_{j_1 k_2}(\omega_p)f_{j_2 k_1}(-\omega_p)
\rightarrow \frac{1}{2\pi}\int_{[0,2\pi L]}\phi_{j_2 j_1}^{(1)}(\omega) \phi_{k_2 k_1}^{(2)}(\omega) f_{j_1 k_2}(\omega)f_{j_2 k_1}(-\omega)\mathrm{d}\omega  ,
\end{align*} 
and
\begin{align*}
T^{-1}\left|\sum_{p_1,p_2=1}^{M_T}\phi_{j_2 j_1}^{(1)}(\omega_{p_1}) \phi_{k_2 k_1}^{(2)}(\omega_{p_2})V_{T}^{j_1 k_2}(p_1,p_2) V_{T}^{j_2 k_2}(-p_1,-p_2) \right|
\leqslant \|\phi_{j_2 j_1}^{(1)}\|_{\infty}\|\phi_{k_2 k_1}^{(2)}\|_{\infty}O\left(\ln M_T/ T\right) = o(1)
\end{align*}
when $T\rightarrow\infty$. Moreover, we know from Corollary~\ref{cov_J_corollary} that
\begin{align*}
&T^{-1}\sum_{\substack{p_1,p_2=1 \\ p_1\neq p_2}}^{M_T}\phi_{j_2 j_1}^{(1)}(\omega_{p_1}) \phi_{k_2 k_1}^{(2)}(\omega_{p_2})V_{T}^{j_1 k_2}(p_1,-p_2) V_{T}^{j_2 k_1}(-p_1,p_2)
= \int_{\mathbb{R}^2} \varphi_{T}(u,v)\mathrm{d}C_{j_1 k_2}^{\mathrm{red}}(u)\mathrm{d}C_{j_2 k_1}^{\mathrm{red}}(v) ,
\end{align*}
where
\begin{align*}
&\varphi_{T}(u,v) = \left\{\frac{4}{T^3}\sum_{\substack{p_1,p_2=1 \\ p_1\neq p_2}}^{M_T}\phi_{j_2 j_1}^{(1)}(\omega_{p_1}) \phi_{k_2 k_1}^{(2)}(\omega_{p_2}) 
\left[\left(-\mathds{1}_{\{u \geqslant 0\}} + \mathds{1}_{\{u < 0\}}\right)\frac{1 - \exp(\mathrm{i}(\omega_{p_1} - \omega_{p_2})u)}{\mathrm{i}(\omega_{p_1} - \omega_{p_2})}\right]\right. \\
&\left. \left[\left(-\mathds{1}_{\{v \geqslant 0\}} + \mathds{1}_{\{v < 0\}}\right)\frac{1 - \exp(\mathrm{i}(\omega_{p_2} - \omega_{p_1})v)}{\mathrm{i}(\omega_{p_2} - \omega_{p_1})}\right] \right\}\mathds{1}_{\{(u,v)\in (-T,T)^2\}} .   
\end{align*}
Since $|1 - \exp(\mathrm{i}\omega u)|/|\omega|\leqslant |u|$ and $|1 - \exp(\mathrm{i}\omega u)|\leqslant 2$ for any $u, \omega\in\mathbb{R}$, we have
\begin{align*}
|\varphi_{T}(u,v)|
\leqslant 4\|\phi_{j_2 j_1}^{(1)}\|_{\infty}\|\phi_{k_2 k_1}^{(2)}\|_{\infty}\frac{M_T^2}{T^3}|u||v|\mathds{1}_{\{(u,v)\in (-T,T)^2\}} \rightarrow 0
\end{align*}
when $T\rightarrow \infty$, and
\begin{align*}
&|\varphi_{T}(u,v)|
\leqslant \frac{16}{T^3}\|\phi_{j_2 j_1}^{(1)}\|_{\infty}\|\phi_{k_2 k_1}^{(2)}\|_{\infty} \sum_{\substack{p_1,p_2=1 \\ p_1\neq p_2}}^{M_T}\frac{1}{|\omega_{p_1} - \omega_{p_2}|^2} \mathds{1}_{\{(u,v)\in (-T,T)^2\}} \\
&= \frac{4}{\pi^2}\|\phi_{j_2 j_1}^{(1)}\|_{\infty}\|\phi_{k_2 k_1}^{(2)}\|_{\infty}\frac{1}{T}\sum_{\substack{p_1,p_2=1 \\ p_1\neq p_2}}^{M_T} \frac{1}{|p_1 - p_2|^2} \mathds{1}_{\{(u,v)\in (-T,T)^2\}} 
= O(M_T/T) \mathds{1}_{\{(u,v)\in (-T,T)^2\}}
\end{align*}
because of Corollary~\ref{indecomp_sum_corollary} (1). Notice that $\mathds{1}_{\{(u,v)\in (-T,T)^2\}}$ is integrable with respect to $C_{j_1 k_2}^{\mathrm{red}}\times C_{j_2 k_1}^{\mathrm{red}}$. Therefore, an application of the dominated convergence theorem yields 
\begin{align*}
T^{-1}\left|\sum_{\substack{p_1,p_2=1 \\ p_1\neq p_2}}^{M_T}\phi_{j_2 j_1}^{(1)}(\omega_{p_1}) \phi_{k_2 k_1}^{(2)}(\omega_{p_2})V_{T}^{j_1 k_2}(p_1,-p_2) V_{T}^{j_2 k_1}(-p_1,p_2) \right|
\leqslant o(M_T^2/T^2)\|\phi_{j_2 j_1}^{(1)}\|_{\infty}\|\phi_{k_2 k_1}^{(2)}\|_{\infty} = o(1) .
\end{align*}

Hence,
\begin{align*}
&\mathrm{cov}(\sqrt{T}A_T(\Phi_1), \sqrt{T}A_T(\Phi_2))
\rightarrow \frac{1}{(2\pi)^2}\int_{[0,2\pi L]^2} \sum_{j_1=1}^{D} \sum_{j_2=1}^{D}\sum_{k_1=1}^{D} \sum_{k_2=1}^{D}\phi_{j_2 j_1}^{(1)}(\omega_{1}) \phi_{k_2 k_1}^{(2)}(\omega_{2})f_{j_1 j_2 k_1 k_2}(\omega_{1},-\omega_{1},\omega_{2})\mathrm{d}\omega_1\mathrm{d}\omega_2 \\
&+ \frac{1}{2\pi}\int_{[0,2\pi L]} \sum_{j_1=1}^{D} \sum_{j_2=1}^{D}\sum_{k_1=1}^{D} \sum_{k_2=1}^{D} \phi_{j_2 j_1}^{(1)}(\omega) \phi_{k_2 k_1}^{(2)}(\omega) f_{j_1 k_2}(\omega)f_{j_2 k_1}(-\omega)\mathrm{d}\omega .
\end{align*}

\end{proof}
\end{proposition}

\begin{proposition}\label{higher_order_Expectation_prop}
Let
\begin{align*}
A_T(\Phi) = \frac{1}{T}\sum_{p=1}^{M_T} J_T^{H}(\omega_p)\Phi(\omega_p) J_T(\omega_p),
\end{align*}
where $\omega_p = 2\pi p/T$, $p=1,2,\cdots$ are the Fourier frequencies and $\Phi(\cdot) = \{\phi_{ij}(\cdot)\}_{i,j=1,\cdots,D}$ is a $D\times D$ matrix whose elements are functions. Suppose $\Phi_j =\{\phi_{jk}^{(r)}\}_{j,k=1,\cdots D}$, $r=1,\cdots$ are any matrices of continuous and bounded  (complex-valued) functions, and $M_T/T = O(1)$.
\begin{enumerate}[(a)]
\item For $l\geqslant3$,
\begin{align*}
\left|\mathrm{cum}\left(T^{1/2}(A_T(\Phi_i) - \mathrm{E}A_T(\Phi_i);i=1,\cdots,l)\right)\right|
= \max\{1, M_T/T\}^{l} o(1).
\end{align*}

\item For sufficiently large $T$,
\begin{align*}
&\left|\mathrm{cum}\left(T^{1/2}(A_T(\Phi_i) - \mathrm{E}A_T(\Phi_i);i=1,\cdots,l)\right)\right| \\
&\leqslant  \max_{j_{11},\cdots,j_{2l} = 1,\cdots,D}\prod_{i=1}^l\|\phi_{j_{2i} j_{1i}}^{(i)}\|_{\infty}\max\{1, M_T/T\}^{l} C_1^l (2l-1)! ,
\end{align*}
holds for any $l\geqslant2$, where $C_1=2C_0^2 D^2$, $C_0$ is the constant in Proposition~\ref{HPCum_bound_prop}.

\item Moreover, for any matrix of continuous and bounded (complex-valued) functions $\Phi$, we also have
\begin{align*}
\left|\mathrm{E}\left[T^{1/2}(A_T(\Phi) - \mathrm{E}A_T(\Phi))\right]^l\right|
\leqslant \left(\max_{j_{1},j_{2} = 1,\cdots,D}\|\phi_{j_{2} j_{1}}\|_{\infty}\right)^l\max\{1, M_T/T\}^{l} C_1^l (2l)! .
\end{align*}
\end{enumerate}

\begin{proof}
We have by direct calculations that
\begin{align*}
&\left|\mathrm{cum}\left(T^{1/2}(A_T(\Phi_j) - \mathrm{E}A_T(\Phi_j);j=1,\cdots,l)\right)\right|
=T^{l/2}|\mathrm{cum}\left(A_T(\phi_j);j=1,\cdots,l\right)| \\
&=T^{-l/2}\left|\sum_{j_{21}=1}^D\sum_{j_{11}=1}^D\cdots\sum_{j_{2l}=1}^D\sum_{j_{1l}=1}^D\sum_{p_1=1}^{M_T}\cdots\sum_{p_l=1}^{M_T}\left(\prod_{i=1}^l \phi_{j_{2i} j_{1i}}^{(i)}(\omega_{p_i})\right)\mathrm{cum}(J_T^{j_{11}}(\omega_{p_1})J_T^{j_{21}}(-\omega_{p_1}),\cdots,J_T^{j_{1l}}(\omega_{p_l})J_T^{j_{2l}}(-\omega_{p_l}))\right| \\
&\leqslant T^{-l/2}\sum_{j_{21}=1}^D\sum_{j_{11}=1}^D\cdots\sum_{j_{2l}=1}^D\sum_{p_1=1}^{M_T}\cdots\sum_{p_l=1}^{M_T}\prod_{i=1}^l\left| \phi_{j_{2i} j_{1i}}^{(i)}(\omega_{p_i})\right|\left|\mathrm{cum}(J_T^{j_{11}}(\omega_{p_1})J_T^{j_{21}}(-\omega_{p_1}),\cdots,J_T^{j_{1l}}(\omega_{p_l})J_T^{j_{2l}}(-\omega_{p_l}))\right| \\
&\leqslant \max_{j_{11},\cdots,j_{2l} = 1,\cdots,D}\prod_{i=1}^l\|\phi_{j_{2i} j_{1i}}^{(i)}\|_{\infty} \\ 
&\sum_{j_{21}=1}^D\sum_{j_{11}=1}^D\cdots\sum_{j_{2l}=1}^D \left[T^{-l/2}\sum_{p_1=1}^{M_T}\cdots\sum_{p_l=1}^{M_T}\left|\mathrm{cum}(J_T^{j_{11}}(\omega_{p_1})J_T^{j_{21}}(-\omega_{p_1}),\cdots,J_T^{j_{1l}}(\omega_{p_l})J_T^{j_{2l}}(-\omega_{p_l}))\right| \right].
\end{align*}
By applying the product theorem for cumulants (see Theorem 2.3.2 of \citet{Brillinger2001}), we get
\begin{align*}
&\left|\mathrm{cum}(J_T^{j_{11}}(\omega_{p_1})J_T^{j_{21}}(-\omega_{p_1}),\cdots,J_T^{j_{1l}}(\omega_{p_l})J_T^{j_{2l}}(-\omega_{p_l}))\right| \\
&\leqslant \sum_{\textrm{ i.p.}} |\mathrm{cum}(J_T^{\cdot}(2\pi i/T); i\in Q_1)|\cdots|\mathrm{cum}(J_T^{\cdot}(2\pi i/T); i\in Q_r)|,
\end{align*}
where $J_T^{\cdot}(2\pi p_{k}/T) = J_T^{j_{1k}}(2\pi p_{k}/T)$, $J_T^{\cdot}(-2\pi p_{k}/T) = J_T^{j_{2k}}(-2\pi p_{k}/T)$, and the sum is over all indecomposable partitions $\{Q_1,\cdots,Q_r\}$ with $|Q_j|\geqslant2$ of the following table
\begin{align*}
\begin{matrix}
&p_1&-p_1 \\
&\vdots&\vdots \\
&p_l&-p_l
\end{matrix}.
\end{align*}
The reason we don't  need to consider partitions with singleton block is because $\mathrm{cum}(J_T^{j}(2\pi p/T)) = \mathrm{E}J_T^{j}(2\pi p/T)=0$ for $p\in\mathbb{Z}\setminus\{0\}$ and $j=1,\cdots,D$. This cardinality restriction also leads to $r\leqslant l$.
Corollary~\ref{cum_J_corollary} (b) and Proposition~\ref{HPCum_bound_prop} together ensure that there exists a constant $C_{0}>0$ such that for any block $Q$ of any indecomposable partition,
\begin{align*}
&|\mathrm{cum}(J_T^{\cdot}(2\pi i/T), i\in Q)|
\leqslant T^{1-|Q|/2}C_0^{|Q|}(|Q|-1)!\left(\mathds{1}_{\{\sum_{i\in Q}i = 0\}} + \frac{1}{|\sum_{i\in Q}i|}\mathds{1}_{\{\sum_{i\in Q}i \neq 0\}}\right)
\end{align*}
Therefore, 
\begin{align}
&T^{-l/2}\sum_{p_1=1}^{M_T}\cdots\sum_{p_l=1}^{M_T}\sum_{\textrm{ i.p.}} |\mathrm{cum}(J_T^{\cdot}(2\pi i/T); i\in Q_1)|\cdots|\mathrm{cum}(J_T^{\cdot}(2\pi i/T); i\in Q_r)| \notag\\
&\leqslant C_0^{2l}\sum_{\textrm{ i.p.}}\left(\prod_{j=1}^{r}(|Q_j| - 1)!\right)T^{-l/2+r-l}\sum_{p_1=1}^{M_T}\cdots\sum_{p_l=1}^{M_T} \prod_{j=1}^{r}  \left( \mathds{1}_{\{\sum_{i\in Q_j}i = 0\}} + \frac{1}{|\sum_{i\in Q_j}i|}\mathds{1}_{\{\sum_{i\in Q_j}i \neq 0\}}\right) . \label{ineq_he1}
\end{align}

To prove part (a), by direct expansion and applying Proposition~\ref{indecomp_sum_corollary}, when $r=1$,
\begin{align*}
&T^{-l/2+r-l}\sum_{p_1=1}^{M_T}\cdots\sum_{p_l=1}^{M_T}  \left( \mathds{1}_{\{\sum_{i\in Q_1}i = 0\}} + \frac{1}{|\sum_{i\in Q_1}i|}\mathds{1}_{\{\sum_{i\in Q_1}i \neq 0\}}\right) \\
&= T^{1-l/2} (M_T/T)^{l} 
\leqslant T^{1-l/2}\max\{1, (M_T/T)^{l}\};
\end{align*}
when $r=2$,
\begin{align*}
&T^{-l/2+r-l}\sum_{p_1=1}^{M_T}\cdots\sum_{p_l=1}^{M_T}  \prod_{j=1}^2\left( \mathds{1}_{\{\sum_{i\in Q_j}i = 0\}} + \frac{1}{|\sum_{i\in Q_j}i|}\mathds{1}_{\{\sum_{i\in Q_j}i \neq 0\}}\right) \\
&\leqslant T^{2-3l/2}(M_T^{l-1} + 4M_T^{l-1})
= 5T^{1 - l/2} (M_T/T)^{l-1}
\leqslant 5T^{1-l/2}\max\{1, (M_T/T)^{l}\};
\end{align*}
when $r\geqslant3$,
\begin{align*}
&\sum_{p_1=1}^{M_T}\cdots\sum_{p_l=1}^{M_T} \prod_{j=1}^{r}  \left( \mathds{1}_{\{\sum_{i\in Q_j}i = 0\}} + \frac{1}{|\sum_{i\in Q_j}i|}\mathds{1}_{\{\sum_{i\in Q_j}i \neq 0\}}\right) \\
&\leqslant T^{r - 3l/2} \left(M_T^{l-r+1} + {r \choose 2} 4M_T^{l-r+1} + \sum_{q=3}^{r}{r \choose q}4^{q-1}M_T^{l-r+1} (\log M_T)^{q-1} \right) \\
&\leqslant T^{1 - l/2} (1 + 4\log M_T)^{r} (M_T/T)^{l-r+1}
\leqslant o(1)\max\{1, (M_T/T)^{l}\} .
\end{align*}
Therefore, the right-hand side of \eqref{ineq_he1} is upperbounded by $\max\{1, (M_T/T)^{l}\}o(1)$ in view of the fact that $l\geqslant 3$ and $M_T/T = O(1)$. Part (a) can then be proven by noticing that the $o(1)$ term is independent of $j_{1k},j_{2k},k=1,\cdots,l$.

As for part (b), note that from the above derivation, for sufficiently large $T$,
\begin{align*}
&T^{-l/2+r-l}\sum_{p_1=1}^{M_T}\cdots\sum_{p_l=1}^{M_T}  \left( \mathds{1}_{\{\sum_{i\in Q_j}i = 0\}} + \frac{1}{|\sum_{i\in Q_j}i|}\mathds{1}_{\{\sum_{i\in Q_j}i \neq 0\}}\right) 
\leqslant \max\{1, (M_T/T)^{l}\}
\end{align*}
holds for any $l\geqslant 2$, and consequently,
\begin{align*}
&\left|\mathrm{cum}\left(T^{1/2}(A_T(\Phi_i) - \mathrm{E}A_T(\Phi_i);i=1,\cdots,l)\right)\right| \\
&\leqslant \max_{j_{11},\cdots,j_{2l} = 1,\cdots,D}\prod_{i=1}^l\|\phi_{j_{2i} j_{1i}}^{(i)}\|_{\infty}  \max\{1, (M_T/T)^{l}\}
\sum_{j_{21}=1}^D\sum_{j_{11}=1}^D\cdots\sum_{j_{2l}=1}^D C_0^{2l}\sum_{\textrm{ i.p.}}\left(\prod_{j=1}^{r}(|Q_j| - 1)!\right) \\
&= \max_{j_{11},\cdots,j_{2l} = 1,\cdots,D}\prod_{i=1}^l\|\phi_{j_{2i} j_{1i}}^{(i)}\|_{\infty}  \max\{1, (M_T/T)^{l}\}
 (C_0 D)^{2l}\sum_{\textrm{ i.p.}}\left(\prod_{j=1}^{r}(|Q_j| - 1)!\right).
\end{align*}
Note that 
\begin{align*}
\sum_{\textrm{ i.p.}}\prod_{j=1}^{r}(|Q_j|-1)!
\leqslant \sum_{\nu_1,\cdots,\nu_w}\prod_{v=1}^{w}(|\nu_v|-1)!
= \sum_{k=1}^{2l} B_{2l,k}(0!,1!,\cdots,(2l -k)!)=(2l)!,
\end{align*}
where $\nu_1,\cdots,\nu_w$ is a partition of $\{1,2,\cdots,2l\}$ and the sum at the right-hand side of the first inequality is over all partitions of $\{1,2,\cdots,2l\}$, $B_{n,k}$ is the $(n,k)$-th partial exponential Bell polynomial. The last two equalities are due to the combinatorial meaning of $B_{n,k}$ (see e.g. \citet[pp. 95--97]{BellPolynomial-encyc-maths}). The proof of part (b) is concluded by noting $2l\leqslant 2^l$ and letting $C_1 = 2C_0^2 D^2$.

As for part (c), we again use the product theorem for cumulants and part (b) to get
\begin{align*}
&\left|\mathrm{E}\left[T^{1/2}(A_T(\Phi) - \mathrm{E}A_T(\Phi))\right]^l\right|
= \left|\sum_{\mu_1,\cdots,\mu_m} \prod_{j=1}^m \mathrm{cum}_{|\mu_j|}\left(T^{1/2}(A_T(\Phi) - \mathrm{E}A_T(\Phi))\right)\right| \\
&\leqslant \sum_{\mu_1,\cdots,\mu_m} \prod_{j=1}^m \left|\mathrm{cum}_{|\mu_j|}\left(T^{1/2}(A_T(\phi) - \mathrm{E}A_T(\phi))\right)\right| \\
&\leqslant \max_{j_{11},\cdots,j_{2l} = 1,\cdots,D}\prod_{i=1}^l\|\phi_{j_{2i} j_{1i}}\|_{\infty} \max\{1, M_T/T\}^{l} C_1^l\sum_{\mu_1,\cdots,\mu_m} (2|\mu_1|-1)!\cdots(2|\mu_m|-1)! \\
&\leqslant \left(\max_{j_{1},j_{2} = 1,\cdots,D}\|\phi_{j_{2} j_{1}}\|_{\infty}\right)^l \max\{1, M_T/T\}^{l} C_1^l\sum_{\mu_1,\cdots,\mu_m} (2|\mu_1|-1)!\cdots(2|\mu_m|-1)! ,
\end{align*}
where $\mathrm{cum}_{n}(X)$ is the cumulant of $n$ copies of $X$, $\mu_1,\cdots,\mu_m$ is a partition of $\{1,2,\cdots,l\}$ and the sums are over all such partitions. Let $\tilde{\mu}_1,\cdots,\tilde{\mu}_m$ be a partition of $\{1,\cdots,2l\}$ constructed based on $\mu_1,\cdots,\mu_m$ by mapping $i$ to $(i,l+i)$ (i.e. if $i\in\mu_j$ for some $j$, then $i,l+i\in\tilde{\mu}_j$), we then get
\begin{align*}
\sum_{\mu_1,\cdots,\mu_m} (2|\mu_1|-1)!\cdots(2|\mu_m|-1)!
= \sum_{\tilde{\mu}_1,\cdots,\tilde{\mu}_m} (|\tilde{\mu}_1|-1)!\cdots(|\tilde{\mu}_m|-1)!
\leqslant (2l)!.
\end{align*}
\end{proof}
\end{proposition}

\begin{proposition}[Exponential inequality]\label{concentration_ineq_strong_prop}
Let
\begin{align*}
A_T(\Phi) = \frac{1}{T}\sum_{p=1}^{M_T} J_T^{H}(\omega_p)\Phi(\omega_p) J_T(\omega_p),
\end{align*}
where $\omega_p = 2\pi p/T$, $p=1,2,\cdots$ are the Fourier frequencies and $\Phi(\cdot) = \{\phi_{ij}(\cdot)\}_{i,j=1,\cdots,D}$ is a $D\times D$ matrix whose elements are continuous (complex-valued) functions. For any $\epsilon>0$ and sufficiently large $T$,
\begin{align*}
\mathrm{P}\left(|\sqrt{T}(A_T(\Phi) - \mathrm{E}A_T(\Phi))| \geqslant \epsilon\right) \leqslant 96\exp\left(-\frac{1}{2} \left(\frac{\epsilon}{C_1 \max_{j_{1},j_{2} = 1,\cdots,D}\|\phi_{j_{2} j_{1}}\|_{\infty}\max\{1, M_T/T\}} \right)^{\frac{1}{2}}\right),
\end{align*}
where $C_1>0$ is a constant introduced in Proposition~\ref{higher_order_Expectation_prop}.
\begin{proof}
The proof of the inequality is similar to the proof of Lemma 2.3 in \citet{Dahlhaus1988}. Let
\begin{align*}
S = T^{1/2}(A_T(\Phi) - \mathrm{E}A_T(\Phi)).
\end{align*}
Then for any $\epsilon>0$ and $t>0$,
\begin{align*}
&\mathrm{P}\left(|S| \geqslant \epsilon\right)
=\mathrm{P}\left(\exp\left[t(|S|^{\frac{1}{2}} - \epsilon^{\frac{1}{2}})\right]\geqslant 1\right)
\leqslant \exp(-t \epsilon^{\frac{1}{2}})\mathrm{E}\exp(t |S|^{\frac{1}{2}}) \\
&=\exp(-t \epsilon^{\frac{1}{2}})\sum_{k=0}^{\infty}\frac{t^k}{k!}\mathrm{E}|S|^{\frac{k}{2}}
\end{align*}
With the help of Proposition~\ref{higher_order_Expectation_prop} (c), it can be shown that for any positive integer $l$ and $j\in\{1,2,3\}$,
\begin{align*}
\mathrm{E}|S|^{\frac{4l+j}{2}} \leqslant \left(\mathrm{E}S^{2l+2}\right)^{\frac{4l+j}{4l+4}}
\leqslant \left(\max_{j_{1},j_{2} = 1,\cdots,D}\|\phi_{j_{2} j_{1}}\|_{\infty}\max\{1, M_T/T\}C_1\right)^{\frac{4l+j}{2}} (4l+4)! .
\end{align*}
Therefore,
\begin{align*}
&\mathrm{P}\left(|S| \geqslant \epsilon\right)
\leqslant \exp(-t \epsilon^{\frac{1}{2}})\sum_{k=0}^{\infty} t^k\left(\max_{j_{1},j_{2} = 1,\cdots,D}\|\phi_{j_{2} j_{1}}\|_{\infty}\max\{1, M_T/T\}C_1\right)^{\frac{k}{2}}(k+1)(k+2)(k+3) .
\end{align*}
Choosing $t =(4\max_{j_{1},j_{2} = 1,\cdots,D}\|\phi_{j_{2} j_{1}}\|_{\infty}\max\{1, M_T/T\}C_1)^{-1/2}$, we then have
\begin{align*}
&\mathrm{P}\left(|S| \geqslant \epsilon\right)
\leqslant \exp\left(-(4C_1)^{-\frac{1}{2}}\left(\max_{j_{1},j_{2} = 1,\cdots,D}\|\phi_{j_{2} j_{1}}\|_{\infty}\right)^{-\frac{1}{2}} \max\{1, M_T/T\}^{-\frac{1}{2}}  \epsilon^{\frac{1}{2}}\right)\sum_{k=0}^{\infty} 2^{-k}(k+1)(k+2)(k+3) \\
&\leqslant 96\exp\left(-\frac{1}{2} \left(\frac{\epsilon}{C_1\max_{j_{1},j_{2} = 1,\cdots,D}\|\phi_{j_{2} j_{1}}\|_{\infty}\max\{1, M_T/T\}} \right)^{\frac{1}{2}}\right) .
\end{align*}
\end{proof}
\end{proposition}

\begin{theorem}[Maximal inequality]\label{stochastic_conv_rate_thm}
Let $\Theta\subset \mathbb{R}^d$ be a compact set (equipped with the Euclidean norm). For each $\theta\in\Theta$, $\Phi_{\theta}(\cdot)=\{\phi_{\theta,ab}(\cdot)\}_{a,b=1,\cdots,D}$ is a Hermitian matrix of continuous, bounded functions satisfying $\tau_0:=\sup_{\theta\in\Theta}\max_{a,b=1,\cdots,D}\|\phi_{\theta,ab}\|_{\infty}< \infty$. Moreover, there exist constants $C,\beta>0$ such that for any $\theta_1,~\theta_2\in\Theta$, $\max_{a,b=1,\cdots,D}\|\phi_{\theta_1, ab} - \phi_{\theta_2, ab}\|_{\infty} \leqslant C\|\theta_1 - \theta_2\|^{\beta}$.


Suppose $M_T/T \rightarrow L < \infty$. Then there exists $c_1,c_2,c_3>0$ such that for any $\eta,\tau_0>0$ and any sufficiently large $T$, we have
\begin{align*}
\mathrm{P}\left(\sup_{\theta\in\Theta}|T^{1/2}(A_T(\Phi_{\theta}) - \mathrm{E}A_T(\Phi_{\theta}))| > \eta, \mathcal{B}_T\right) \leqslant \frac{c_1}{\tau_0^{c_3}}\exp(-c_2\sqrt{\eta/\tau_0}),
\end{align*}
where $A_T(\cdot)$ is defined in Proposition~\ref{concentration_ineq_strong_prop}, and $\mathcal{B}_T$ is a sequence of sets (independent of $\Theta$) satisfying $\mathrm{P}(\mathcal{B}_T)\rightarrow 1$.
\begin{proof}
The proof is similar to that of \citet[Theorem 2.9]{Dahlhaus_and_Polonik2009}. Let $S_T(\Phi_{\theta}) = T^{1/2}(A_T(\Phi_{\theta}) - \mathrm{E}A_T(\Phi_{\theta}))$. Then Proposition~\ref{concentration_ineq_strong_prop}, in conjunction with the assumption $M_T/T \rightarrow L$, show that there exists constant $\tau_1>0$ such that for sufficiently large $T$,
\begin{align*}
\mathrm{P}\left(|S_T(\Phi_{\theta})| \geqslant \epsilon\right) \leqslant 96\exp\left(-\tau_1\sqrt{\frac{\epsilon}{\max_{a,b=1,\cdots,D}\|\phi_{\theta_,ab}\|_{\infty}}} \right),
\end{align*}
We will rely on chaining technique to prove the maximal inequality. First, define
\begin{align*}
\mathcal{B}_T = \left\{A_T(\boldsymbol{\mathrm{I}}_{D}) + \mathrm{E}A_T(\boldsymbol{\mathrm{I}}_{D}) \leqslant \frac{5}{2}\int_{0}^{2\pi L} \sum_{a=1}^{D}|f_{aa}(x)|\mathrm{d}x\right\},
\end{align*}
where $\boldsymbol{\mathrm{I}}_{D}$ is $D\times D$ identity matrix. Note that $A_T(\boldsymbol{\mathrm{I}}_{D}) = \sum_{p=1}^{M_T} \boldsymbol{J}_T^{H}(\omega_p)\boldsymbol{J}_T(\omega_p)/T$ is non-negative. Then Propositions~\ref{concentration_ineq_strong_prop}, \ref{cov_J_prop} and \ref{conv_Riemann_sum_prop} jointly imply $\mathrm{P}(\mathcal{B}_T)\rightarrow 1$. Let $\{\delta_j\}_{j\geqslant0}$ be a sequence satisfying $\delta_j>0$, $\delta_{j+1} \leqslant \delta_j/2$ for any $j\geqslant0$. So we know $\{\delta_j\}_{j\geqslant0}$ is strictly decreasing and $\delta_j\leqslant \delta_0/2^j$. Denote by $\mathcal{N}(r,\Theta,\|\cdot\|)$ the covering number of $\Theta$ with radius $r$. For each $\delta_j$, choose a set $\mathcal{A}_j\subset\Theta$ of cardinality $|\mathcal{A}_j|=\mathcal{N}(\delta_j,\Theta,\|\cdot\|)$ consisting of centers of Euclidean balls with radius $\delta_j$ such that these balls form a cover of $\Theta$. The for any $\theta\in\Theta$, we can always select $\theta_j$ from $\mathcal{A}_j$ such that $\|\theta-\theta_j\|_{\infty}<\delta_j$. For any fixed $T$, under $\mathcal{B}_T$, we have
\begin{align*}
&|S_T(\Phi_{\theta} - \Phi_{\theta_j})|
=|S_T(\Phi_{\theta}) - S_T(\Phi_j)| \\
&\leqslant T^{1/2} |A_T(\Phi_{\theta}) - A_T(\Phi_{\theta_j})| + T^{1/2} \mathrm{E}|A_T(\Phi_{\theta}) - A_T(\Phi_{\theta_j})| \\
&\leqslant T^{1/2}(A_T(\boldsymbol{\mathrm{I}}_{D}) + \mathrm{E}A_T(\boldsymbol{\mathrm{I}}_{D}))\max_{a,b=1,\cdots,D}\|\phi_{\theta, ab} - \phi_{\theta_j, ab}\|_{\infty} \\
&\leqslant \left(T^{1/2} \frac{5}{2}\int_{0}^{2\pi L} \sum_{a=1}^{D}|f_{aa}(x)|\mathrm{d}x \right)CD\|\theta - \theta_j\|^{\beta} \\
& \leqslant CD \left(T^{1/2} \frac{5}{2}\int_{0}^{2\pi L} \sum_{a=1}^{D}|f_{aa}(x)|\mathrm{d}x \right)\delta_j^{\beta}
\leqslant CD \left(T^{1/2} \frac{5}{2}\int_{0}^{2\pi L} \sum_{a=1}^{D}|f_{aa}(x)|\mathrm{d}x \right)\delta_0^{\beta} 2^{-\beta j}
\end{align*}
for any $j$. Note that although the choice of $\theta_j$ may depend on $\theta$, the rightmost-hand side of above inequalities is independent of $\theta$ and is summable with respect to $j$. Therefore, we can write $S_T(\Phi_{\theta})$ as a telescoping series
\begin{align*}
S_T(\Phi_{\theta}) = S_T(\Phi_{\theta_0}) + \sum_{j=0}^{\infty} S_T(\Phi_{\theta_{j+1}} - \Phi_{\theta_j}).
\end{align*}
Notice that the choice of $\{\Phi_{\theta_j}\}_{j\geqslant0}$ described above implies $\|\theta_{j+1} - \theta_j\|< \delta_{j+1} + \delta_j$. Thus,
\begin{align*}
&\mathrm{P}\left(\sup_{\theta\in\Theta}|S_T(\Phi_{\theta})|>\eta, B_T\right)
\leqslant \mathrm{P}\left(\sup_{\theta\in\Theta}|S_T(\Phi_{\theta_0})|>\frac{\eta}{2}\right) + \mathrm{P}\left(\sum_{j=0}^{\infty}\sup_{\theta\in\Theta}\left| S_T(\Phi_{\theta_{j+1}} - \Phi_{\theta_j})\right|>\frac{\eta}{2}\right) \\
&\leqslant \mathrm{P}\left(\sup_{\theta\in\Theta}|S_T(\Phi_{\theta_0})|>\frac{\eta}{2}\right) + \sum_{j=0}^{\infty}\mathrm{P}\left(\sup_{\theta\in\Theta}\left| S_T(\Phi_{\theta_{j+1}} - \Phi_{\theta_j})\right|>\frac{\eta}{2^{j+2}}\right) \\
&\leqslant \mathrm{P}\left(\max_{\theta_0\in\mathcal{A}_0}|S_T(\Phi_{\theta_0})|>\frac{\eta}{2}\right) + \sum_{j=0}^{\infty}\mathrm{P}\left(\max_{\substack{\theta_j\in\mathcal{A}_j, \theta_{j+1}\in\mathcal{A}_{j+1} \\ \|\theta_{j+1} - \theta_j\|< \delta_{j+1} + \delta_j}}\left| S_T(\Phi_{\theta_{j+1}} - \Phi_{\theta_j})\right|>\frac{\eta}{2^{j+2}}\right) \\
&\leqslant  |\mathcal{A}_0|\max_{\theta_0\in\mathcal{A}_0}\mathrm{P}\left(|S_T(\Phi_{\theta_0})|>\frac{\eta}{2}\right) 
+ \sum_{j=0}^{\infty}|\mathcal{A}_j||\mathcal{A}_{j+1}|\max_{\substack{\theta_j\in\mathcal{A}_j, \theta_{j+1}\in\mathcal{A}_{j+1} \\ \|\theta_{j+1} - \theta_j\|< \delta_{j+1} + \delta_j}}\mathrm{P}\left(\left| S_T(\Phi_{\theta_{j+1}} - \Phi_{\theta_j})\right|>\frac{\eta}{2^{j+2}}\right) \\
&\leqslant 96\exp\left[\log|\mathcal{A}_0|  - \tau_1  \frac{\eta^\frac{1}{2}}{\sqrt{2}}\left(\max_{\theta_0\in\mathcal{A}_0}\max_{a,b=1,\cdots,D}\|\phi_{\theta_0, ab}\|_{\infty}\right)^{-\frac{1}{2}}\right] \\
&+ 96\sum_{j=0}^{\infty}\exp\left[\log|\mathcal{A}_j| + \log|\mathcal{A}_{j+1}| - \tau_1 \frac{\eta^{\frac{1}{2}}}{2^{\frac{j+2}{2}}}\left(\max_{\substack{\theta_j\in\mathcal{A}_j, \theta_{j+1}\in\mathcal{A}_{j+1} \\\|\theta_{j+1} - \theta_j\|< \delta_{j+1} + \delta_j}}\max_{a,b=1,\cdots,D}\|\phi_{\theta_{j+1}, ab} - \phi_{\theta_{j}, ab}\|_{\infty}\right)^{-\frac{1}{2}} \right] \\
&\leqslant 96\exp\left(\log|\mathcal{A}_{0}| - \tau_0^{-\frac{1}{2}}\tau_1  \frac{\eta^{\frac{1}{2}}}{\sqrt{2}}\right)
+ 96\sum_{j=0}^{\infty}\exp\left(\log|\mathcal{A}_{j}|+\log|\mathcal{A}_{j+1}| -  C^{-\frac{1}{2}}\tau_1\frac{\eta^{\frac{1}{2}}}{2^{\frac{j+2}{2}}} (\delta_j+\delta_{j+1})^{-\frac{\beta}{2}}\right) \\
&=I + II.
\end{align*}
Since there exists $\tilde{C}>0$ such that $\log|\mathcal{A}_j|\leqslant \log \tilde{C} - d\log\delta_j$ according to equation (37) of~\citet{covnum2}. It is easy to see that 
\begin{align*}
I  \leqslant 96\tilde{C}\delta_0^{-d} \exp(-2^{-\frac{1}{2}}\tau_1 \sqrt{\eta/\tau_0}).
\end{align*}
As for $II$, we need to design $\{\delta_j\}_{j\geqslant1}$ so that the desired result can be achieved. In the following derivation, we use $c$ to represent generic positive constants independent of $j$, $\eta$ and $\tau_0$. So if there are multiple $c$ in the same equation, they might be different. Let $\delta_0 = \tau_0^{\frac{1}{\beta}}$, $\delta_j = 2^{-\frac{4}{\beta}j}\delta_0$. For $\eta>0$,
\begin{align*}
&\exp\left(\log|\mathcal{A}_{j}|+\log|\mathcal{A}_{j+1}| -  C^{-\frac{1}{2}}\tau_1\frac{\eta^{\frac{1}{2}}}{2^{\frac{j+2}{2}}} (\delta_j+\delta_{j+1})^{-\frac{\beta}{2}}\right) \\
&\leqslant \exp\left(2\log|\mathcal{A}_{j+1}| -  C^{-\frac{1}{2}}\tau_1\frac{\eta^{\frac{1}{2}}}{2^{\frac{j+2}{2}}} (\frac{3}{2}\delta_j)^{-\frac{\beta}{2}}\right) 
\leqslant c\exp\left(-c\log\delta_{j+1} - c \frac{\eta^{\frac{1}{2}}}{2^{\frac{j}{2}}} \delta_j^{-\frac{\beta}{2}}\right) \\
&\leqslant \frac{c}{\tau_0^c}\exp\left(cj - c \sqrt{\eta/\tau_0} 2^{\frac{3}{2}j} \right)
\leqslant \frac{c}{\tau_0^c}\exp\left(- c \sqrt{\eta/\tau_0} 2^{\frac{j}{2}}\right)
\leqslant \frac{c}{\tau_0^c}\exp\left[- c \eta^{\frac{1}{2}}(j+2)\right].
\end{align*}
Therefore,
\begin{align*}
II \leqslant \frac{c}{\tau_0^c}\sum_{j=0}^{\infty}\exp[- c \sqrt{\eta/\tau_0}(j+2)]
=\frac{c}{\tau_0^c}\frac{\exp(- 2c \sqrt{\eta/\tau_0})}{1 - \exp(- c \sqrt{\eta/\tau_0}) }
\leqslant \frac{c}{\tau_0^c}\exp(- c \sqrt{\eta/\tau_0}) .
\end{align*}
Finally, it is clear that there exists $c_1,c_2, c_3>0$ such that
\begin{align*}
I+II \leqslant \frac{c_1}{\tau_0^{c_3}}\exp(-c_2\sqrt{\eta/\tau_0}).
\end{align*}
\end{proof}
\end{theorem}

\subsection{Proof of Theorem~\ref{Thm_CLT}}\label{proof_CLT_sect}
Let $S_T(\Phi_{\cdot}) = T^{1/2}(A_T(\Phi_{\cdot}) - \mathrm{E}A_T(\Phi_{\cdot}))$. It is worth noting that $S_T(\Phi_{\cdot})$ is real due to the fact that $\Phi_{\cdot}$ is Hermitian. According to Theorems 1.5.4 and 1.5.7 of \citet{van_der_vaart_and_wellner1996}, in order to show the weak convergence of $S_T(\Phi_{\cdot})$, we need to establish weak convergence of the finite-dimensional distributions and asymptotic equicontinuity in probability of $S_T(\Phi_{\cdot})$. Convergence of the finite-dimensional distributions can be seen from Propositions~\ref{cov_sep_prop} and \ref{higher_order_Expectation_prop}, and Lemma P4.5 of \citet{Brillinger2001}. The covariance structure can be obtained from Proposition~\ref{cov_sep_prop} (b). As for the asymptotic
equicontinuity, we need to show that for any $\eta,\epsilon>0$, there exists a $\tau>0$, such that
\begin{align*}
\limsup_{T\rightarrow\infty}\mathrm{P}\left(\sup_{\|\theta_1 - \theta_2\| < \tau}|S_T(\Phi_{\theta_1}) - S_T(\Phi_{\theta_2})|>\eta\right)<\epsilon .    
\end{align*}
To this end, we apply Proposition~\ref{stochastic_conv_rate_thm}. Let $\tilde{\Theta}_{\tau} = \{(\theta_1,\theta_2)\in\Theta\times\Theta:\|\theta_1 - \theta_2\| < \tau\}$ and $\tilde{\Phi}_{(\theta_1,\theta_2)} = \Phi_{\theta_1} - \Phi_{\theta_2}$  Note that $\tilde{\Phi}_{\cdot}$ is also Hermitian satisfying the assumptions in Theorem~\ref{stochastic_conv_rate_thm}. 
Therefore, we have
\begin{align*}
&\limsup_{T\rightarrow\infty}\mathrm{P}\left(\sup_{\|\theta_1 - \theta_2\| < \tau}|S_T(\Phi_{\theta_1}) - S_T(\Phi_{\theta_2})|>\eta\right) \\
&\leqslant \limsup_{T\rightarrow\infty}\mathrm{P}\left(\sup_{(\theta_1,\theta_2)\in\tilde{\Theta}_{\tau}}|S_T(\tilde{\Phi}_{(\theta_1, \theta_2)})|>\eta, \mathcal{B}_T\right) + \lim_{T\rightarrow\infty}\mathrm{P}(\mathcal{B}_T^c) \\
&\leqslant \frac{\tilde{c}_1}{\tau^{\tilde{c}_3}}\exp(-\tilde{c}_2\sqrt{\eta/\tau}),
\end{align*}
where $\tilde{c}_1,\tilde{c}_2,\tilde{c}_3>0$ are constants. Apparently, we can choose sufficiently small $\tau$ such that the right-hand side of the inequality is smaller than $\epsilon$.

\subsection{Proof of Theorem~\ref{Thm_McP} and related results}
\subsubsection{Proof of Theorem~\ref{Thm_McP}}\label{proof_Thm_McP_sect}
Similar to \citet{McElroy_and_Politis2025}, let $c_{2l,T} = M_T^{-1}\sum_{s=1}^{M_T} K_{\delta_T}(\omega_s)\omega_s^{2l}$, $l=0,1,2$, and $g_T(\omega) = (c_{4,T} - c_{2,T}\omega^2)K_{\delta_T}(\omega)/(c_{4,T}c_{0,T} - c_{2,T}^2)$. Then we can write 
\begin{align*}
\hat{\phi}_{uv} = \frac{1}{M_T}\sum_{s=1}^{M_T}g_T(\omega_s)\mathrm{Re}J_{T}^{u}(\omega_s)J_{T}^{v}(-\omega_s)
=\frac{\xi_{uv} + \bar{\xi}_{uv}}{2}
=\frac{\xi_{uv} + \xi_{vu}}{2}.
\end{align*}
where $\xi_{uv}=(1/M_T)\sum_{s=1}^{M_T}g_T(\omega_s)J_{T}^{u}(\omega_s)J_{T}^{v}(-\omega_s)$, and $\bar{\xi}_{uv}$ is the complex conjugate of $\xi_{uv}$. To further analyze $\hat{\phi}_{uv}$, we need to derive several properties of $g_T$. An obvious one is $\sum_{s=1}^{M_T}g_T(\omega_s)=1$. Let $H_{2l,T} = M_T^{-1}\sum_{s=1}^{M_T} K(s/M_T)(s/M_T)^{2l}$, $l=0,1,2$. Then we have $c_{2l,T} = (2\pi\delta_T)^{2l-1}H_{2l,T}$ and $H_{2l,T}\rightarrow H_{2l}:=\int_0^1 K(x)x^{2l}\mathrm{d}x$ when $T\rightarrow\infty$. Moreover, $H_{4,T}H_{0,T} - H_{2,T}^2>0$ and $H_{4}H_{0} - H_{2}^2>0$ due to the Cauchy-Schwarz inequality. Consequently, there exists $C_1=2(H_4 + H_2)\|K\|_{\infty}/(H_4H_0 - H_2^2)$ such that
\begin{align}\label{eq_bound_gT}
\sup_{\omega\in[0,2\pi\delta_T]}|g_T(\omega)|
\leqslant \frac{c_{4,T} + c_{2,T}(2\pi\delta_T)^2}{c_{4,T}c_{0,T} - c_{2,T}^2}\frac{1}{2\pi\delta_T}\|K\|_{\infty}
= \frac{H_{4,T} + H_{2,T}}{H_{4,T}H_{0,T} - H_{2,T}^2}\|K\|_{\infty}
\leqslant C_1
\end{align}
holds for sufficiently large $T$.

We will now prove part (a) of the theorem by analyzing the first- and second-order properties of $\{\hat{\phi}_{uv}\}_{u,v=1,\cdots,D}$. Notice that $f_{uv}(0)=f_{vu}(0)$, we have by Corollary~\ref{cov_J_corollary} that
\begin{align*}
&|\mathrm{E}\hat{\phi}_{uv} - f_{uv}(0)|
=\left|\frac{1}{M_T}\sum_{s=1}^{M_T}g_T(\omega_s)\left(\frac{f_{uv}(\omega_s) + f_{vu}(-\omega_s)}{2} + o(1)\right) - f_{uv}(0)\right| \\
&\leqslant \frac{C_1}{2}\left(\sup_{\omega\in[0,2\pi\delta_T]}|f_{uv}(\omega) - f_{uv}(0)| + \sup_{\omega\in[0,2\pi\delta_T]}|f_{vu}(-\omega) - f_{uv}(0)|\right) + o(1)
=o(1)
\end{align*}
as both $f_{uv}$ and $f_{vu}$ are uniformly continuous and $\delta_T\rightarrow0$. To prove the convergence in probability, it suffices to show 
\begin{align*}
\mathrm{var}(\hat{\phi}_{uv})
=\frac{1}{4}\mathrm{cum}(\xi_{uv}, \xi_{uv})
+\frac{1}{4}\mathrm{cum}(\bar{\xi}_{uv}, \bar{\xi}_{uv}) 
+\frac{1}{2}\mathrm{cum}(\xi_{uv}, \bar{\xi}_{uv})\rightarrow 0
\end{align*}
when $T\rightarrow\infty$ for any $u,v=1,\cdots,D$. In view of Equation~\eqref{eq_bound_gT} and the fact that $\xi_{uv}=(T/M_T)(1/T)\sum_{s=1}^{M_T}g_T(\omega_s)J_{T}^{u}(\omega_s)J_{T}^{v}(-\omega_s)$, we deduce
\begin{align*}
0\leqslant\mathrm{var}(\hat{\phi}_{uv})
\leqslant(T/M_T)^2 (O(M_T^2/T^3) 
+O(M_T/T^2) ) = O(1/T) + O(1/M_T) = o(1)
\end{align*}
in a manner similar to the proof of Proposition~\ref{cov_sep_prop} (a). 

To prove part (b) of the theorem, we first notice that the subprocesses of the underlying multivariate stationary Hawkes process are jointly independent due to the diagonality of the interactions matrix. Therefore, for any $1\leqslant u<v\leqslant D$,
\begin{align*}
\mathrm{E}\xi_{uv} = \frac{1}{M_T}\sum_{s=1}^{M_T}g_T(\omega_s)\mathrm{E}[J_{T}^{u}(\omega_s)J_{T}^{v}(-\omega_s)]
=\frac{1}{M_T}\sum_{s=1}^{M_T}g_T(\omega_s)\mathrm{E}[J_{T}^{u}(\omega_s)]\mathrm{E}[J_{T}^{v}(-\omega_s)]
=0.
\end{align*}
This implies that $\mathrm{E}\sqrt{M_T}\hat{\phi}_{uv} = 0$, $1\leqslant u<v\leqslant D$. We now study the second-order properties. Let $1\leqslant u_1<v_1\leqslant D$, $1\leqslant u_2<v_2\leqslant D$. If the sets $\{u_1,v_1\}$ and $\{u_2,v_2\}$ are not identical, then we have
\begin{align*}
&\mathrm{cum}(\sqrt{M_T}\xi_{u_1v_1}, \sqrt{M_T}\xi_{u_2v_2}) \\
&= \frac{1}{M_T}\sum_{s_1=1}^{M_T}\sum_{s_2=1}^{M_T}g_T(\omega_{s_1})g_T(\omega_{s_2})\mathrm{cum}(J_{T}^{u_1}(\omega_{s_1})J_{T}^{v_1}(-\omega_{s_1}), J_{T}^{u_2}(\omega_{s_2})J_{T}^{v_2}(-\omega_{s_2}))
=0,
\end{align*}
due to the independence of subprocesses, and $\mathrm{cum}(\sqrt{M_T}\xi_{u_1v_1}, \sqrt{M_T}\bar{\xi}_{u_2v_2})=0$ similarly. Therefore, $\mathrm{cov}(\sqrt{M_T}\hat{\phi}_{u_1v_1}, \sqrt{M_T}\hat{\phi}_{u_2v_2})=0$ when $\{u_1,v_1\} \neq \{u_2,v_2\}$. We now consider the case when $u_1=u_2=u$ and $v_1=v_2=v$. Notice that 
\begin{align*}
\frac{1}{M_T}\sum_{s=1}^{M_T}g_T^2(\omega_{s})
=\frac{1}{M_T}\sum_{s=1}^{M_T} \left(\frac{(H_{4,T} - H_{2,T}(s/M_T)^2)K(s/M_T)}{H_{4,T}H_{0,T} - H_{2,T}^2}\right)^2 
\rightarrow \int_{0}^{1} \left(\frac{(H_4 - H_2 x^2)K(x)}{H_4H_0 - H_2^2}\right)^2\mathrm{d}x
\end{align*}
and
\begin{align*}
\frac{1}{M_T}\sum_{s=1}^{M_T}g_T^2(\omega_{s}) f_{u u}(\pm\omega_s)f_{v v}(\mp\omega_s)
\rightarrow f_{u u}(0)f_{v v}(0)\int_{0}^{1} \left(\frac{(H_4 - H_2 x^2)K(x)}{H_4H_0 - H_2^2}\right)^2\mathrm{d}x,
\end{align*} 
then similar to the proof of Proposition~\ref{cov_sep_prop} (b), we get
\begin{align*}
&\mathrm{cum}(\sqrt{M_T}\xi_{uv}, \sqrt{M_T}\bar{\xi}_{uv}) \\
&= \frac{1}{M_T}\sum_{s_1=1}^{M_T}\sum_{s_2=1}^{M_T}g_T(\omega_{s_1})g_T(\omega_{s_2})\mathrm{cum}(J_{T}^{u}(\omega_{s_1})J_{T}^{v}(-\omega_{s_1}), J_{T}^{u}(-\omega_{s_2})J_{T}^{v}(\omega_{s_2})) \\
&= \frac{1}{M_T}\sum_{s_1=1}^{M_T}\sum_{s_2=1}^{M_T}g_T(\omega_{s_1})g_T(\omega_{s_2})\mathrm{cum}(J_{T}^{u}(\omega_{s_1}), J_{T}^{u}(-\omega_{s_2}))\mathrm{cum}(J_{T}^{v}(-\omega_{s_1}), J_{T}^{v}(\omega_{s_2})) \\
&\rightarrow f_{uu}(0)f_{vv}(0)\int_{0}^{1} \left(\frac{(H_4 - H_2 x^2)K(x)}{H_4H_0 - H_2^2}\right)^2\mathrm{d}x
\end{align*}
and
\begin{align*}
&\mathrm{cum}(\sqrt{M_T}\xi_{uv}, \sqrt{M_T}\xi_{uv}) \\
&= \frac{1}{M_T}\sum_{s_1=1}^{M_T}\sum_{s_2=1}^{M_T}g_T(\omega_{s_1})g_T(\omega_{s_2})\mathrm{cum}(J_{T}^{u}(\omega_{s_1})J_{T}^{v}(-\omega_{s_1}), J_{T}^{u}(\omega_{s_2})J_{T}^{v}(-\omega_{s_2})) \\
&= \frac{1}{M_T}\sum_{s_1=1}^{M_T}\sum_{s_2=1}^{M_T}g_T(\omega_{s_1})g_T(\omega_{s_2})\mathrm{cum}(J_{T}^{u}(\omega_{s_1}), J_{T}^{u}(\omega_{s_2}))\mathrm{cum}(J_{T}^{v}(-\omega_{s_1}), J_{T}^{v}(-\omega_{s_2})) \\
&=o(1) .
\end{align*}
These results lead to 
\begin{align*}
\mathrm{var}(\sqrt{M_T}\hat{\phi}_{uv})
\rightarrow \frac{1}{2}f_{uu}(0)f_{vv}(0)\int_{0}^{1} \left(\frac{(H_4 - H_2 x^2)K(x)}{H_4H_0 - H_2^2}\right)^2\mathrm{d}x 
\end{align*}
for any $1\leqslant u<v\leqslant D$. As for higher order mixed cumulant, we have for any $1\leqslant u_k<v_k\leqslant D$, $k=1,\cdots,l$, $l\geqslant3$,
\begin{align*}
\mathrm{cum}(\sqrt{M_T}\hat{\phi}_{u_1v_1},\cdots,\sqrt{M_T}\hat{\phi}_{u_lv_l}) = o(1)
\end{align*}
when $T\rightarrow\infty$ and the proof can be carried out in a manner resembling that of Proposition~\ref{higher_order_Expectation_prop}~(a) in view of Equation~\eqref{ineq_he1}. The asymptotic normality can be proved by combining the above results regarding cumulants with Lemma P4.5 of \citet{Brillinger2001}.

\subsubsection{Unique solution of the matrix equation}\label{proof_uniq_mat_eq_sect}
\begin{proposition}\label{uniq_mat_eq_proposition}
Let $\boldsymbol{\nu}$ be an element-wise nonnegative matrix with spectral radius smaller than 1 satisfying 
\begin{align*}
(\boldsymbol{\mathrm{I}}_D - \boldsymbol{\nu})^{-1}\boldsymbol{\Lambda}[(\boldsymbol{\mathrm{I}}_D - \boldsymbol{\nu})^{-1}]^{\top} = \boldsymbol{D},
\end{align*}
where $\boldsymbol{\Lambda}$ and $\boldsymbol{D}$ are diagonal matrices with positive diagonal elements, and $[\boldsymbol{\Lambda}]_{ii}/[\boldsymbol{D}]_{ii} \leqslant 1$ for $i=1,\cdots,D$. Then $\boldsymbol{\nu}$ has to be a diagonal matrix.
\begin{proof}
The given equation is equivalent to $\boldsymbol{Q}\boldsymbol{Q}^{\top} = \boldsymbol{\mathrm{I}}_D$, where $\boldsymbol{Q}=\boldsymbol{D}^{-\frac{1}{2}}(\boldsymbol{\mathrm{I}}_D - \boldsymbol{\nu})^{-1}\boldsymbol{\Lambda}^{\frac{1}{2}}$. This indicates that $\boldsymbol{Q}$ is an orthogonal matrix. Since $(\boldsymbol{\mathrm{I}}_D - \boldsymbol{\nu})^{-1}=\sum_{n\geqslant0}\boldsymbol{\nu}^n$ is a nonnegative matrix, $\boldsymbol{Q}$ consequently should be nonnegative. It is known that a nonnegative orthogonal matrix has to be a permutation matrix (i.e. a matrix having exactly one entry of 1 in each row and each column with all other entries 0). If $\boldsymbol{Q}$ is a permutation matrix other than the identity matrix, we observe that $\boldsymbol{\nu} = \boldsymbol{\mathrm{I}}_D - \boldsymbol{\Lambda}^{\frac{1}{2}}\boldsymbol{Q}^{\top}\boldsymbol{D}^{-\frac{1}{2}}$ has at least one negative off-diagonal entry, contradicting the nonnegativity of $\boldsymbol{\nu}$. Therefore, we have $\boldsymbol{Q}=\boldsymbol{\mathrm{I}}_D$ and $\boldsymbol{\nu} = \boldsymbol{\mathrm{I}}_D - \boldsymbol{\Lambda}^{\frac{1}{2}}\boldsymbol{D}^{-\frac{1}{2}}$ is a nonnegative diagonal matrix.
\end{proof}
\end{proposition}

\section{Indecomposable partitions}\label{indecomposable_partition_sect}

    This appendix formalizes the ``indecomposable” (equivalently, ``hooked”) condition used in Appendix~\ref{app_sect_weak_convergence} to glue the signs in the constraints
\(\sum_{p_i\in Q_j} p_i-\sum_{-p_i\in Q_j} p_i=0\).
We recast a partition \(\{Q_1,\dots,Q_r\}\) of the two-column table \(\{p_i,-p_i\}_{i=1}^l\) as a labelled multigraph on the blocks \(Q_j\) and observe that the constraint matrix \(A\) is, up to column sign choices, an oriented incidence matrix of this graph. This viewpoint yields a clean rank formula and clarifies why indecomposability corresponds to connectivity cf. \citet[Theorem 2.3.2]{Brillinger2001}.

If the associated graph has \(c\) connected components (loops ignored), then \(\mathrm{rank}(A)=r-c\); in particular, an indecomposable partition (connected graph) gives \(\mathrm{rank}(A)=r-1\).

\noindent Fix an integer $l\ge 1$. Consider the two-column table whose $i$th row is
\[
R_i=\{ p_i, -p_i \}, \qquad i=1,\dots,l,
\]
and let $\{Q_1,\dots,Q_r\}$ be a partition of 
\begin{align*}
\begin{matrix}
&p_1&-p_1 \\
&\vdots&\vdots \\
&p_l&-p_l
\end{matrix}
\end{align*}
into $r\ge 1$ nonempty blocks. Following the main text, each block $Q_j$ induces a linear constraint
\[
P_j:\quad \sum_{p_i\in Q_j} p_i \;-\!\!\! \sum_{-p_i\in Q_j} p_i \;=\; 0 ,
\]
and we collect these constraints in a matrix $A\in\mathbb{R}^{r\times l}$ via
\[
A_{j,i} \;=\;
\begin{cases}
\phantom{-}1, & p_i\in Q_j,\\
-1, & -p_i\in Q_j,\\
\phantom{-}0, & \text{otherwise}.
\end{cases}
\]
Thus column $i$ has a $+1$ at the row of the block holding $p_i$ and a $-1$ at the row of the block holding $-p_i$. If $p_i$ and $-p_i$ lie in the same block, the $i$th column is the zero column.

We can encode the partition by a multigraph $G$ on the vertex set $\{1,\dots,r\}$. Specifically,
for each $i\in\{1,\dots,l\}$, if $p_i\in Q_a$ and $-p_i\in Q_b$ with $a\ne b$, add an (undirected) edge between $a$ and $b$ labelled $i$, and if $p_i,-p_i\in Q_a$, add a self-loop at vertex $a$ labelled $i$. With an arbitrary orientation chosen (recall Boris' observation about multiplying by $-1$) for each nonloop edge, $A$ is  an oriented incidence matrix of $G$. Loops correspond to zero columns and do not affect rank or connectivity.

\begin{definition}[Indecomposable partition]
We say $\{Q_1,\dots,Q_r\}$ is \emph{indecomposable} if there is no strict subfamily $\{Q_{j_1},\dots,Q_{j_s}\}$, $1\le s<r$, whose union equals a union of whole rows $\bigcup_{i\in S} R_i$ for some $S\subset\{1,\dots,l\}$. 
\end{definition}

\noindent A nonloop edge labelled $i$ says: $p_i$ and $-p_i$ live in different blocks and therefore column $i$ of $A$ is $e_a-e_b$ (up to sign).
A loop labelled $i$ means $p_i$ and $-p_i$ live in the same block and therefore column $i$ of $A$ is the zero vector.

Lemma~\ref{lem:indec-connected} identifies indecomposability with graph connectivity; Lemma~\ref{lem:kernel} describes the kernel of \(A^\top\); Theorem~\ref{thm:rank} then gives \(\mathrm{rank}(A)=r-c\).

\begin{lemma}[Indecomposable $\iff$ connected]\label{lem:indec-connected}
Let $G$ be the graph defined above (ignore loops when talking about connectivity). Then the partition is indecomposable if and only if $G$ is connected.
\end{lemma}

\begin{proof}
\emph{($\Rightarrow$)} Suppose, towards a contradiction, that $G$ is disconnected. Take the vertex set $V_1$ of any nonempty proper connected component. Because edges never cross between components, every row $R_i=\{p_i,-p_i\}$ is either entirely supported on blocks indexed by $V_1$ or entirely on blocks indexed by the complement $V_1^c$. Therefore
\[
\bigcup_{j\in V_1} Q_j \;=\; \bigcup_{i: R_i\subset \cup_{j\in V_1}Q_j} R_i
\]
is a union of whole rows, contradicting indecomposability.\\

\emph{($\Leftarrow$)} Conversely, suppose some strict subfamily $\{Q_{j_1},\dots,Q_{j_s}\}$ equals a union of whole rows. Then no row $R_i$ is split across this subfamily and its complement; equivalently, there is no edge in $G$ joining $\{j_1,\dots,j_s\}$ to its complement. Hence $G$ is disconnected.
\end{proof}

\begin{lemma}[Kernel description]\label{lem:kernel}
Let $c$ be the number of connected components of $G$ (again, loops ignored). Then
\[
\ker(A^\top)=\left\{x\in\mathbb{R}^r:\ x \text{ is constant on each connected component of }G\right\},
\]
so $\dim\ker(A^\top)=c$.
\end{lemma}

\begin{proof}
Every nonloop column of $A$ has the form $e_a-e_b$ for some edge $\{a,b\}$, so $A^\top x=0$ forces $x_a=x_b$ on every edge. Moving along paths, $x$ must be constant on each component. Loops give zero columns and impose no constraint. The converse is immediate: if $x$ is componentwise constant, then $x_a-x_b=0$ on every edge, hence $A^\top x=0$.
\end{proof}

\begin{theorem}[Rank formula]\label{thm:rank}
Let $c$ be the number of connected components of $G$. Then
\[
\mathrm{rank}(A)=r-c.
\]
In particular, for an indecomposable partition (i.e.\ a connected graph), $\mathrm{rank}(A)=r-1$.
\end{theorem}

\begin{proof}
By Lemma~\ref{lem:kernel} and rank--nullity,
\(
\mathrm{rank}(A)=r-\dim\ker(A^\top)=r-c.
\)
\end{proof}

\begin{remark}
The identity $\mathrm{rank}(A)=r-c$ for an oriented incidence matrix is standard in algebraic graph theory; cf. Theorem 8.3.1 of \citet{godsil2013algebraic}. For the $0$--$1$ (unoriented) incidence matrix $B$, the rank is $r-c_0$, where $c_0$ is the number of bipartite components; see Thm.~8.2.1 of the same text. Loops contribute zero columns in either model and do not affect the rank or connectivity.
\end{remark}

\begin{example}[Example A (indecomposable; $r=2$)]
\[
Q_1=\{p_1,-p_1,p_2,p_3\},\qquad Q_2=\{-p_2,-p_3\}.
\]
Edges: $i=2$ and $i=3$ join $Q_2$ to $Q_1$; $i=1$ is a loop at $Q_1$.
\[
A=\begin{pmatrix}
0 & 1 & 1\\
0 & -1 & -1
\end{pmatrix},\qquad \mathrm{rank}(A)=1=r-c \ (c=1).
\]
\begin{center}
\begin{tikzpicture}[node distance=28mm]
\node[vtx] (Q1) {$Q_1$};
\node[vtx, right=of Q1] (Q2) {$Q_2$};
\path (Q1) edge[loop above, min distance=12mm] node[above] {$i=1$} (Q1);
\draw[bend left=18] (Q2) to node[above] {$i=2$} (Q1);
\draw[bend right=18] (Q2) to node[below] {$i=3$} (Q1);
\end{tikzpicture}
\end{center}
\end{example}

A key observation from the above example is that for any block of an indecomposable partition (say $Q_1$), the linear constraint associated with the block (e.g. $\sum_{i\in Q_1}i = k$ for some $k$) will restrict exactly one variable. This observation will help us prove the following proposition which is important to the proof of asymptotic normality.
\begin{proposition}\label{indecomp_sum_corollary}
Let $Q_1,\cdots,Q_r$ be an indecomposable partition of of the following table
\begin{align*}
\begin{matrix}
&p_1&-p_1 \\
&\vdots&\vdots \\
&p_l&-p_l
\end{matrix}
\end{align*}
with $l\geqslant 2$ and $|Q_j|\geqslant2$, $j=1,\cdots,r$. Let $M_T$ be a sufficiently large positive integer.
\begin{enumerate}
\item When $l=2$, we have
\begin{align*}
\sum_{p_1=1}^{M_T}\sum_{p_2=1}^{M_T}\frac{1}{|p_1 - p_2||-p_1 + p_2|}\mathds{1}_{\{p_1 - p_2 \neq 0\}\cap\{-p_1 + p_2 \neq 0\}}
\leqslant 4 M_T
\end{align*}
and
\begin{align*}
\sum_{p_1=1}^{M_T}\sum_{p_2=1}^{M_T}\frac{1}{|p_1 + p_2||-p_1 - p_2|}\mathds{1}_{\{p_1 + p_2 \neq 0\}\cap\{-p_1 - p_2 \neq 0\}}
\leqslant 4 \log M_T .
\end{align*}
\item When $r\geqslant2$,
\begin{align*}
\sum_{p_1=1}^{M_T}\cdots\sum_{p_l=1}^{M_T}\mathds{1}_{\cap_{j=1}^r\{\sum_{i\in Q_j} i = 0\}}
\leqslant M_T^{l-r+1} .
\end{align*}
\item Suppose $l\geqslant3$. If $r=2$,
\begin{align*}
\sum_{p_1=1}^{M_T}\cdots\sum_{p_l=1}^{M_T}\frac{1}{|\sum_{i\in Q_1} i||\sum_{i\in Q_2} i|}\mathds{1}_{\cap_{j=1}^2\{\sum_{i\in Q_j} i \neq 0\}}
\leqslant 4 M_T^{l-1} .
\end{align*}
If $r\geqslant3$,
\begin{align*}
\sum_{p_1=1}^{M_T}\cdots\sum_{p_l=1}^{M_T}\frac{1}{|\sum_{i\in Q_1} i|\cdots|\sum_{i\in Q_r} i|}\mathds{1}_{\cap_{j=1}^r\{\sum_{i\in Q_j} i \neq 0\}}
\leqslant 4^{r-1} M_T^{l-r+1}(\log M_T)^{r-1} .
\end{align*}
\item When $l\geqslant 3$ and $1\leqslant q\leqslant r-1$,
\begin{align*}
&\sum_{p_1=1}^{M_T}\cdots\sum_{p_l=1}^{M_T}\frac{1}{|\sum_{i\in Q_1} i|\cdots|\sum_{i\in Q_q} i|}\mathds{1}_{\cap_{j=1}^q\{\sum_{i\in Q_j} i \neq 0\}\cap\cap_{j=q+1}^{r}\{\sum_{i\in Q_j} i = 0\}} \\
&\leqslant \left\{
    \begin {aligned}
         & 0 \quad & q = 1 \\
         & 4 M_T^{l-r+1} \quad & q = 2  \\
         & 4^{q-1} M_T^{l-r+1}(\log M_T)^{q-1} \quad & 3 \leqslant q \leqslant r-1
    \end{aligned}
\right. .
\end{align*}

\end{enumerate}
\begin{proof}
We first introduce some notation. Let $\|\{\sum_{i\in Q}i = k\}\|:=\sum_{p_1=1}^{M_T}\cdots\sum_{p_l=1}^{M_T}\mathds{1}_{\{\sum_{i\in Q} i = k\}}$. Then $\|\{\sum_{i\in Q}i = k\}\|$ is the number of solutions of $\sum_{i\in Q}i = k$ as an equation of $p_1,\cdots,p_l$ with $p_j=1,\cdots,M_T$, $j=1,\cdots,l$. For example, $\|\{p_1 - p_2 = 0\}\| = M_T$ and $\|\{p_1 + p_2 + p_3 = 3\}\| = 1$. It should be kept in mind that $\|\{\sum_{i\in Q}i = k\}\|$ is not a function of $p_1,\cdots,p_l$.

For the proof of part 1, by direct calculations,
\begin{align*}
&\sum_{p_1=1}^{M_T}\sum_{p_2=1}^{M_T}\frac{1}{|p_1 - p_2||-p_1 + p_2|}\mathds{1}_{\{p_1 - p_2 \neq 0\}\cap\{-p_1 + p_2 \neq 0\}}
\
=\sum_{k=1}^{M_T-1} \frac{\|\{p_1 - p_2 = k\}\| + \|\{p_1 - p_2 = -k\}\|}{k^2}  \\
&=\sum_{k=1}^{M_T-1}\frac{2(M_T - k)}{k^2}
\leqslant 4 M_T
\end{align*}
and
\begin{align*}
&\sum_{p_1=1}^{M_T}\sum_{p_2=1}^{M_T}\frac{1}{|p_1 + p_2||-p_1 - p_2|}\mathds{1}_{\{p_1 + p_2 \neq 0\}\cap\{-p_1 - p_2 \neq 0\}}
=\sum_{k=2}^{2M_T} \frac{\|\{p_1 + p_2 = k\}\|}{k^2}  \\
&\leqslant\sum_{k=2}^{2M_T}\frac{{k-1\choose 1}}{k^2}
\leqslant 4 \log M_T .
\end{align*}

To prove part 2, note that in view of the fact that $p_j\in\{1,\cdots, M_T\}$ for any $j=1,\cdots,r$ and Theorem~\ref{thm:rank}, either $\cap_{j=1}^r\{\sum_{i\in Q_j} i = 0\} = \emptyset$ or the system of equations $\sum_{i\in Q_j} i = 0$, $j=1,\cdots,r$ contains exactly $r-1$ distinct constraints. This leads to $\sum_{p_1=1}^{M_T}\cdots\sum_{p_l=1}^{M_T}\mathds{1}_{\cap_{j=1}^r\{\sum_{i\in Q_j} i = 0\}}
\leqslant M_T^{l-r+1}$.

Regarding part 3, when $r=2$, note that since $Q_1\cap Q_2 = \emptyset$ and $Q_1\cup Q_2$ is the whole table, we have $\sum_{i\in Q_1} i + \sum_{i\in Q_2} i = 0$ (i.e., the linear constraints induced by $Q_1$ and $Q_2$ are identical). Therefore,
\begin{align*}
&\sum_{p_1=1}^{M_T}\cdots\sum_{p_l=1}^{M_T}\frac{1}{|\sum_{i\in Q_1} i| |\sum_{i\in Q_2} i|}\mathds{1}_{\{\sum_{i\in Q_1} i \neq 0\}\cap\{\sum_{i\in Q_2} i \neq 0\}}
=\sum_{p_1=1}^{M_T}\cdots\sum_{p_l=1}^{M_T}\frac{1}{|\sum_{i\in Q_1} i|^2}\mathds{1}_{\{\sum_{i\in Q_1} i \neq 0\}} \\
&\leqslant\sum_{k=1}^{|Q_1| M_T}\frac{\|\{\sum_{i\in Q_1} i = k\}\| + \|\{\sum_{i\in Q_1} i = -k\}\|}{k^2}.
\end{align*}
Notice that due to the indecomposability of the partition, $Q_1$ confines exactly one variable via $\sum_{i\in Q_1} i = \pm k$. Therefore, $\|\{\sum_{i\in Q_1} i = k\}\|,~\|\{\sum_{i\in Q_1} i = -k\}\|\leqslant M_T^{l-1}$ and consequently, $\sum_{k=1}^{|Q_1| M_T}(\|\{\sum_{i\in Q_1} i = k\}\| + \|\{\sum_{i\in Q_1} i = -k\}\|)/k^2\leqslant \sum_{k}2M_T^{l-1}/k^2 \leqslant 4 M_T^{l-1}$. When $r\geqslant 3$, it is worth noting that unlike the case of $r=2$, the blocks of the partition will no longer lead to identical constraints (If there exist $Q_i$ and $Q_j$ leading to the same constraint, then $Q_i$ and $Q_j$ together form a partition of some rows, contradicting the indecomposability). Moreover, Theorem~\ref{thm:rank} and the fact that $\sum_{j=1}^r\sum_{i\in Q_j}i = 0$ jointly reveal that the first $r - 1$ constraints are distinct, and $\sum_{i\in Q_r}i = -\sum_{j=1}^{r-1}\sum_{i\in Q_j}i$. Therefore, we can write
\begin{align*}
&\sum_{p_1=1}^{M_T}\cdots\sum_{p_l=1}^{M_T}\frac{1}{|\sum_{i\in Q_1} i|\cdots|\sum_{i\in Q_r} i|}\mathds{1}_{\cap_{j=1}^r\{\sum_{i\in Q_j} i \neq 0\}} \\
&=\sum_{p_1=1}^{M_T}\cdots\sum_{p_l=1}^{M_T}\frac{1}{|\sum_{i\in Q_1} i|\cdots|\sum_{i\in Q_{r-1}}i||\sum_{j=1}^{r-1}\sum_{i\in Q_j}i|}\mathds{1}_{\cap_{j=1}^{r-1}\{\sum_{i\in Q_j} i \neq 0\}\cap\{\sum_{j=1}^{r-1}\sum_{i\in Q_j}i\neq 0\}} \\
&\leqslant \sum_{p_1=1}^{M_T}\cdots\sum_{p_l=1}^{M_T}\frac{1}{|\sum_{i\in Q_1} i|\cdots|\sum_{i\in Q_{r-1}}i|}\mathds{1}_{\cap_{j=1}^{r-1}\{\sum_{i\in Q_j} i \neq 0\}} \\
&\leqslant 2^{r-1}M_T^{l-r+1}\sum_{k_1=1}^{|Q_1| M_T}\cdots\sum_{k_{r-1}=1}^{|Q_{r-1}|M_T}\frac{1}{k_1\cdots k_{r-1}}
\leqslant 4^{r-1} M_T^{l-r+1}(\log M_T)^{r-1} .    
\end{align*}
where the second-to-last inequality is due to the fact that the number of variables confined by the first $r-1$ constraints is $r-1$.

Finally, to prove part 4, we first note that when $q = 1$, $\{\sum_{i\in Q_1} i \neq 0\}\cap\cap_{j=2}^{r}\{\sum_{i\in Q_j} i = 0\} = \emptyset$ due to the fact that $\sum_{i\in Q_1}i = -\sum_{j=2}^{r}\sum_{i\in Q_j}i$. When $2\leqslant q\leqslant r-1$, there will be $r - q$ equations (with $1\leqslant r-q\leqslant r-2$) forming $r-q$ distinct linear constraints due to Theorem~\ref{thm:rank}. Without loss of generality, suppose $p_1,\cdots,p_{r-q}$ are subject to the $r-q$ linear constraints. If we substitute $p_1,\cdots,p_{r-q}$ in the $q$ inequalities, we will then have $q$ inequalities associated with an indecomposable partition of the following $(l - r + q)\times 2$ table 
\begin{align*}
\begin{matrix}
&p_{r-q+1}&-p_{r-q+1} \\
&\vdots&\vdots \\
&p_l&-p_l
\end{matrix}
\end{align*}
with $q$ blocks. We will explain why it is the case. Suppose $p_1$ is restricted by $\sum_{i\in Q_1} i = 0$. Then $-p_1$ must also occur in another block (say, $Q_2$).
This indicates that $Q_1$ and $Q_2$ hook in row $\{p_1, -p_1\}$. Substituting $p_1$ in $\sum_{i\in Q_2} i$ according to $\sum_{i\in Q_1} i = 0$ is equivalent to merging $Q_1$ and $Q_2$ and deleting the row $\{p_1,-p_1\}$. We call the new block resulting from the operations $\tilde{Q}_1$. It is easy to see that $\tilde{Q}_1,Q_3,\cdots,Q_l$ form an indecomposable partition of the $(l -1)\times 2$ table $\{p_j, -p_j\}_{j=2}^l$. We can also describe the procedure in the language of multigraph. There is an edge connecting vertices $1$ and $2$ due to the fact that $Q_1$ and $Q_2$ hook. If we glue vertices $1$ and $2$ to form a new vertex $\tilde{1}$, then the edge connecting vertices $1$ and $2$ becomes a self-loop. After deleting this self-loop, the final graph corresponds to an indecomposable partition of a $(l -1)\times 2$ table.

After the subsequent substitutions, we get
\begin{align*}
&\sum_{p_1=1}^{M_T}\cdots\sum_{p_l=1}^{M_T}\frac{1}{|\sum_{i\in Q_1} i|\cdots|\sum_{i\in Q_q} i|}\mathds{1}_{\cap_{j=1}^q\{\sum_{i\in Q_j} i \neq 0\}\cap\cap_{j=q+1}^{r}\{\sum_{i\in Q_j} i = 0\}} \\
&=\sum_{p_{r-q+1}=1}^{M_T}\cdots\sum_{p_l=1}^{M_T}\frac{1}{|\sum_{i\in \tilde{Q}_1} i|\cdots|\sum_{i\in \tilde{Q}_q} i|}\mathds{1}_{\cap_{j=1}^q\{\sum_{i\in \tilde{Q}_j} i \neq 0\}}
\end{align*}
Notice that $l - r + q \geqslant 2$. Then part 3 shows that the above sum can be bounded by $4 M_T^{l - r + 1}$ when $q = 2$, or $4^{q-1} M_T^{l-r+1}(\log M_T)^{q-1}$ when $q\geqslant3$.

\end{proof}
\end{proposition}

\end{document}